\documentclass[11pt]{article}
\usepackage[margin=.8in]{geometry}
\usepackage{comment}

\usepackage{tikz,calc}
\usepackage{subcaption}
\usepackage{mwe}

\usepackage{enumerate}

\usepackage{amssymb}
\usepackage{amsmath}
\usepackage{mathrsfs}
\usepackage{amsthm}
\usepackage{mathtools}
\usepackage{float}
\usepackage{braket}
\usepackage{cancel}
\usepackage{eufrak}
\usepackage{todonotes}

\newcommand{\N}{\mathbb{N}}
\newcommand{\R}{\mathbb{R}}
\newcommand{\Q}{\mathbb{Q}}
\newcommand{\abs}[1]{\left|#1\right|}
\renewcommand{\set} [1] {\left\{#1\right\}}

\usepackage[colorlinks=true]{hyperref}
\usepackage[nameinlink]{cleveref}  
\crefname{graph}{Graph}{Graphs}

\newtheorem*{theorem*}{Theorem}
\newtheorem{theorem}{Theorem}[section]
\newtheorem{proposition}[theorem]{Proposition}
\newtheorem{corollary}[theorem]{Corollary}
\newtheorem{lemma}[theorem]{Lemma}
\newtheorem{remark}[theorem]{Remark}
\newtheorem{definition}[theorem]{Definition}
\newtheorem{observation}[theorem]{Observation}
\newtheorem{open}[theorem]{Open Problem}

\usetikzlibrary{calc}
\usetikzlibrary{arrows.meta}
\usepackage[backend=biber, style=alphabetic,isbn=false,doi=false]{biblatex}
\renewbibmacro{in:}{%
	\ifentrytype{article}{}{\printtext{\bibstring{in}\intitlepunct}}}
\DeclareLabelalphaTemplate{%
	\labelelement{%
		\field{citekey}
	}
}

\bibliography{GeodesicGeometry} 

\definecolor{red1}{RGB}{230,25,75}
\definecolor{blue1}{RGB}{0,130,200}
\definecolor{yellow1}{RGB}{255, 225, 25}
\definecolor{green1}{RGB}{60,180,75}
\definecolor{orange1}{RGB}{245,130,48}
\definecolor{purple1}{RGB}{145,30,180}
\definecolor{cyan1}{RGB}{70,240,240}
\definecolor{magenta1}{RGB}{240,50,230}
\definecolor{lime1}{RGB}{210,245,60}
\definecolor{pink1}{RGB}{250,190,190}
\definecolor{teal1}{RGB}{0,128,128}
\definecolor{lavendar1}{RGB}{230,190,255}
\definecolor{brown1}{RGB}{170,10,40}
\definecolor{beige1}{RGB}{255,250,200}
\definecolor{maroon1}{RGB}{128,0,0}
\definecolor{mint1}{RGB}{170,255,195}
\definecolor{olive1}{RGB}{128,128,0}
\definecolor{apricot1}{RGB}{255,215,180}
\definecolor{navy1}{RGB}{0,0,128}
\definecolor{grey1}{RGB}{128,128,128}
\definecolor{darkgrey1}{RGB}{57, 59, 58}

\newsavebox\ideabox
\newenvironment{idea}[1][]
{\begin{equation*}
	\tag{#1}
	\begin{lrbox}{\ideabox}
	\begin{minipage}{\dimexpr\columnwidth-2\leftmargini}
	\vspace{3mm}
	\setlength{\leftmargini}{0pt}%
	\begin{quote}}
{\end{quote}
	\end{minipage}
	\end{lrbox}\makebox[0pt]{\usebox{\ideabox}}
	\end{equation*}}

\title{Geodesic Geometry on Graphs}
\author{Daniel Cizma\thanks{Department of Mathematics, Hebrew University, Jerusalem 91904,
		Israel. e-mail: daniel.cizma@mail.huji.ac.il.} \and  {Nati Linial\thanks{School of Computer Science and Engineering, Hebrew University, Jerusalem 91904,
			Israel. e-mail: nati@cs.huji.ac.il~. {Supported by BSF US-Israel Grant 2018313 "Between Topology and Combinatorics"}}}}

\begin{document}
	\maketitle
	
\begin{abstract}
We investigate a graph theoretic analog of geodesic geometry. In a graph $G=(V,E)$ we consider a system of paths $\mathcal{P}=\{P_{u,v}|u,v\in V\}$ where $P_{u,v}$ connects vertices $u$ and $v$. This system is {\em consistent} in that if vertices $y, z$ are in $P_{u,v}$, then the sub-path of $P_{u,v}$ between them coincides with $P_{y,z}$. A map $w: E\to(0,\infty)$ is said to {\em induce} $\mathcal{P}$ if for every $u, v\in V$ the path $P_{u,v}$ is $w$-geodesic. We say that $G$ is {\em metrizable} if every consistent path system is induced by some such $w$. As we show, metrizable graphs are very rare, whereas there exist infinitely many $2$-connected metrizable graphs.
\end{abstract}	
	
\section{Introduction}
The idea of viewing graphs from a geometric perspective has been immensely fruitful. We refer the reader to Lov\'asz' recent book \cite{Lo} for a beautiful exposition of many of these success stories. Most of the existing connections between graph theory and differential geometry concern the eigenvalues of graphs. Here we study graphs from the perspective of {\em geodesic geometry}. Our main discovery is that for the vast majority of graphs the geodesic theory is way richer than the metric one. The relevant literature seems rather sparse. Ore \cite{Ore} defines a graph to be {\em geodetic} if between any two vertices in it there is a unique shortest path. He sought to characterize geodetic graphs, and notwithstanding a considerable body of work on this problem (e.g., \cite{BKZ, Br, BlBr, PS} ) no definitive solution is in sight. Bodwin's beautiful recent work \Cite{Bo} is somewhat relevant, but as we explain below, his problems and results differ from ours.

In this work we introduce and study the notion of graph metrizability. A graph $G$ is said to be {\em metrizable} if every consistent path system $\mathcal{P}$ in $G$ is induced by a graph metric $d$, in that every path in $\mathcal{P}$ is $d$-geodetic. Graph metrics are defined in terms of positive edge weights. Here a path system is a collection of paths in $G$ which contains a unique path $P_{u,v}$ between every pair of vertices $u,v$. Consistency means that if vertices $y, z$ are in $P_{u,v}$, then the sub-path of $P_{u,v}$ between them coincides with $P_{y,z}$. Strict metrizability means that $d$ makes every $P_{u,v}$ the unique shortest $uv$ path.
	
Here are our main findings:
\begin{itemize}
\item 
Metrizability is rare: E.g., (i) Every large $2$-connected metrizable graph is planar, \Cref{thm:NonPlamarMet}, (ii) No large $3$-connected graph is metrizable, \Cref{cor:3ConnMet}.
\item
However, arbitrarily large $2$-connected metrizable graphs do exist: E.g., every outerplanar graph is metrizable, \Cref{cor:outerPlanarMet}.
\item
We reveal some of the structural underpinnings of metrizability. The class of metrizable graphs is closed under the topological minor relation and is characterized by finitely many forbidden topological minors, \Cref{thm:finiteMinGraphs}.
\item
		On the computational side, metrizability can be decided in polynomial time, \Cref{thm:graphMetPoly}.
	\end{itemize}
	
Our main focus is on metrizability as a property of graphs. In contrast, Bodwin \Cite{Bo} investigates metrizability as a property of path systems. His main question is which {\em partial} path systems are strictly\footnote{In his terminology {\em strongly metrizable}} induced by a metric. A partial path system is a collection of paths such that if the vertices $u$ and $v$ are in two of these paths, then their $uv$ subpaths must coincide. Bodwin has found an infinite family of intersection patterns such that a partial path system is strictly metrizable if and only if no such pattern occurs within the system. The difference between his work and ours goes deeper, since, as we show in \Cref{sec:basic}, not every partial path system can be extended to a full path system. Moreover, it is possible for a graph to be strictly metrizable and yet contain a partial path system which is not strictly metrizable.
	
\subsubsection*{The role of computers in this work:} Everything in here can be verified by hand, although this paper would not exist without our use of the computer. Although we are initially able to prove by hand that Petersen's graph \Cref{fig:nonMetPetersen} is non-metrizable, it quickly transpired that we need a larger supply of such graphs. To this end we wrote a brute-force search program that found eleven such graphs (\Cref{fig:MinimumGraphs}) and gave certificates that they are indeed non-metrizable. These certificates, see \Cref{append:certificates}, are easily verifiable by hand.
	
	\section{Definitions}
	Unless explicitly stated otherwise, paths that are mentioned throughout are {\em simple}.

	Let $G=(V,E)$ be a connected graph.
	\begin{itemize}
		\item 
		A {\em path system} $\mathcal P$ in $G$ is a collection of simple paths in $G$ such that for every $u,v\in V$  there is exactly one member $P_{u,v}\in\mathcal P$ that connects between $u$ and $v$.
		
		\item
		A {\em tree system} $\mathcal T$ in $G$ is a collection of spanning trees in $G$ such that for every $u\in V$  there is exactly one member $T_{u}\in\mathcal T$ which we think of as {\em rooted} at $u$.
	\end{itemize}

	\begin{itemize}
		\item
		Let $\mathcal P$ be a path system in $G$. We say that it is {\em consistent} if for every $P\in\mathcal{P}$ and two vertices $x,y$ in $P$, the $xy$ subpath of $P$ coincides with $P_{x,y}$.
		\item
		Let $\mathcal T$ be a tree system in $G$. We say that it is {\em consistent} if 
		for every two vertices $u,v\in V$ the $uv$ paths in $T_u$ and in $T_v$ are identical.
	\end{itemize}
	
	As we observe next, consistent path systems $\mathcal{P}$ and consistent tree systems  $\mathcal{T}$ in the same graph $G$ are in a simple one-to-one correspondence:\\ Given $\mathcal{P}$, we define the tree $T_u$ for every $u\in V$ via $E(T_u):=\cup_v E(P_{u,v})$. This is clearly a spanning subgraph of $G$ and it is acyclic due to the consistency of $\mathcal{P}$.\\
	Given $\mathcal{T}$, we let $P_{u,v}$ be the $vu$ path in $T_u$, or, what is the same, the $uv$ path in $T_v$. This yields a consistent path system, because the path between any two vertices in a tree is unique.\\
	Therefore we can and will interchangeably talk of consistent path systems and consistent tree systems. Unless otherwise stated, all path systems and tree systems mentioned henceforth are assumed to be consistent.\\
	Given a weight function $w: E(G) \to \R$, we assign to every subgraph $H$ of $G$ the weight $w(H) = \sum_{e\in E(H)} w(e)$. 
	
	\begin{definition}
		A path system $\mathcal{P}$ in $G=(V,E)$ is {\em induced} by $w:E(G) \to (0,\infty)$ if for each $u,v \in V$, $P_{u,v}$ is a $uv$ {\em geodesic}, i.e., $w(P_{u,v}) \leq w(Q)$ for every $uv$ path $Q$. A path system that is induced by a positive weight function is said to be {\em metrizable}\\
		A map $w:E(G) \to (0,\infty)$ {\em strictly induces} a path system $\mathcal{P}$ if for each $u,v \in V$, $P_{u,v}$ is the {\em unique} $uv$ geodesic, i.e., $w(P_{u,v}) < w(Q)$ for every $uv$ path $Q\neq P_{u,v}$. A path system that is strictly induced by some positive weigh function is said to be {\em strictly metrizable}.\\
		A graph $G$ is {\em (strictly) metrizable} if every path system in $G$ is (strictly) metrizable.
	\end{definition}
	
\begin{remark}\label{rem:2conn}
It clearly suffices to consider connected graphs. In a disconnected graph, we deal with each connected component separately. In fact, it suffices to consider only to $2$-connected graphs. If vertex $a$ separates a connected graph $G$, then $G\setminus\{a\}$ is the disjoint union of two (not necessarily connected) graphs $H_1$ and $H_2$, where $G_1:=H_1\cup \{a\}, G_2:=H_2\cup \{a\}$ are connected graphs, and any path system in $G$ is uniquely defined by its restrictions to $G_1, G_2$. Indeed, if $u$ and $v$ belong to the same $G_i$, then any simple $uv$ path is contained in $G_i$, while if $u\in G_1$ and $v\in G_2$, then any $uv$ path is the concatenation of a $ua$ path and an $av$ path.
\end{remark}

\section{Some Examples}
	
Not all path systems are metrizable.  \Cref{fig:nonMetPetersen} exhibits a non-metrizable path system in the Petersen graph $\Pi$.
\begin{figure}[h]
\centering
		\begin{subfigure}[t]{0.4\textwidth}
			\centering
			\begin{tikzpicture}[scale=0.6, every node/.style={scale=0.6}]

			\node[draw,circle,minimum size=.5cm,inner sep=1pt] (1) at (-0*360/5 +90: 5cm) [scale=1.6] {$1$};
			\node[draw,circle,minimum size=.5cm,inner sep=1pt] (2) at (-1*360/5 +90: 5cm) [scale=1.6]{$2$};
			\node[draw,circle,minimum size=.5cm,inner sep=1pt] (3) at (-2*360/5 +90: 5cm)[scale=1.6] {$3$};
			\node[draw,circle,minimum size=.5cm,inner sep=1pt] (4) at (-3*360/5 +90: 5cm) [scale=1.6]{$4$};
			\node[draw,circle,minimum size=.5cm,inner sep=1pt] (5) at (-4*360/5 +90: 5cm)[scale=1.6] {$5$};
			
			\node[draw,circle,minimum size=.5cm,inner sep=1pt] (6) at (-5*360/5 +90: 2.5cm) [scale=1.6]{$6$};
			\node[draw,circle,minimum size=.5cm,inner sep=1pt] (7) at (-6*360/5 +90: 2.5cm) [scale=1.6]{$7$};
			\node[draw,circle,minimum size=.5cm,inner sep=1pt] (8) at (-7*360/5 +90: 2.5cm) [scale=1.6]{$8$};
			\node[draw,circle,minimum size=.5cm,inner sep=1pt] (9) at (-8*360/5 +90: 2.5cm) [scale=1.6]{$9$};
			\node[draw,circle,minimum size=.5cm,inner sep=1pt] (10) at (-9*360/5 +90: 2.5cm)[scale=1.6] {$10$};
			
			\coordinate (c1) at ($(7)!0.5!(10)$);
			\coordinate (c2) at ($(6)!0.5!(8)$);
			\coordinate (c3) at ($(7)!0.5!(9)$);
			\coordinate (c4) at ($(8)!0.5!(10)$);
			\coordinate (c5) at ($(6)!0.5!(9)$);
			
			\draw [line width=3.1pt,-,red1] (1) -- (2);
			\draw [line width=3.1pt,-,blue1] (1) -- (5);
			\draw [line width=3.1pt,-,red1] (1) -- (6);
			\draw [line width=3.1pt,-,green1] (2) -- (3);
			\draw [line width=3.1pt,-,green1] (2) -- (7);
			\draw [line width=3.1pt,-,darkgrey1] (3) -- (4);
			\draw [line width=3.1pt,-,darkgrey1] (3) -- (8);
			\draw [line width=3.1pt,-,yellow1] (4) -- (5);
			\draw [line width=3.1pt,-,yellow1] (4) -- (9);
			\draw [line width=3.1pt,-,blue1] (5) -- (10);
			\draw [line width=3.1pt,-,red1] (6) -- (8);
			\draw [line width=3.1pt,-,yellow1] (6) -- (9);
			\draw [line width=3.1pt,-,green1] (7) -- (9);
			\draw [line width=3.1pt,-,blue1] (7) -- (10);
			\draw [line width=3.1pt,-,darkgrey1] (8) -- (10);
			
			\draw [line width=3.1pt,-,red1] (c2) -- (6);
			\draw [line width=3.1pt,-,yellow1] (c5) -- (9);

			\end{tikzpicture}
			\caption{Non-metrizable path system in the Petersen Graph}
			\label{fig:nonMetPetersen}
		\end{subfigure}
		\hspace{20mm}
		\begin{subfigure}[t]{0.4\textwidth}
			\centering
			\begin{tikzpicture}[scale=0.75, every node/.style={scale=0.75}]

			\node[draw,circle,minimum size=.7cm,inner sep=0pt] (x1) at (0*360/3 +90: 1cm) {$x_1$};
			\node[draw,circle,minimum size=.7cm,inner sep=0pt] (x2) at (1*360/3 +90: 1cm) {$x_2$};
			\node[draw,circle,minimum size=.7cm,inner sep=0pt](x3) at (2*360/3 +90: 1cm) {$x_3$};
			\node[draw,circle,minimum size=.7cm,inner sep=0pt] (y1) at (0*360/3 +90: 4cm) {$y_1$};
			\node[draw,circle,minimum size=.7cm,inner sep=0pt] (y2) at (1*360/3 +90: 4cm) {$y_2$};
			\node[draw,circle,minimum size=.7cm,inner sep=0pt](y3) at (2*360/3 +90: 4cm) {$y_3$};

			\draw [line width=2.5pt,-] (x1) -- (x2);
			\draw [line width=2.5pt,-] (x1) -- (x3);
			\draw [line width=2.5pt,-] (x2) -- (x3);
			\draw [line width=2.5pt,-] (y1) -- (y2);
			\draw [line width=2.5pt,-] (y1) -- (y3);
			\draw [line width=2.5pt,-] (y2) -- (y3);
			\draw [line width=2.5pt,-] (x1) -- (y1);
			\draw [line width=2.5pt,-] (x2) -- (y2);
			\draw [line width=2.5pt,-] (x3) -- (y3);

			\end{tikzpicture}
			\caption{A graph which has path system which is a metrizable but not strictly metrizable}
			\label{fig:notStrictlyMetExample}
		\end{subfigure}
		\caption{}
	\end{figure}
	If $uv\in E(\Pi)$, then the path $P_{uv}$ is comprised of the single edge $uv$. Between any two nonadjacent vertices $x,y \in V(\Pi)$ there is a unique path of length $2$. For most such pairs this is taken to be $P_{x,y}$. There are $5$ exceptional pairs of nonadjacent vertices, those which are connected by a colored path in \Cref{fig:nonMetPetersen}. For example, $P_{1,7}=1,5,10,7$. It is easily verified that this path system is consistent, and as we show next, this path system is nonmetrizable. If $w$ is a weight function that induces it, then by considering the colored paths, the following inequalities must hold:
	\begin{equation*}
	\begin{split}
	w_{1,2} + w_{1,6} + w_{6,8} & \leq w_{2,3} + w_{3,8}\\
	w_{2,3} + w_{2,7} + w_{7,9} & \leq w_{3,4} + w_{4,9}\\
	w_{3,4} + w_{3,8} + w_{8,10} & \leq w_{4,5} + w_{5,10}\\
	w_{4,5} + w_{4,9} + w_{6,9} & \leq w_{1,5} + w_{1,6}\\
	w_{1,5} + w_{5,10} + w_{7,10} & \leq w_{1,2} + w_{2,7}\\
	\end{split}
	\end{equation*}
	which implies
	$$w_{6,8} +w_{7,9} +w_{8,10} + w_{6,9} +w_{7,10} \leq 0,$$
	showing a weight function inducing these paths cannot be strictly positive. \\
	\Cref{fig:notStrictlyMetExample} shows a metrizable path system which is not strictly metrizable. Namely, every edge is the chosen path between its two vertices. For $i=1,2,3$, let $P_{x_i, y_{i+1 }} = x_i y_i y_{i+1}$ and  $P_{y_i, x_{i+1 }} = y_i x_i x_{i+1},$ with incides taken $\bmod~ 3$.
	It is easy to see that the constant weight function induces this path system. If a weight function $w$ strictly induces this system, then for $i=1,2,3$ the following inequalities must hold:
	$$w(x_i y_i) + w(y_{i} y_{i+1}) <w(x_i x_{i+1}) + w(x_{i+1i} y_{i+1}) $$ and 
	$$w(y_i x_i) + w(x_{i} x_{i+1}) <w(y_i y_{i+1}) + w(y_{i+1i} x_{i+1})$$
	Summing the first inequality for $i=1,2,3$ and canceling identical terms yields
	$$ \sum_{i=1}^{3}w(y_{i} y_{i+1}) <\sum_{i=1}^{3}w(x_i x_{i+1}).$$
	Similarly, adding up the second inequality gives
	$$ \sum_{i=1}^{3}w(x_{i} x_{i+1}) <\sum_{i=1}^{3}w(y_i y_{i+1}),$$
	a contradiction.\\
	
	\section{Basic Properties of Path and Tree Systems}\label{sec:basic}
	Before getting into graph metrizability we establish some basic properties of path and tree systems. The notion of consistency makes sense also for a partial system of paths and we ask when a consistent system of paths can be extended to a full consistent path system. We give an example of a consistent partial collection of paths for which this is impossible. We then go on to prove some conditions under which the answer is positive. Next we establish certain lemmas which help us better elucidate the structure of path systems. Aside of the inherent interest in these lemmas, they help us to fully describe path systems in cycles. Path systems in cycles play a key role in the study of metrizability (\Cref{thm:SuspendedPath}). 
	
	We start with the following easy observation.

	\begin{proposition}\label{prop:weightFuncInducesPS}
		Let $G=(V,E)$ be a connected graph. Then every weight function $w: E\to (0,\infty)$ induces a consistent path system on $G$.
	\end{proposition}
	\begin{proof}
		
		If the $w$-shortest $u,v$-path is unique, then there is nothing to prove. What we need is a rule to break ties between $u,v$-paths of the same $w$ length. To this end fix some ordering $e_1, e_,\ldots$ on the edges of $G$ and break ties between two such paths by lexicographic comparison. It is easy to see that this guarantees consistency. 
	\end{proof}
	
	\begin{figure}[h]
		\centering
		\begin{subfigure}[t]{0.45\textwidth}
			\centering
			\begin{tikzpicture}[scale=0.85, every node/.style={scale=0.85}]

			\node[draw,circle,minimum size=.6cm,inner sep=0pt] (u) at (-4,0) {$u$};
			\node[draw,circle,minimum size=.4cm,fill] (x1) at (-2,2) {};
			\node[draw,circle,minimum size=.4cm,fill](x2) at (-2,-2) {};
			\node[draw,circle,minimum size=.4cm,fill] (y1) at (2,2) {};
			\node[draw,circle,minimum size=.4cm,fill] (y2) at (2,-2) {};
			\node[draw,circle,minimum size=.6cm,inner sep=0pt] (v) at (4,0) {$v$};

			\draw [line width=3pt,-,blue1] (u) -- (x1);
			\draw [line width=3pt,-,orange1] (u) -- (x2);
			\draw [line width=3pt,-,red1] (v) -- (y1);
			\draw [line width=3pt,-,green1] (v) -- (y2);
			\draw [line width=3pt,-,red1] (x1) -- (y1);
			\draw [line width=3pt,-,green1] (x2) -- (y2);
			\draw [line width=3pt,-,blue1] (x1) -- (y2);
			\draw [line width=3pt,-,orange1] (x2) -- (y1);

			\end{tikzpicture}
			\caption{A partial path system which cannot be extended to a full system}
			\label{fig:nonExtendableEx}
		\end{subfigure}
		\hfill
		\begin{subfigure}[t]{0.45\textwidth}
			\centering
			\begin{tikzpicture}[scale=.6, every node/.style={scale=.6}]

			\node[draw,circle, minimum size=.25cm,inner sep=0pt,fill, darkgrey1] (0) at (0*360/3 - 90: 0.5cm) {};
			\node[draw,circle, minimum size=.25cm,inner sep=0pt,fill, darkgrey1](1) at (1*360/3 - 90: 0.5cm) {};
			\node[draw,circle, minimum size=.25cm,inner sep=0pt,fill, darkgrey1] (2) at (2*360/3 - 90: 0.5cm) {};
			
			\node[draw,circle, minimum size=.25cm,inner sep=0pt,fill, darkgrey1](6) at (0*360/6 + 90: 1.5cm) {};
			\node[draw,circle, minimum size=.25cm,inner sep=0pt,fill, darkgrey1](7) at (1*360/6 + 90: 1.5cm) {};
			\node[draw,circle, minimum size=.25cm,inner sep=0pt,fill, darkgrey1] (8) at (2*360/6 + 90: 1.5cm) {};
			\node[draw,circle, minimum size=.25cm,inner sep=0pt,fill, darkgrey1] (3) at (3*360/6 + 90: 1.5cm) {};
			\node[draw,circle, minimum size=.25cm,inner sep=0pt,fill, darkgrey1] (4) at (4*360/6 + 90: 1.5cm) {};
			\node[draw,circle, minimum size=.25cm,inner sep=0pt,fill, darkgrey1](5) at (5*360/6 + 90: 1.5cm) {};
			\node[draw,circle, minimum size=.25cm,inner sep=0pt,fill, darkgrey1] (10) at (0*360/3 + 90: 5cm) {};
			\node[draw,circle, minimum size=.25cm,inner sep=0pt,fill, darkgrey1] (11) at (1*360/3 + 90: 5cm) {};
			\node[draw,circle, minimum size=.25cm,inner sep=0pt,fill, darkgrey1] (9) at (2*360/3 + 90: 5cm) {};
			
			\draw [line width=3pt, -,red1] (0) -- (1);
			\draw [line width=3pt, -,yellow1] (0) -- (2);
			\draw [line width=3pt, -,red1] (0) -- (3);
			\draw [line width=3pt, -,blue1] (1) -- (2);
			\draw [line width=3pt, -,blue1] (1) -- (5);
			\draw [line width=3pt, -,yellow1] (2) -- (7);
			\draw [line width=3pt, -,red1] (3) -- (4);
			\draw [line width=3pt, -,orange1] (3) -- (8);
			\draw [line width=3pt, -,purple1] (4) -- (5);
			\draw [line width=3pt, -,purple1] (4) -- (9);
			\draw [line width=3pt, -,blue1] (5) -- (6);
			\draw [line width=3pt, -,green1] (6) -- (10);
			\draw [line width=3pt, -,green1] (6) -- (7);
			\draw [line width=3pt, -,yellow1] (7) -- (8);
			\draw [line width=3pt, -,orange1] (8) -- (11);
			\draw [line width=3pt, -,purple1] (9) -- (10);
			\draw [line width=3pt, -,orange1] (9) -- (11);
			\draw [line width=3pt, -,green1] (10) -- (11);

			\end{tikzpicture}
			\caption{Another partial path system which cannot be extended to a full system}
			\label{fig:dodecahedronN3}
		\end{subfigure}
		\caption{}
	\end{figure}
	
	A {\em partial} path system $\Pi$ in a connected graph $G=(V,E)$ is a collection of paths between {\em some} pairs $u,v\in V$. We say $\Pi$ is a consistent partial path system, if for all paths $P, Q\in \Pi$ and vertices $u,v\in V(P)\cap V(Q)$, the $uv$ subpaths of $P$ and $Q$ coincide. As \Cref{fig:nonExtendableEx} shows, not every consistent partial path system can be extended to a full consistent path system. Here, the partial path system consists of the colored paths. It easy to see that the addition of any $uv$ path to this system makes it inconsistent. 
	Another example of a non-extendable partial path system is given by \Cref{fig:dodecahedronN3}. Here again the colored paths form our partial path system. It is not difficult to see that any path which connects the two triangles in this graph is inconsistent with this partial system.
	
	A partial path system $\Pi$ is said to be {\em strictly metrizable} if there exists a weight function $w$ such every path in $\Pi$ is the unique geodesic w.r.t.\ $w$ between its end vertices. We observe:
	\begin{observation}\label{obs:nonExtendnonStrict}
		A consistent partial path system $\Pi$ that cannot be extended to a full consistent path system is not strictly metrizable.
	\end{observation}
	\begin{proof}
		Suppose that $\Pi$ is strictly induced by some weight function $w$. By \Cref{prop:weightFuncInducesPS} $w$ induces some full path system $\mathcal{P}$. Moreover, necessarily $\Pi \subseteq \mathcal{P}$ since every path in $\Pi$ is assumed to be a unique shortest path w.r.t.\ to $w$. It follows that $\Pi$ can be extended to a full system.
	\end{proof}
	
	We take this opportunity to explain how our work differs from Bodwin's. Although the basic concepts may seem similar, the difference is substantial. Thus, a path system in the language of \Cite{Bo}, is what we call a partial path system. By \Cref{obs:nonExtendnonStrict} the partial path system in \Cref{fig:nonExtendableEx} is not strictly metrizable.
	Therefore, from the point of view \Cite{Bo} (e.g., his Figure 6) \Cref{fig:nonExtendableEx} shows a path system which is not strictly metrizable. From our perspective, since we only consider full path systems, the graph in \Cref{fig:nonExtendableEx} is in fact an example of a strictly metrizable graph. (We use a computer to prove this claim. For more on this, see \Cref{subsec:praktisch}.) Bodwin lays out an infinite family of intersection patterns and proves that a consistent partial path system is strictly mertizable if and only if it does not contain any of these patterns. The above example illustrates that a graph may admit one of these intersection patterns, and still be strictly metrizable according to our definitions. 
	
	We discuss next some cases where an extension of a partial path system is possible. A path-system is called {\em neighborly} if the path between any two adjacent vertices is the edge between them.
	\begin{proposition}\label{prop:neighborly}
		Let $\mathcal{P}'$ be a consistent neighborly path system in $G' = (V',E')$, an induced subgraph of $G = (V,E)$. Then $\mathcal{P}'$ can be extended to a consistent neighborly path system in $G$.
	\end{proposition}
	\begin{proof}
		We may and will assume that $G$ is connected. Also, by induction it suffices to consider the case where $G' = G \setminus \set{v}$ for some $v\in V$. Of course, we only need to specify the paths $P_{x,v}$ for all $x\neq v$.\\
		Let us start with the case where $G'$ is connected. We set $P_{v,u} \coloneqq vu$ for every neighbor $u$ of $v$. Now fix some $u_0\in N_G(v)$ and set the paths $P_{x,v}$ for $x\not\in N_G(v)$ according to the following intuitive rule: Seek a way from $x$ to $v$ via $u_0$, but if this path visits another neighbor of $v$, then hop to $v$ as early as possible. More formally, $P_{x,v} \coloneqq P'_{x,u_x}v$ where $u_x$ is the first neighbor of $v$ on the path $P'_{x,u_0}$. Since $\mathcal{P}'$ is consistent, it suffices to show that $P_{y,v}$ is a subpath of $P_{x,v}$ whenever $y\in P_{x,v}$. By construction $P_{x,v} = P'_{x,u_x}v$ and therefore $y\in P'_{x,u_x}$. But $P'_{y,u_x}$ is a subpath of $P'_{y,u_0}$, since $\mathcal{P}'$ is consistent. It follows that $u_x = u_y$ is the first neighbor of $v$ along $P'_{y,u_0}$. Therefore $P_{y,v} = P_{y,u_x}v$ is a subpath of $P_{x,v}$, as claimed. \\
		If $G'$ is disconnected, then $v$ is a cut-vertex. We apply the above procedure to every component $C=(V_C, E_C)$ of $G'$ with $V = V_C \cup \set{v}$ and $V' = V_C$, yielding a consistent system of paths $\mathcal{P}_C$ for the graph $G[V_C\cup \set{v}]$. For vertices $u$ and $w$ are in two distinct components $C_1$ and $C_2$, we let $P_{u,w} = P_{u,v} P_{v,w}$ with $P_{u,v} \in \mathcal{P}_{C_1},  P_{v,w} \in \mathcal{P}_{C_2}$ to determine a consistent system in $G$.
	\end{proof}
	
	We turn our attention to tree systems and draw significant structural consequences from the assumption that two trees in the system coincide.
	
	\begin{lemma}
		\label{lem:treePathSys1}
		Let $\mathcal{T}$ be a tree system in a graph $G=(V,E)$, and $\mathcal{P}$ its corresponding path system. For $u,v \in V$, the following are equivalent:
		\begin{enumerate}
			\item $T_u = T_v$
			\item Every tree $T_w\in \mathcal{T}$ contains the path $P_{u,v}$.
			\item $T_z = T_u$ for every $z\in P_{u,v}$. 
		\end{enumerate}
	\end{lemma}
	\begin{proof}
		We recall the following standard fact: Every three (not necessarily distinct) vertices $a, b, c$ in a tree $T$ have a {\em median}. This is a vertex $\mu$ which is uniquely characterized by the property that every vertex other than $\mu$ is in at most one of the paths $\mu$-$a$, $\mu$-$b$, $\mu$-$c$ in $T$.\\
		Recall that for any two vertices $\gamma, \delta\in V$ the $\gamma\delta$ path in $T_{\gamma}$ is $P_{\gamma,\delta}$.\\
		$(3) \implies (1)$. This is obvious.\\
		$(1)\implies (2)$. $P_{u,v}$ is clearly contained in $T=T_u=T_v$. We wish to show that it is contained in $T_w$ for some $w\neq u,v$. Let $x$ be the median in $T$ of $u, v, w$. Note that $P_{u,w} =P_{u,x}P_{x,w}$ and $P_{v,w} =P_{v,x}P_{x,w}$ are both paths in $T_w$.
		Together with the fact that $T_{u} = T_{v}$ this implies that the paths $P_{u,x}, \ P_{v,x}$ are contained in $T_u, T_v$ and $T_w$. Therefore, $P_{u,v}$, the $uv$ path in $T_u = T_v$, is precisely the path $P_{u,x}P_{x,v}$. It follows $P_{u,v} = P_{u,x}P_{x,v}$ is contained in $T_w$ as well.\\
		$(2)\implies (3)$. When are two trees on the same vertex set $V$ (in our case $T_z$ and $T_u$) identical? Fix some $\alpha\in V$. Then two trees on $V$ are identical iff the $\alpha \beta$ paths in both trees coincide for every $\beta\in V$. We use this criterion with $\alpha=z$ and show that for every $w\in V$ the $zw$ paths in $T_z$ and $T_u$ are identical. When $w=u$ this is clearly true as the $zu$ path in both trees is $uz$ subpath of $ P_{u,v}$. The situation is similar when $w=v$. For other vertices $w$, by assumption, the path $P_{u,v}$ is contained in $T_w$. Therefore, the $wu$ path in $T_w$ is $P_{w,u} = P_{w,x}P_{x,u}$, where $x$ is the median in $T_w$ of the vertices $u, v, w$, see \Cref{fig:treeSysLemma}. Likewise the $wv$ path in $T_w$ is $P_{w,v} = P_{w,x}P_{x,v}$. 
		Since $x\in P_{u,v}$ it follows that $P_{u,v}=P_{u,x}P_{x,v}$. Also $z\in P_{u,v}$ so $z$ must belong to either $P_{u,x}$ or $P_{x,v}$ (or both, when $z=x$). In either case it we have that the $wz$ path in  $T_z$ is $P_{w,z} = P_{w,x}P_{x,z}$. Notice that $P_{w,x}$, a subpath of $P_{w,u}$, and $P_{x,z}$, a subpath of $P_{u,v}$, are both contained in $T_u$. It follows that $P_{w,z} = P_{w,x}P_{x,z}$ is also the $wz$ path in $T_u$.
	\end{proof}
	\begin{figure}[h]
		\centering
		\begin{subfigure}[t]{.45\textwidth}
			\centering
			\begin{tikzpicture}[ scale=0.85, every node/.style={scale=0.85}]
			
			\node[draw,circle, minimum size=.5cm,inner sep=0pt,inner sep=0pt] (v) at (0*360/3 + 150: 3cm) {$v$};
			\node[draw,circle, minimum size=.5cm,inner sep=0pt,inner sep=0pt](u) at (1*360/3 +150: 3cm) {$u$};
			\node[draw,circle, minimum size=.5cm,inner sep=0pt,inner sep=0pt] (w) at (2*360/3 +150: 3cm) {$w$};
			\node[draw,circle, minimum size=.5cm,inner sep=0pt,inner sep=0pt] (x) at (0,0) {$x$};
			\node[draw,circle, minimum size=.5cm,inner sep=0pt,inner sep=0pt] (z) at ($(u)!0.6!(x)$){$z$};

			\draw [line width=2pt,-] (u) -- (z) node [midway, left] {$P_{u,z}$};
			\draw [line width=2pt,-] (v) -- (x) node [midway, below left= -.03cm and -.07cm] {$P_{v,x}$};
			\draw [line width=2pt,-] (w) -- (x) node [midway, above left = -.1 and -.05] {$P_{x,w}$};
			\draw [line width=2pt,-] (z) -- (x) node [midway, right] {$P_{z,x}$};

			\end{tikzpicture}
			\caption{The paths $P_{v,x}$, $P_{x,w}$ and $P_{u,x}$ intersect only at $x$, the median of $v$, $u$ and $w$ in $T_{w}$.}
			\label{fig:treeSysLemma}
		\end{subfigure}
		\begin{subfigure}[t]{.45\textwidth}
			\centering
			\begin{tikzpicture} [scale=0.85, every node/.style={scale=0.85}]
			\def \n {7}
			\def \radius {2.4}
			\def \radiuss {2.85}
			
			\draw [-{Latex[length=5mm]},line width=2pt] (360/\n*1-90 -360/\n:\radius) -- (360/\n*1+90 -360/\n:\radius) {};
			\draw [-{Latex[length=5mm]},line width=2pt] (360/\n*3-90 -360/\n:\radius) -- (360/\n*3+90 -360/\n:\radius) {};
			
			\draw[line width =1.5pt] circle(\radius)
			
			foreach\s in{1,...,\n}{
				node[draw,circle,minimum size=.5cm,inner sep=0pt, fill = white] at 	(360/\n*\s-90 -360/\n:\radius) {$\s$}
			}
			
			foreach \s in{1,...,\n}{
				node at (360/\n*\s+90 -360/\n:\radiuss){$f(\s)$}
			}	;

			\end{tikzpicture}
			\caption{A crossing function $f$ in the $7$-Cycle.}
			
			\label{fig:CrossingFunction}
		\end{subfigure}
	\end{figure}
	We denote by $G/e$ the graph obtained from $G$ by contracting the edge $e$. Let $\mathcal{P}$ be a consistent path system in $G$. If $e=uv$ and $u,v$ are as in \Cref{lem:treePathSys1}, then $\mathcal{P}/e$ is a consistent path system in $G/e$. 
	
	An edge $e$ is said to be $\mathcal{T}$-{\em persistent} if it belongs to every tree in the tree system $\mathcal{T}$. It is $\mathcal{P}$-persistent for a path system $\mathcal{P}$ if for every vertex $u$ there is a vertex $v$ with $e\in P_{uv}$. It is easily verified that for a corresponding pair $\mathcal{T}$ and $\mathcal{P}$, as described above, the two conditions are equivalent. 
	
	\begin{proposition} \label{prop:TreeQuotient}
		Let $G=(V,E)$ be a connected graph, $\mathcal{T}$ be a consistent tree system in $G$ and $e$ a $\mathcal{T}$-persistent edge. Then $\mathcal{T}/ e \coloneqq \set{T/e : T\in \mathcal{T}}$ is a consistent tree system in $G/e$.
	\end{proposition}
	
	\begin{proof}
		Trees are closed under edge contraction, and so $\mathcal{T}/ e$ is a collection of trees. Write $e=uv$ and let $z$ be the vertex obtained from contracting $uv$. By \Cref{lem:treePathSys1}, $T_u = T_v$ so that $T_u/e = T_v /e$.
		For $w\in V\setminus \set{u,v}$ the tree rooted at $w$ is then $T_w/e$ and the tree rooted at $z$ is $T_u/e = T_v/e$. The consistency of such a system follows from the consistency of $\mathcal{T}$.
	\end{proof}
	
	\begin{lemma} \label{prop:StrictMetQuot}
		Let $\mathcal{P}$ be a path system in a connected graph $G=(V,E)$, and let $e\in E$ be a $\mathcal{P}$-persistent edge. If $\mathcal{P}/ e$ is strictly metrizable then so is $\mathcal{P}$.
	\end{lemma}
	\begin{proof}
		Let $\tilde{w} : E/ e \to (0,\infty)$ be a weight function that strictly induces $\mathcal{P}/e$. Fix some small $\delta >0$ and large $N>0$. Define $w: E \to (0,\infty)$ by 
		$$w(e') \coloneqq \begin{cases}
		\tilde{w}(e'/e) & e'\neq e \text{~participates in some paths of~} \mathcal{P}\\
		N & e'\neq e  \text{~does not participate in any path of~}   \mathcal{P}\\
		\delta & e' =e
		\end{cases}.$$
		
		We argue that $w$ strictly induces $\mathcal{P}$. Let $P_{u,v} \in \mathcal{P}$ and $Q$ some other $uv$ path. We first consider the case where
		$P_{u,v}/e \neq Q/e$. By construction of $\mathcal{P}/ e$, $P_{u,v}/ e \in \mathcal{P}/ e$, so that $\tilde{w}(Q/e) - \tilde{w}(P/e) > 0 $. We take $\delta >0$ to be small small enough so that $w(Q) - w(P) \geq \tilde{w}(Q/e) - (\tilde{w}(P/e) +\delta) > 0$.
		
		Now suppose $P_{u,v}/e = Q/e$. We argue that $Q$ contains an edge not contained in $\mathcal{P}$. Indeed, first note that since $P_{u,v}$ and $Q$ are distinct $uv$ paths the set $E(P_{u,v})\triangle E(Q)$ contains a cycle. Moreover, all the edges in the set $E(P_{u,v})\triangle E(Q)$ must be incident to the edge $e$. To see this note that if $e',e''\in E$ are edges which are not incident to $e$ then $e'/e = e''/e$ implies $e = e''$. Since $P_{u,v}/e = Q/e$, it follows that for any edge $e'$ not incident with $e$, $e' \in P_{u,v}$ iff $e'\in Q$. Let $\mathcal{T}$ be the tree system corresponding to $\mathcal{P}$ and write $e=xy$. If $\mathcal{P}$ contains all the edges in $E(P_{u,v})\triangle E(Q)$ then these edges are either in $T_x$ or $T_y$. Since $e$ is a persistent edge, $T_x =T_y$, implying $T_x$ contains a cycle, a contradiction. Therefore $Q$ contains an edge not in $\mathcal{P}$ and $w(Q) - w(P_{u,v}) \geq N -w(P_{u,v})>0.$
	\end{proof}
	
	We note that the condition of strict metrizability cannot be removed from the statement of \Cref{prop:StrictMetQuot}, see \Cref{fig:metQuotient}.
	
	\begin{figure}[h]
		\centering
		\begin{subfigure}[t]{.4\textwidth}
			\centering
			\begin{tikzpicture}[ scale=0.45, every node/.style={scale=0.45}]

			\node[draw,circle,minimum size=.5cm,inner sep=1pt] (1) at (0*360/5 +90: 5cm) [scale=2]{$1$};
			\node[draw,circle,minimum size=.5cm,inner sep=1pt] (2) at (1*360/5 +90: 5cm)  [scale=2]{$2$};
			\node[draw,circle,minimum size=.5cm,inner sep=1pt] (3) at (2*360/5 +90: 5cm)  [scale=2]{$3$};
			\node[draw,circle,minimum size=.5cm,inner sep=1pt] (4) at (3*360/5 +90: 5cm)  [scale=2]{$4$};
			\node[draw,circle,minimum size=.5cm,inner sep=1pt] (5) at (4*360/5 +90: 5cm) [scale=2] {$5$};
			
			\node (m1) at ($(2)!.5! (3)$) {};
			\node (m2) at ($(4)!.5! (5)$) {};
			\node[draw,circle,minimum size=.5cm,inner sep=1pt] (6) at ($(m1)!.8!270:(2)$) [scale=2] {$6$};
			\node[draw,circle,minimum size=.5cm,inner sep=1pt] (7) at  ($(m2)!.8!90:(5)$) [scale=2] {$7$};

			\draw [line width=2pt,-] (1) -- (2);
			\draw [line width=2pt,-] (2) -- (3);
			\draw [line width=2pt,-] (3) -- (4);
			\draw [line width=2pt,-] (4) -- (5);
			\draw [line width=2pt,-] (5) -- (1);
			\draw [line width=2pt,-] (6) -- (7);
			\draw [line width=2pt,-] (2) -- (6);
			\draw [line width=2pt,-] (5) -- (7);
			\draw [line width=2pt,-] (3) -- (6);
			\draw [line width=2pt,-] (4) -- (7);

			\end{tikzpicture}
			
		\end{subfigure}
		\hfill
		\begin{subfigure}[m]{.55\textwidth}
			\vspace{-5cm}
			\hspace{-100mm}
			$$(12) , \ (123), \ (1234), \ (15),\  (126), \ (157),$$\vspace{-.25mm}	$$(23), \ (234), \ (215), \ (26), \ (2157), \ (34), \ (345), $$\vspace{-.25mm}$$(36), \ (367), \ (45), \ (476), \ (47), \  (5126), \ (57), \ (67)$$

		\end{subfigure}

		\caption{This non-metrizable path system becomes metrizable by contracting the persistent edge $12$.}
		\label{fig:metQuotient}
	\end{figure}	
	
	Let $F$ be the set of $\mathcal{P}$-persistent edges. It is possible to contract the edges in $F$ either sequentially or in parallel. It is easy to check that the resulting path system $\mathcal{P} / F$ is {\em reduced}, i.e., it is a path system with no persistent edges.
	If every edge in $\mathcal{P}$ is persistent then all trees in the corresponding tree system are identical,  and $\mathcal{P}$ is the path system of a tree. We call such a path system {\em trivial}, and observe that this happens
	if and only if $G/F$ is a single vertex. 
	
	Trees are geodetic graphs, so the simplest non-trivial graphs for us are cycles. We turn to investigate path systems in cycles. For odd $n$ the $n$-cycle $C_n$ is a geodetic graph. We denote by $\mathcal{S}_n$ the geodesics in $C_n$ w.r.t.\ unit weight edges and note that $\mathcal{S}_n$ is a consistent path system. We show that these are essentially the only consistent path systems of cycles. 
	
	\begin{proposition} \label{prop:CyclePathSystem}
		Let $\mathcal{P}$ be a path system in the cycle $C_n$, $n\geq 3$, and let $F\subseteq E(C_n)$ be the set of all $\mathcal{P}$-persistent edges. Then either $\mathcal{P}$ is trivial or $\mathcal{P}/F = \mathcal{S}_m$, for some odd $3\leq m\leq n$. 
	\end{proposition}
	A spanning tree $T$ in $G=C_n$ has the form $E(C_n)\setminus\{e\}$ for some edge $e\in E(C_n)$. Therefore a tree system in $C_n$ is completely specified by a map $f:V(C_n)\to E(C_n)$ where $T_v = G \setminus f(v)$ for all $v\in V$. The following lemma says for which $f:V\to E$ the resulting path system is consistent:

	\begin{lemma}\label{lem:CrossingCycle}
		Let $C_n = (V,E)$, $n\geq 3$. A mapping $f:V\to E$ defines a consistent path system in $C_n$ if and only if for each $x,y \in C_n$ either $f(x) = f(y)$ or $x$ and $f(x)$ separate $y$ from $f(y)$. 
	\end{lemma}
	A look at \Cref{fig:CrossingFunction} shows why it is called the {\em crossing} condition. This is further elaborated in \Cref{sec:continuous}.
	\begin{proof}
		Suppose that $f$ defines a consistent system $\mathcal{P}$, $\mathcal{T}$ as above. 
		If $f(v) = wz$, the the spanning tree $T_v$ is comprised of two path $P_{v,w}, P_{v,z}$ whose intersection is the vertex $v$. Thus any $x \neq v$ is in either $P_{v,z}$ or $P_{v,w}$. Suppose the latter and moreover $f(x) \neq wz$. We wish to show that $f(x) \not\in P_{v,w}$ which implies that $v$ and $f(v)$ separate $x$ from $f(x)$. But $P_{v,w} = P_{v,x} \cup P_{x,w}$ and by the consistency of $\mathcal{P}$ both $P_{v,x}$ and $P_{x,w}$ are paths in $T_x$. However, if $f(x) \in P_{v,w} $ then $f(x) \in T_x$, contrary to the definition of $f$ and we conclude that $f$ is indeed crossing.\\
		Now we show that when $f$ is crossing the $uv$ path in $T_v$ and $T_u$ coincide for all $u,v \in C_n$. As before $T_v = C_n \setminus f(v)$ is the union of two paths $P_1$, $P_2$ that share $v$ as an endpoint. The $uv$ path of $T_v$ is contained in either $P_1$ or $P_2$, say $u\in P_1$. By the crossing condition $f(u) \notin P_1$ and we see that $P_1$ is a subpath of $T_u$ implying the $uv$ paths of $T_u$ and $T_v$ coincide. 
	\end{proof}
	\begin{proof}{[\Cref{prop:CyclePathSystem}]}
		Let $\mathcal{T}$ be the tree system corresponding to $\mathcal{P}$ and $f:V\to E$ its crossing function. By \Cref{lem:treePathSys1} an edge $e$ is persistent iff it belongs to every tree in $\mathcal{T}$, i.e., the set of persistent edges of $\mathcal{T}$ is precisely $F = E(C_n) \setminus \text{Im} f$. But $|V(C_n)|=|E(C_n)|=n$, so $\mathcal{P}$ is reduced if and only if $f$ is a bijection. If $\mathcal{P}$ is not trivial, the graph $C_n/F$ is a cycle and  $\mathcal{P}/F$ is a reduced path system over $C_n/F$. \\
		We now argue if $\mathcal{P}$ is a reduced system then $\mathcal{P} = \mathcal{S}_n$. As before, $T_v$ is the union of the paths $P_{v,w}$ and $P_{v,z}$ with $V(P_{v,w}) \cap V(P_{v,z}) = \{v\}$, where $f(v) = wz$. Consider $x\in P_{v,w}$, $v\neq x$. As $\mathcal{P}$ is reduced $f$ is a bijection and therefore $f(x) \neq f(v)$, and by the crossing condition $f(x) \in P_{v,z}$. This gives an injection from $V(P_{v,w})\setminus \{v\}$ into $E(P_{v,z})$. By symmetry there is also an injection from $V(P_{v,z})\setminus \{v\}$ into $E(P_{v,w})$. This implies $|E(P_{v,w}) | = |E(P_{v,z})| = k$ and $n=2k+1$. Moreover, for each $v\in C_n$ the edge $f(v)$ is the the edge antipodal to $v$ in $C_n$. This is the precisely the path system $\mathcal{S}_n$.
	\end{proof}
	\section{Metrizable Graphs}
	In this section we seek to to improve our understanding of (strictly) metrizable graphs. Recall that a graph $H$ is a topological minor of a graph $G$ if $G$ contains a subgraph which is a subdivision of $H$. For more on this, see \Cref{sec:est_met} and \Cref{sec:std}.
	\begin{proposition}\label{prop:TopMinClosed}
		A topological minor of a (strictly) metrizable graph $G=(V,E)$ is (strictly) metrizable.
	\end{proposition}
	\begin{proof}	
		We start with edge removals. Let $G'$ be a graph obtained by removing an edge from $G$. By remark \ref{rem:2conn}, we can assume that $G$ is $2$-connected, whence $G'$ is connected. Every path system in $G'$ is also a path system of $G$, and since $G$ is metrizable, so is $G'$. This applies verbatim to the strict case as well.
		
		Consider next what happens when we suppress a vertex $z$ of degree $2$ in $G$. Let $u, v$ be the two neighbors of $z$. We can assume w.l.o.g.\ that $uv\not\in E$, or else we first delete this edge from $G$ to obtain another metrizable graph. Let $G'$ be the graph obtained from $G$ by suppressing the vertex $z$. Given any path system $\mathcal{P}'$ in $G'$, we construct the following path system $\mathcal{P}$ in $G$. For $x,y\in V(G) \setminus \set{z}$, if $P'_{x,y}$ does not contain the edge $uv$ then $P_{x,y} := P'_{x,y}$. Otherwise, and if $P'_{x,y}=P'_{x,u}uvP'_{v,y}$, then $P_{x,y} := P'_{x,u}uzv P'_{v,y}$. Ditto when $P'_{x,y}=P'_{x,v}vuP'_{u,y}$. Next we need to define the paths $P_{z,w}$ for all $w\neq z$. If $P'_{u,w}$ does not contain the vertex $v$, set $P_{z,w} = z P'_{u,w}$. Otherwise, set $P_{z,w} \coloneqq z P'_{v,w}.$ 
		
		Note that $\mathcal{P}$ contains exactly one path between any two vertices, and we turn to show its consistency: I.e., that for every $P\in \mathcal{P}$ and every two vertices $x, y\in P$ the subpath of $P$ between $x$ and $y$ is also a member of $\mathcal{P}$. Since $\mathcal{P}'$ is consistent, it suffices to prove this when $z\in P$, i.e., when $P =P_1zP_2$ where $P_1$ is either empty or a path in $\mathcal{P}'$ that ends in $u$, and $P_2$ is either empty or a path in $\mathcal{P}'$ starting with $v$. If $x,y \in P_1$ or $x,y \in P_2$ then the consistency of $\mathcal{P}'$ implies that the $xy$ subpath of $P$ is in $\mathcal{P}$. If $x\in P_1, \ y\in P_2$ then both $P_1$ and $P_2$ are non-empty and the path $P'_{x,y}$ contains the edge $uv$. By construction $P_{x,y} = P'_{x,u}zP'_{v,y}$ and $P_{x,y}$ is a subpath of $P$. If $x\in P_1$ then $P'_{u,x}$ is a subpath of $P_1$ not containing $v$ so that $P_{z,x} = zP'_{u,x}$ is a subpath of $P$. Finally if $x\in P_2$ this means that the path $P'_{u,x} =P'_{u,v}P'_{v,x}$ contains $v$
		and therefore $P_{z,x} = zP'_{v,x}$ is again a subpath of $P$. Since $G$ is metrizable, there exists a weight function $w:E(G) \to (0,\infty)$ inducing $\mathcal{P}$. The weight function $w':E(G') \to (0,\infty)$ defined by $w'(uv) = w(uz) + w(zv)$ and $w'(e) = w(e)$, $e\neq uv$, induces $\mathcal{P}'$. Indeed, let $P'\in \mathcal{P}'$ be an $xy$ path and $Q$ another $xy$ path. If $P'$ contains the edge $uv$, then we can write $P' = R_1uvR_2$, for some subpaths $R_1$ and $R_2$ of $P'$, and we set $\tilde{P}' = R_1uzvR_2$. If $uv\notin P'$ we set $\tilde{P}' = P'$. We define $\tilde{Q}$ similarly. Notice that by construction of $\mathcal{P}$,  $\tilde{P}' \in \mathcal{P}$.  Since $w$ induces $\mathcal{P}$, by definition of $w'$ it follows, $w'(P') = w(\tilde{P}')\leq w(\tilde{Q}) =w'(Q)$. Moreover, if the original path system is strictly metrizable then $w'(P') = w(\tilde{P}')< w(\tilde{Q}) =w'(Q)$.
	\end{proof}
	
	A subgraph is also a topological minor, so
	
	\begin{corollary}\label{cor:subdiv}
		If a graph $G$ contains a subdivision of a non-metrizable graph then $G$ is non-metrizable. In particular, if $G$ has non-metrizable subgraph, then it is non-metrizable.
	\end{corollary}
	
	While the property of graph metrizability is closed under topological minors it is not closed under minors. In general, the graph obtained by contracting an edge of a metrizable graph is not necessarily metrizable. An example of this is given by \Cref{fig:EdgeContraction}. 
	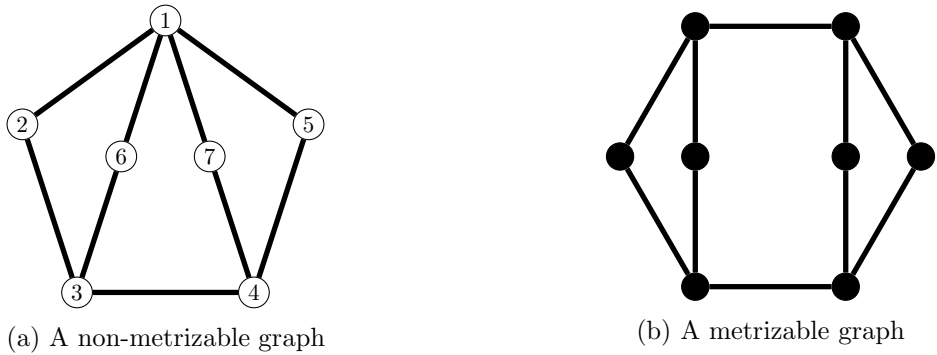
\begin{figure}[h]
		\centering
		\begin{subfigure}{0.33\textwidth}
			\centering
			\begin{tikzpicture}

			\node[draw,circle,minimum size=.5cm,inner sep=2pt] (1) at (0*360/5 +90: 2cm) [scale =.8] {$1$};
			\node[draw,circle,minimum size=.5cm,inner sep=2pt] (2) at (1*360/5 +90: 2cm) [scale =.8] {$2$};
			\node[draw,circle,minimum size=.5cm,inner sep=2pt] (3) at (2*360/5 +90: 2cm) [scale =.8] {$3$};
			\node[draw,circle,minimum size=.5cm,inner sep=2pt] (4) at (3*360/5 +90: 2cm) [scale =.8] {$4$};
			\node[draw,circle,minimum size=.5cm,inner sep=2pt] (5) at (4*360/5 +90: 2cm) [scale =.8] {$5$};
			
			\node[draw,circle,minimum size=.5cm,inner sep=2pt] (6) at ($(1)!0.5!(3)$) [scale =.8] {$6$};
			\node[draw,circle,minimum size=.5cm,inner sep=2pt] (7) at  ($(1)!0.5!(4)$) [scale =.8] {$7$};

			\draw [line width=2pt,-] (1) -- (2);
			\draw [line width=2pt,-] (2) -- (3);
			\draw [line width=2pt,-] (3) -- (4);
			\draw [line width=2pt,-] (4) -- (5);
			\draw [line width=2pt,-] (5) -- (1);
			\draw [line width=2pt,-] (1) -- (6);
			\draw [line width=2pt,-] (3) -- (6);
			\draw [line width=2pt,-] (1) -- (7);
			\draw [line width=2pt,-] (4) -- (7);
			
			\end{tikzpicture}
			\caption{A non-metrizable graph}
			\label{fig:edgeContractionNonMet}
		\end{subfigure}
		\hspace{20mm}
		\begin{subfigure}{0.33\textwidth}
			\centering
			\begin{tikzpicture}

			\node[draw,circle,fill] (1) at (0*360/6 +120: 2cm) {};
			\node[draw,circle,fill] (2) at (1*360/6 +120: 2cm) {};
			\node[draw,circle,fill] (3) at (2*360/6 +120: 2cm) {};
			\node[draw,circle,fill] (4) at (3*360/6 +120: 2cm) {};
			\node[draw,circle,fill] (5) at (4*360/6 +120: 2cm) {};
			\node[draw,circle,fill] (8) at (5*360/6 +120: 2cm) {};
			
			\node[draw,circle,fill] (6) at ($(1)!0.5!(3)$) {};
			\node[draw,circle,fill] (7) at  ($(4)!0.5!(8)$) {};

			\draw [line width=2pt,-] (1) -- (2);
			\draw [line width=2pt,-] (1) -- (6);
			\draw [line width=2pt,-] (1) -- (8);
			\draw [line width=2pt,-] (2) -- (3);
			\draw [line width=2pt,-] (3) -- (4);
			\draw [line width=2pt,-] (4) -- (5);
			\draw [line width=2pt,-] (5) -- (8);
			\draw [line width=2pt,-] (3) -- (6);
			\draw [line width=2pt,-] (7) -- (8);
			\draw [line width=2pt,-] (4) -- (7);
			
			\end{tikzpicture}
			
			\caption{A metrizable graph}
			\label{fig:edgeContractionMet}
		\end{subfigure}

		\caption{Contracting an edge of the metrizable graph in (b) yields the non-metrizable graph in (a).}
		\label{fig:EdgeContraction}
	\end{figure}
	
	Here is a non-metrizable path system in the graph \Cref{fig:edgeContractionNonMet}.
	Paths $P_{u,v}$ with $u<v$ are listed in lexicographic order. 
	\[
	(12); ~(163); ~(174); ~(15); ~(16); ~(17); ~(23); ~(234); ~(2345); ~(216); ~(217);
	\]
	\[
	(34); ~(345); ~(36); ~(347); ~(45); ~(436); ~(47); ~(5436); ~(517); ~(6347) 
	\]

	A weight function $w$ that induces this system must satisfy:
	\begin{equation*}
	\begin{split}
	w_{2,3} + w_{3,4} +w_{4,5}& \leq w_{1,2} + w_{1,5}\\
	w_{1,2} + w_{1,6} & \leq w_{2,3} + w_{3,6}\\
	w_{1,5} + w_{1,7} & \leq w_{4,5} + w_{4,7}\\
	w_{3,6} + w_{3,4} + w_{4,7} & \leq w_{1,6} + w_{1,7}\\
	\end{split}
	\end{equation*}
	Adding these inequalities and cancelling terms yields $w_{3,4} \leq 0$. At present, we can only verify that \Cref{fig:edgeContractionMet} is metrizable by using a computer program to check that all of the possible path systems in this graph is metrizable. For more on this, see \Cref{subsec:praktisch}.
	\section{Metrizable Graphs are Rare} \label{sec:metRare}
	
	By \Cref{prop:TopMinClosed} metrizability can be characterized by a set of forbidden topological minors. Of course it suffices to consider the minimal such forbidden graphs. As we will see (\Cref{thm:finiteMinGraphs}), there are only finitely many such graphs. With the help of a computer we found several such graphs, \Cref{fig:MinimumGraphs}, but we suspect that this list is not exhaustive. For a discussion on how these graphs were found see \Cref{subsec:praktisch}. \Cref{prop:TopMinClosed} and the graphs in \Cref{fig:MinimumGraphs} already imply that metrizability is a rare property of graphs.
	\begin{theorem}\label{thm:4Conn}
		Every $4$-connected metrizable graph has at most $6$ vertices. 
	\end{theorem}
	\begin{figure}[h]
		\centering
		\begin{subfigure}[t]{0.3\textwidth}
			\centering
			\begin{tikzpicture}[scale=0.85, every node/.style={scale=0.85}]

			\node[draw,circle,minimum size=.6cm,inner sep=0pt] (u) at (-2,0) {$u$};
			\node[draw,circle,minimum size=.6cm,inner sep=0pt] (1) at (0,1.5) {$w_1$};
			\node[draw,circle,fill] (2) at (0,.5) {};
			\node[draw,circle,fill] (3) at (0,-.5) {};
			\node[draw,circle,fill] (4) at (0,-1.5) {};
			\node[draw,circle,minimum size=.6cm,inner sep=0pt] (v) at (2,0) {$v$};
			\node[draw,circle,minimum size=.6cm,inner sep=0pt] (z) at (-2,1.5) {$z$};

			\draw [line width=2pt,-] (u) -- (1);
			\draw [red1,line width=2pt,-] (u) -- (2);
			\draw [red1,line width=2pt,-] (u) -- (3);
			\draw [red1,line width=2pt,-] (u) -- (4);
			\draw [red1,line width=2pt,-] (v) -- (1);
			\draw [red1,line width=2pt,-] (v) -- (2);
			\draw [red1,line width=2pt,-] (v) -- (3);
			\draw [red1,line width=2pt,-] (v) -- (4);
			\draw [red1,line width=2pt,-] (u) -- (z) ;
			\draw [red1,line width=1.5pt,loosely dashed] (z) -- (1) node [scale = 1.2,black, midway, above] {\large$P$};

			\end{tikzpicture}
			\caption{The red edges form a subdivision of \Cref{fig:graph1}.}
			\label{fig:4ConnFig1}
		\end{subfigure}
		\hfill
		\begin{subfigure}[t]{0.3\textwidth}
			\centering
			\begin{tikzpicture}[scale=0.85, every node/.style={scale=0.85}]

			\node[draw,circle,minimum size=.6cm,inner sep=0pt] (u) at (-2,0) {$u$};
			\node[draw,circle,fill](1) at (0,1.5) {};
			\node[draw,circle,fill] (2) at (0,.5) {};
			\node[draw,circle,fill] (3) at (0,-.5) {};
			\node[draw,circle,fill] (4) at (0,-1.5) {};
			\node[draw,circle,minimum size=.6cm,inner sep=0pt] (v) at (2,0) {$v$};
			\node[draw,circle,minimum size=.6cm,inner sep=0pt] (z) at (2,1.5) {$z$};

			\draw [line width=2pt,-] (u) -- (1);
			\draw [line width=2pt,-] (u) -- (2);
			\draw [line width=2pt,-] (u) -- (3);
			\draw [line width=2pt,-] (u) -- (4);
			\draw [line width=2pt,-] (v) -- (1);
			\draw [line width=2pt,-] (v) -- (2);
			\draw [line width=2pt,-] (v) -- (3);
			\draw [line width=2pt,-] (v) -- (4);
			\draw [line width=1pt,-] (z) -- (1);
			\draw [line width=1pt,-] (z) -- (2);
			\draw [line width=1pt,-] (z) -- (3);
			\draw [line width=1pt,-] (z) -- (4);

			\end{tikzpicture}
			\caption{If $u$ and $v$, as in the proof of \Cref{thm:4Conn}, are both of degree $4$ then $G$ contains a $K_{3,4}$ subgraph.}
			\label{fig:4ConnFig2}
		\end{subfigure}
		\hfill
		\begin{subfigure}[t]{0.3\textwidth}
			\centering
			\begin{tikzpicture}[scale=0.85, every node/.style={scale=0.85}]

			\node[draw,circle,minimum size=.6cm,inner sep=0pt] (u) at (-2,0) {$u$};
			\node[draw,circle,minimum size=.6cm,inner sep=0pt] (1) at (0,1.5) {$w_1$};
			\node[draw,circle,minimum size=.6cm,inner sep=0pt] (2) at (0,.5) {$w_2$};
			\node[draw,circle,fill] (3) at (0,-.5) {};
			\node[draw,circle,fill] (4) at (0,-1.5) {};
			\node[draw,circle,minimum size=.6cm,inner sep=0pt] (v) at (2,0) {$v$};
			\node[draw,circle,minimum size=.6cm,inner sep=0pt] (z) at (2,1.5) {$z$};

			\draw [line width=2pt,-] (u) -- (2);
			\draw [line width=2pt,-] (u) -- (3);
			\draw [line width=2pt,-] (u) -- (4);
			\draw [line width=2pt,-] (v) -- (2);
			\draw [line width=2pt,-] (v) -- (3);
			\draw [line width=2pt,-] (v) -- (4);
			\draw [line width=2pt,-] (z) -- (1);
			\draw [line width=1pt,-] (z) -- (3);
			\draw [line width=1pt,-] (z) -- (4);
			\draw [line width=2pt,-] (1) -- (2);

			\end{tikzpicture}
			\caption{This graph is a copy of \Cref{fig:graph3}.}
			\label{fig:4ConnFig3}
		\end{subfigure}
		\caption{}
	\end{figure}
	\begin{proof}
		Let $G$ be a $4$-connected graph with at least $7$ vertices. In view of \Cref{cor:subdiv} is suffices to show that $G$ contains a subgraph or subdivision of a non-metrizable graph.
		Since there exist non-metrizable graphs on $7$ vertices, $K_n$ is metrizable only if $n\leq 6$. Therefore, it suffices to consider the case where $G$ contains a pair of non-adjacent vertices, say $u$ and $v$. There exist at least four disjoint paths of length at least $2$ connecting $u$ and $v$. If any of these paths has length at least $3$, then $G$ is not metrizable, since it contains \Cref{fig:graph1} as a topological minor. Therefore, $u$ and $v$ have at least four common neighbors, say $A=\set{w_1,w_2,w_3,w_4}$ is a set of four common neighbors. Suppose that $u$ has a neighbor $z \notin A$. By the fan lemma there exist four openly disjoint paths from $z$ to $A$. At least two of these paths contain neither $u$ nor $v$. Say $P$ is a $z$-$w_1$ path with this property, then the paths $uzPw_1v, \ uw_2v, \ uw_3v, \ uw_4v$ are four disjoint $u$-$v$ paths, \Cref{fig:4ConnFig1}. Therefore $G$ is not metrizable, since it contains \Cref{fig:graph1} as a t.m. There remains the case where $N(u)=A$. Every vertex $z \notin A\cup \set{u,v}$ has an $A$-fan. As argued above this fan must simply be comprised of four edges, i.e., $N(z) \supseteq A$, \Cref{fig:4ConnFig2}. If $G$ contains another vertex $y$ then again $N(y) \supseteq A$ and $G$ hss a $K_{4,4}$ subgraph. But $K_{4,4}$ is not metrizable as it contains a subdivision of \Cref{fig:graph3}. What if $G$ has only seven vertices? Since $G$ is $4$-connected, the set $\{u,v,z\}$ cannot disconnect it, and there must be edges within $A$, say, $w_1w_2\in E$. But then $G$ contains \Cref{fig:graph3} as a subgraph, \Cref{fig:4ConnFig2}, and is therefore not metrizable.
	\end{proof}
	\begin{figure}[h]
		\centering
		\begin{subfigure}[t]{0.3\textwidth}
			\centering
			\begin{tikzpicture}[scale=0.85, every node/.style={scale=0.85}]
			
			\node[draw,circle,minimum size=.6cm,inner sep=0pt] (v1) at (0*360/5 +90: 2.5cm) {$v_1$};
			\node[draw,circle,minimum size=.6cm,inner sep=0pt] (v5) at (1*360/5 +90: 2.5cm) {$v_5$};
			\node[draw,circle,minimum size=.6cm,inner sep=0pt] (v3) at (2*360/5 +90: 2.5cm) {$v_3$};
			\node[draw,circle,minimum size=.6cm,inner sep=0pt] (v2) at (3*360/5 +90: 2.5cm) {$v_2$};
			\node[draw,circle,minimum size=.6cm,inner sep=0pt] (v4) at (4*360/5 +90: 2.5cm) {$v_4$};
			\node[draw,circle,minimum size=.2cm,fill] (w1) at ($(v1)!0.5!(v2)$) {};
			\node[draw,circle,minimum size=.2cm,fill] (w2) at ($(v1)!0.5!(v3)$) {};

			\draw [line width=2pt,-,red1] (v1) -- (v4);
			\draw [line width=2pt,-,red1] (v1) -- (v5);
			\draw [line width=2pt,-,red1] (v2) -- (v3);
			\draw [line width=2pt,-,red1] (v2) -- (v4);
			\draw [line width=2pt,-,darkgrey1] (v2) -- (v5);
			\draw [line width=2pt,-,darkgrey1] (v3) -- (v4);
			\draw [line width=2pt,-,red1] (v3) -- (v5);
			\draw [line width=2pt,-,darkgrey1] (v4) -- (v5);
			
			\draw [line width=2pt,-,red1] (v1) -- (w1);
			\draw [line width=2pt,-,red1] (w1) -- (v2);
			\draw [line width=2pt,-,red1] (v1) -- (w2);
			\draw [line width=2pt,-,red1] (w2) -- (v3);
			
			\end{tikzpicture}
			\caption{If $H$, as in the proof of \Cref{thm:NonPlamarMet}, contains two incident subdivided edges then $G$ contains a subdivision of \Cref{fig:graph2}.}
			\label{fig:2ConnNonPlanarFig1}
		\end{subfigure}
		\hfill
		\begin{subfigure}[t]{0.3\textwidth}
			\centering
			\begin{tikzpicture}[scale=0.85, every node/.style={scale=0.85}]
			
			\node[draw,circle,minimum size=.6cm,inner sep=0pt] (v1) at (2,-2) {$v_1$};
			\node[draw,circle,minimum size=.6cm,inner sep=0pt] (v2) at (0,-2) {$v_2$};
			\node[draw,circle,minimum size=.6cm,inner sep=0pt] (v3) at (-2,-2) {$v_3$};
			\node[draw,circle,minimum size=.6cm,inner sep=0pt] (v4) at (0,2) {$v_4$};
			\node[draw,circle,minimum size=.6cm,inner sep=0pt] (v5) at (2,2) {$v_5$};
			\node[draw,circle,minimum size=.6cm,inner sep=0pt] (w) at (-2,2) {$w$};
			\node[draw,circle,minimum size=.6cm,inner sep=0pt] (u) at (-2,0) {$u$};
			
			\draw [line width=2pt,-,red1] (v1) -- (w);
			\draw [line width=2pt,-,red1] (u) -- (w);
			\draw [line width=2pt,-,red1] (v2) -- (w);
			\draw [line width=2pt,-] (v1) to[out=-150,in=-30] (v3);
			\draw [line width=2pt,-,red1] (v1) -- (v4);
			\draw [line width=2pt,-,red1] (v1) -- (v5);
			\draw [line width=2pt,-] (v2) -- (v3);
			\draw [line width=2pt,-,red1] (v2) -- (v4);
			\draw [line width=2pt,-,red1] (v2) -- (v5);
			\draw [line width=2pt,-,red1] (v3) -- (v4);
			\draw [line width=2pt,-,red1] (v3) -- (u);
			\draw [line width=2pt,-,red1] (v3) -- (v5);
			\draw [line width=2pt,-] (v4) -- (v5);
			
			\end{tikzpicture}
			\caption{This graph is not metrizable since it contains a \Cref{fig:graph3} subgraph.}
			\label{fig:2ConnNonPlanarFig2}
		\end{subfigure}
		\hfill
		\begin{subfigure}[t]{0.3\textwidth}
			\centering
			\begin{tikzpicture}[scale=0.85, every node/.style={scale=0.85}]
			
			\node[draw,circle,minimum size=.6cm,inner sep=0pt] (v1) at (0*360/5 +90: 2.5cm) {$v_1$};
			\node[draw,circle,minimum size=.6cm,inner sep=0pt] (v5) at (1*360/5 +90: 1.75cm) {$v_5$};
			\node[draw,circle,minimum size=.6cm,inner sep=0pt] (v3) at (2*360/5 +90: 2.5cm) {$v_3$};
			\node[draw,circle,minimum size=.6cm,inner sep=0pt] (v2) at (3*360/5 +90: 2.5cm) {$v_2$};
			\node[draw,circle,minimum size=.6cm,inner sep=0pt] (v4) at (4*360/5 +90: 1.75cm) {$v_4$};
			\node[draw,circle,minimum size=.6cm,inner sep=0pt] (u1) at (1*360/5 +90: 3cm) {$u$};
			\node[draw,circle,minimum size=.6cm,inner sep=0pt] (u2) at (4*360/5 +90: 3cm) {$\tilde{u}$};

			\draw [line width=2pt,-,darkgrey1] (v1) -- (v2);
			\draw [line width=2pt,-,darkgrey1] (v1) -- (v3);
			\draw [line width=2pt,-,red1] (v1) -- (v4);
			\draw [line width=2pt,-,red1] (v1) -- (v5);
			\draw [line width=2pt,-,red1] (v2) -- (v3);
			\draw [line width=2pt,-,red1] (v2) -- (v4);
			\draw [line width=2pt,-,darkgrey1] (v2) -- (v5);
			\draw [line width=2pt,-,darkgrey1] (v3) -- (v4);
			\draw [line width=2pt,-,red1] (v3) -- (v5);
			\draw [line width=2pt,-,darkgrey1] (v4) -- (v5);
			
			\draw [line width=2pt,-,red1] (v1) -- (u2);
			\draw [line width=2pt,-,red1] (v2) -- (u2);
			\draw [line width=2pt,-,red1] (v1) -- (u1);
			\draw [line width=2pt,-,red1] (v3) -- (u1);

			\end{tikzpicture}
			\caption{The red edges form a copy of \Cref{fig:graph2}.}
			\label{fig:2ConnNonPlanarFig3}
		\end{subfigure}
		\caption{}
	\end{figure}
	Next we show most non-planar graphs are not metrizable.
	\begin{theorem}\label{thm:NonPlamarMet}
		Every metrizable $2$-connected non-planar graph has at most $7$ vertices.
	\end{theorem}
	\begin{proof}
		Let $G$ be a $2$-connected non-planar graph with at least $8$ vertices. By Kuratowski's theorem $G$ contains a subgraph $H$ that is a subdivision of either $K_5$ or $K_{3,3}$. Consider first the case where $H$ is a subdivision of $K_5$.
		Let $v_1,v_2,v_3,v_4, v_5$ be the vertices of $H$ with degree $4$ and $P_{ij}$ the $v_iv_j$ path in $H$ that is the subdivided edge $v_iv_j$.\\
		{\bf Comment:} Below we arbitrarily and without further notice break symmetries between $v_1,\dots, v_5$.\\
		We may assume that every $P_{ij}$ has length $1$ or $2$, for if, say, $P_{12}$, has length at least $3$, then the union of paths $P_{12}, [P_{13} P_{32}], [P_{14} P_{42}], [P_{15} P_{52}]$ is a subdivision of \Cref{fig:graph1}. If, say, $P_{12}$ and $P_{13}$ have length $2$, then the union of the paths $P_{12}, P_{13}, \ P_{14} P_{42},  \ P_{15} P_{53}, P_{23}$ is a subdivision of \Cref{fig:graph2}, see \Cref{fig:2ConnNonPlanarFig1}. Henceforth we can and will assume that distinct paths $P_{ij}$ of length $2$ are disjoint. This, in particular means that $H$ has at most seven vertices and there is some vertex, say $u\in V(G)\setminus V(H)$. Next we show that some specific conditions imply that $G$ is not metrizable. 
		\begin{enumerate}[(1)]
			\item There is an edge $uw \in E(G)$, where $w$ is a degree $2$ vertex in $H$.
			\begin{proof}
				W.l.o.g.\ $w$ is the middle vertex of $P_{12}$. Since $G$ is $2$-connected there is another path from $u$ to $H$ which doesn't contain $w$. It suffices to consider the case where this other path is also an edge $e=u\tilde{w}$. If $d_H(\tilde{w})=2$, say it is the middle vertex of $P_{34}$, then $\tilde{w}v_4v_2v_5v_1w_1u\tilde{w}$ is a $7$-cycle in $G$. Along with the edges $v_3\tilde{w}$, $v_3v_1$ and $v_3v_2$ this gives a copy of \Cref{fig:graph6}, see \Cref{fig:2ConnNonPlanarFig4}.\\
				Now suppose that $e=uv_i$. If $i=1,2$, say $i=1$, then we redefine $P_{12}$ to be the path $v_1uwv_2$, yielding a subdivision of $K_5$ where $P_{12}$ has length $3$. It follows that $G$ is not metrizable. So we can assume $i=3,4,5$, say $i=3$. In this case, the union of the paths $wuv_3, wv_1, \ wv_2, \ P_{14} , \ P_{24}, \ P_{34}, \ P_{15}, \ P_{25}, \ P_{35} $ forms a subdivision of \Cref{fig:graph3}, see \Cref{fig:2ConnNonPlanarFig2}.
			\end{proof}
			\item $uv_1, uv_2\in E(G)$ and some path $P_{1k}$ has length $2$.
			\begin{proof}
				If $k=2$ then the paths $v_1uv_2, \ P_{12}, P_{13}P_{23} , \ P_{14} P_{45} P_{25}$ form a subdivision of \Cref{fig:graph1}. Otherwise, say $k=3$. The paths $v_1uv_3, \ P_{12}P_{23}, \ P_{14}P_{34}, \ P_{15} P_{35}$ form a subdivision of \Cref{fig:graph1}. 
			\end{proof}
			\item $\tilde{u} \neq u\in V(G)\setminus V(H)$ and
			$uv_1$, $uv_2, \tilde{u}v_1, \tilde{u}v_k \in E(G)$, and $k\neq 1$.
			\begin{proof}
				If $k=2$, then the paths $v_1uv_2, \ v_1\tilde{u}v_2, \ P_{13}P_{23} , \ P_{14}P_{45}P_{25}$ form a subdivision of \Cref{fig:graph1}. Otherwise, say $k=3$ and the paths $v_1uv_2, \ v_1\tilde{u}v_3, \ P_{14}P_{24}, \ P_{15}P_{35}, \ P_{2,3}$ form a subdivision of \Cref{fig:graph2}, see \Cref{fig:2ConnNonPlanarFig3}. 
			\end{proof}
			\item There is an edge $e=u_1u_2$ in $G$ that is disjoint from $H$
			\begin{proof}
				By Menger's theorem there are disjoint $u_1-H$, $u_2-H$ paths $Q_1$, $Q_2$. By $(1)$, we can assume that $Q_1, Q_2$ end at $v_i\neq v_j$ respectively. Replacing the path $P_{ij}$ with $v_iQ_1u_1u_2Q_2v_j$ we get another subdision of $K_5$ with $P_{ij}$ having length $3$.
			\end{proof}
		\end{enumerate}
		Our analysis of the general case is based on the above conditions and is parametrized by $|V(H)|$.
		Since $G$ is $2$-connected for every $u\in V(G) \setminus V(H)$ there are two disjoint $u-H$ paths in $G$. By condition $(4)$ we can assume that these paths are in fact edges.
		
		We have previously dealt with the case $|V(H)|\ge 8$ and move now to assume $|V(H)|=7$. This means that two paths, say $P_{12}$ and $P_{34}$, have length $2$ and there is a vertex $u\in V(G)\setminus V(H)$ with two neighbors in $H$. We may assume that these two neighbors are $v_i$ and $v_j$ for some $i\neq j$. For if $u$ has a neighbor in $H$ whose degree in $H$ is $2$, then case $(1)$ applies. Therefore the neighbors of $u$ must be $v_i$ and $v_j$, $i\neq j$. Since $\set{i,j} \cap \set{1,2,3,4}\neq\emptyset$, $G$ is not metrizable by condition $(2)$. 
		\begin{figure}[h]
			\centering
			\begin{subfigure}[t]{0.3\textwidth}
				\centering
				\begin{tikzpicture}[scale=0.85, every node/.style={scale=0.85}]
				
				\node[draw,circle,minimum size=.6cm,inner sep=0pt] (v2) at (0*360/7 +90: 2.5cm) {$v_2$};
				\node[draw,circle,minimum size=.6cm,inner sep=0pt] (v5) at (1*360/7 +90: 2.5cm) {$v_5$};
				\node[draw,circle,minimum size=.6cm,inner sep=0pt] (v1) at (2*360/7 +90: 2.5cm) {$v_1$};
				\node[draw,circle,minimum size=.6cm,inner sep=0pt] (w1) at (3*360/7 +90: 2.5cm) {$w_1$};
				\node[draw,circle,minimum size=.6cm,inner sep=0pt] (u) at (4*360/7 +90: 2.5cm) {$u$};
				\node[draw,circle,minimum size=.6cm,inner sep=0pt] (w2) at (5*360/7 +90: 2.5cm) {$w_2$};
				\node[draw,circle,minimum size=.6cm,inner sep=0pt] (v4) at (6*360/7+90: 2.5cm) {$v_4$};
				\node[draw,circle,minimum size=.6cm,inner sep=0pt] (v3) at (0,0) {$v_3$};

				\draw [line width=2pt,-,darkgrey1] (v1) -- (v4);
				\draw [line width=2pt,-,red1] (v1) -- (v5);
				
				\draw [line width=2pt,-,red1] (v2) -- (v4);
				\draw [line width=2pt,-,red1] (v2) -- (v5);
				\draw [line width=2pt,-,darkgrey1] (v3) -- (v4);
				\draw [line width=2pt,-,darkgrey1] (v3) -- (v5);
				\draw [line width=2pt,-,darkgrey1] (v4) -- (v5);
				
				\draw [line width=2pt,-,red1] (v1) -- (w1);
				\draw [line width=2pt,-,red1,darkgrey1] (w1) -- (v2);
				\draw [line width=2pt,-,red1] (w2) -- (v3);
				\draw [line width=2pt,-,red1] (v1) -- (v3);
				\draw [line width=2pt,-,red1] (v2) -- (v3);
				\draw [line width=2pt,-,red1] (w1) -- (u);
				\draw [line width=2pt,-,red1] (w2) -- (u);
				\draw [line width=2pt,-,red1] (v4) -- (w2);
				
				\end{tikzpicture}
				\caption{The colored edges form a \Cref{fig:graph6}  subgraph.}
				\label{fig:2ConnNonPlanarFig4}
			\end{subfigure}
			\hfill
			\begin{subfigure}[t]{0.3\textwidth}
				\centering
				\begin{tikzpicture}[scale=0.85, every node/.style={scale=0.85}]
				
				\node[draw,circle,minimum size=.6cm,inner sep=0pt] (v1) at (-2,2) {$v_1$};
				\node[draw,circle,minimum size=.6cm,inner sep=0pt] (v2) at (0,2) {$v_2$};
				\node[draw,circle,minimum size=.6cm,inner sep=0pt] (v3) at (2,2) {$v_3$};
				\node[draw,circle,minimum size=.6cm,inner sep=0pt] (u1) at (-2,-2) {$u_1$};
				\node[draw,circle,minimum size=.6cm,inner sep=0pt] (u2) at (0,-2) {$u_2$};
				\node[draw,circle,minimum size=.6cm,inner sep=0pt] (u3) at (2,-2) {$u_3$};
				\node[draw,circle,minimum size=.6cm,inner sep=0pt] (w1) at (-3,0) {$w_1$};

				\draw [line width=2pt,-,darkgrey1] (v1) -- (u1);
				\draw [line width=2pt,-,red1] (v1) -- (u2);
				\draw [line width=2pt,-,red1] (v1) -- (u3);
				\draw [line width=2pt,-,red1] (v2) -- (u1);
				\draw [line width=2pt,-,red1] (v2) -- (u2);
				\draw [line width=2pt,-,red1] (v2) -- (u3);
				\draw [line width=2pt,-,red1] (v3) -- (u1);
				\draw [line width=2pt,-,red1] (v3) -- (u2);
				\draw [line width=2pt,-,red1] (v3) -- (u3);
				\draw [line width=2pt,-,red1] (v1) -- (w1) node [midway, above left,black] {$Q_1$};
				\draw [line width=2pt,-,red1] (u1) -- (w1) node [midway,below left,black ] {$Q_2$};

				\end{tikzpicture}
				\caption{A proper subdivision of $K_{3,3}$ is a subdivision of \Cref{fig:graph3}.}
				\label{fig:2ConnNonPlanarFig5}
			\end{subfigure}
			\hfill
			\begin{subfigure}[t]{0.3\textwidth}
				\centering
				\hspace*{-.6cm}
				\begin{tikzpicture}[scale=0.85, every node/.style={scale=0.85}]
				\node[draw,circle,minimum size=.6cm,inner sep=0pt] (w1) at (0*360/5 +90: 2.25cm) {$w_1$};
				\node[draw,circle,minimum size=.6cm,inner sep=0pt] (v1) at (1*360/5 +90: 2.25cm) {$v_1$};
				\node[draw,circle,minimum size=.6cm,inner sep=0pt] (u1) at (2*360/5 +90: 2.25cm) {$u_1$};
				\node[draw,circle,minimum size=.6cm,inner sep=0pt] (u2) at (3*360/5 +90: 2.25cm) {$u_2$};
				\node[draw,circle,minimum size=.6cm,inner sep=0pt] (v2) at (4*360/5 +90: 2.25cm) {$v_2$};
				\node[draw,circle,minimum size=.6cm,inner sep=0pt] (u3) at ($(v1)!.5!(v2)+(0,-.5)$) {$u_3$};
				
				\draw [line width=2pt,-,darkgrey1] (v1) to [out=-65,in=160, distance=1cm]  (u2);
				\draw [line width=2pt,-,darkgrey1] (v2) to [out=245,in=20, distance=1cm] (u1);
				\node[draw,circle,minimum size=.6cm,inner sep=0pt,fill = white]  (v3) at ($(v1)!.5!(v2) + (0,-1.5)$) {$v_3$};
				
				\node[draw,circle,minimum size=.6cm,inner sep=0pt] (w2) at ($(u1) !.5!(u2)$) {$w_2$};
				
				\draw [line width=2pt,-,red1] (v1) -- (u1);
				
				\draw [line width=2pt,-,red1] (v1) -- (u3);
				\draw [line width=2pt,-,red1] (v2) -- (u2);
				\draw [line width=2pt,-,red1] (v2) -- (u3);
				\draw [line width=2pt,-,red1] (v3) -- (u1);
				\draw [line width=2pt,-,red1] (v3) -- (u2);
				\draw [line width=2pt,-,red1] (v3) -- (u3);
				\draw [line width=2pt,-,red1] (v1) -- (w1);
				\draw [line width=2pt,-,red1] (v2) -- (w1);
				\draw [line width=2pt,-,red1] (u1) -- (w2);
				\draw [line width=2pt,-,red1] (u2) -- (w2);

				\end{tikzpicture}
				\caption{This graph contains a \Cref{fig:graph5} subgraph and is therefore not metrizable. }
				\label{fig:2ConnNonPlanarFig6}
			\end{subfigure}
			\caption{}
		\end{figure}

		We next assume $|V(H)|=6$ in which case $P_{12} = v_1wv_2$ is a path of length $2$ and there are two vertices $u_1, u_2 \not \in V(H)$. Again by $(1)$ we can assume that $w$ is neither a neighbor of $u_1$ nor $u_2$. If $u_1$ or $u_2$ share a neighbor in $H$ then condition $(3)$ holds. Otherwise the neighbors of $u_1$ and $u_2$ in $H$ are all distinct. This implies at least one of these neighbors is either $v_1$ or $v_2$ and $G$ is not metrizable by $(2)$.
		
		The last case to check is $|V(H)|=5$. Then there exists three vertices $u_1,u_2,u_3 \in V(G) \setminus V(H)$, each having at least $2$ neighbors in $H$. But $|V(H)|=5$, so these neighbors cannot all be distinct and condition $(3)$ applies.
		
		Now we consider the case where $G$ contains a subdivision $H$ of $K_{3,3}$. If $H$ contains any vertices of degree $2$ then $G$ contains a subdivision of \Cref{fig:graph3}. Therefore we can assume that $H=K_{3,3}$ with vertex bipartition $A=\set{v_1,v_2,v_3}$, $B=\set{u_1,u_2,u_3}$ and $w_1,w_2 \in V(G) \setminus V(H)$. By the fan lemma there exists two disjoint $w_1-H$ paths $Q_1$ and $Q_2$. If $Q_1$ is a $w_1-A$ path and $Q_2$ a $w_1 - B$ path then $G$ again contains a subdivision of \Cref{fig:graph3}, see \Cref{fig:2ConnNonPlanarFig5}. So we can assume that $Q_1$ and $Q_2$ are both $u_1-A$ paths and have endpoints $v_1$ and $v_2$, respectively. If either $Q_1$ or $Q_2$ have length greater than $1$ then the paths  $v_1Q_1w_1Q_2v_2, \ v_1u_1v_2, \ v_1u_2v_2, \ v_1u_3v_2$ form a subdivision of \Cref{fig:graph1}. Therefore $Q_1 = w_1v_1$ and $Q_2= w_1v_2$. Likewise $w_2$ has two neighbors in $H$ which both belong to either $A$ or $B$. Say both are in $A$. If the neighbors of $w_1$ and $w_2$ coincide then the paths $v_1w_1v_2, v_1w_2v_2, \ v_1u_1v_2, \ v_1u_2v_3u_3v_2$ form a subdivision of \Cref{fig:graph1}. Otherwise, $w_2$ shares only one neighbor with $w_1$, say $v_2$. Again, the paths $v_1 w_1v_2w_2v_3, \ v_1u_1v_3, \ v_1u_2v_3, \ v_1u_3v_3$ form a subdivision of \Cref{fig:graph1}. Now suppose the neighbors of $w_2$ are $u_1,\ u_2\in B$. Then the graph $H' = (H- \set{v_1u_2, v_2u_1}) \cup \set{w_1v_1,w_1v_2, w_2u_1, w_2u_2} $ is a subdivision of \Cref{fig:graph5}, see \Cref{fig:2ConnNonPlanarFig6}.
	\end{proof}
	
	Next we consider the planar case, starting with some technical lemmas. Recall the $n$-wheel $W_n$, which is the graph obtained by joining an $n$-cycle with a single vertex called its {\em pivot}.

	\begin{lemma}\label{lem:list3minConnGraphs}
		The only minimally $3$-connected graph on $5$ vertices is $W_4$.
		The minimally $3$-connected planar graphs on $6$ vertices are precisely $W_5$ and the $3$-prism $Y_3$, \Cref{fig:min3ConnPlane6Vtxs}. 	The minimally $3$-connected planar graphs on $7$ vertices are precisely those in \Cref{fig:min3ConnPlane7Vtxs}.
	\end{lemma}
	\begin{proof}
		This can be shown using the procedure outlined in \cite{Da}. Explicitly,  if $G\neq K_4$  is a minimally $3$-connected graph then there exists a minimally $3$-connected graph $G'$, of order strictly less than $G$, such that $G$ can be obtained from $G'$ by applying one of the following operations:
		
		\begin{enumerate}
			\item  Take a vertex $u$ and a non-incident edge $e$ and subdivide $e$ to include a new vertex $z$. Then add the new edge $uz$.
			\item Subdivide two edges $e_1$ and $e_2$ with new vertices $u_1$ and $u_2$, respectively, and add the new edge $u_1 u_2$.
			\item Introduce a new vertex $v$ and make it a neighbor of three distinct vertices $x$, $y$ and $z$.
		\end{enumerate}
		Dawes \cite{Da} gives a sufficient and necessary condition for when one the above operations, applied to  minimally $3$-connected, yields another minimally $3$-connected graph. Moreover, any of the above operations yields a planar graph iff all the involved vertices and/or edges lie in the same face. 
		Therefore, starting with $K_4$ one can generate all the $3$-minimally connected planar graphs. 
	\end{proof}
	\begin{figure}[h]
		\centering
		\begin{subfigure}{0.33\textwidth}
			\centering
			\begin{tikzpicture}[scale=0.85, every node/.style={scale=0.85}]

			\node[draw,circle,fill] (1) at (0*360/5 +90: 2cm) {};
			\node[draw,circle,fill] (2) at (1*360/5 +90: 2cm) {};
			\node[draw,circle,fill] (3) at (2*360/5 +90: 2cm) {};
			\node[draw,circle,fill] (4) at (3*360/5 +90: 2cm) {};
			\node[draw,circle,fill] (5) at (4*360/5 +90: 2cm) {};
			\node[draw,circle,fill] (6) at (0,0) {};

			\draw [line width=2pt,-] (1) -- (2);
			\draw [line width=2pt,-] (1) -- (5);
			\draw [line width=2pt,-] (1) -- (6);
			\draw [line width=2pt,-] (2) -- (3);
			\draw [line width=2pt,-] (2) -- (6);
			\draw [line width=2pt,-] (3) -- (4);
			\draw [line width=2pt,-] (3) -- (6);
			\draw [line width=2pt,-] (4) -- (5);
			\draw [line width=2pt,-] (4) -- (6);
			\draw [line width=2pt,-] (5) -- (6);

			\end{tikzpicture}
			\caption{The $5$-wheel, $W_5$}
			\label{fig:Wheel5}
		\end{subfigure}
		\hspace{20mm}
		\begin{subfigure}{0.33\textwidth}
			\centering
			\begin{tikzpicture}[scale=0.85, every node/.style={scale=0.85}]

			\node[draw,circle,fill] (1) at (0*360/3 +90: .75cm) {};
			\node[draw,circle,fill] (2) at (1*360/3 +90: .75cm) {};
			\node[draw,circle,fill] (3) at (2*360/3 +90: .75cm) {};
			\node[draw,circle,fill] (4) at (0*360/3 +90: 2.5cm) {};
			\node[draw,circle,fill] (5) at (1*360/3 +90: 2.5cm) {};
			\node[draw,circle,fill] (6) at (2*360/3 +90: 2.5cm) {};
			
			\draw [line width=2pt,-] (1) -- (2);
			\draw [line width=2pt,-] (1) -- (3);
			\draw [line width=2pt,-] (1) -- (4);
			\draw [line width=2pt,-] (2) -- (3);
			\draw [line width=2pt,-] (2) -- (5);
			\draw [line width=2pt,-] (3) -- (6);
			\draw [line width=2pt,-] (4) -- (5);
			\draw [line width=2pt,-] (4) -- (6);
			\draw [line width=2pt,-] (5) -- (6);
			
			\end{tikzpicture}
			\caption{The $3$-prism, $Y_3$.}
			\label{fig:3Prism}
		\end{subfigure}
		\caption{Minimally $3$-connected planar graphs on $6$ vertices.}
		\label{fig:min3ConnPlane6Vtxs}
	\end{figure}
	
	\begin{figure}[h]
		\centering
		\begin{subfigure}{0.32\textwidth}
			\centering
			\begin{tikzpicture}[scale=0.8, every node/.style={scale=0.8}]

			\node[draw,circle,fill] (1) at (0*360/6 +120: 2cm) {};
			\node[draw,circle,fill] (2) at (1*360/6 +120: 2cm) {};
			\node[draw,circle,fill] (3) at (2*360/6 +120: 2cm) {};
			\node[draw,circle,fill] (4) at (3*360/6 +120: 2cm) {};
			\node[draw,circle,fill] (5) at (4*360/6 +120: 2cm) {};
			\node[draw,circle,fill] (6) at (5*360/6 +120: 2cm) {};
			\node[draw,circle,fill] (7) at (0,0) {};

			\draw [line width=2pt,-] (1) -- (2);
			\draw [line width=2pt,-] (1) -- (6);
			\draw [line width=2pt,-] (1) -- (7);
			\draw [line width=2pt,-] (2) -- (3);
			\draw [line width=2pt,-] (2) -- (7);
			\draw [line width=2pt,-] (3) -- (4);
			\draw [line width=2pt,-] (3) -- (7);
			\draw [line width=2pt,-] (4) -- (5);
			\draw [line width=2pt,-] (4) -- (7);
			\draw [line width=2pt,-] (5) -- (6);
			\draw [line width=2pt,-] (5) -- (7);
			\draw [line width=2pt,-] (6) -- (7);

			\end{tikzpicture}
			\caption{}
			\label{fig:Wheel6}
		\end{subfigure}
		\hfill
		\begin{subfigure}{0.32\textwidth}
			\centering
			\begin{tikzpicture}[scale=0.8, every node/.style={scale=0.8}]

			\node[draw,circle,fill] (1) at (0*360/3 +90: .75cm) {};
			\node[draw,circle,fill] (2) at (1*360/3 +90: .75cm) {};
			\node[draw,circle,fill] (3) at (2*360/3 +90: .75cm) {};
			\node[draw,circle,fill] (4) at (0*360/3 +90: 2.5cm) {};
			\node[draw,circle,fill] (5) at (1*360/3 +90: 2.5cm) {};
			\node[draw,circle,fill] (6) at (2*360/3 +90: 2.5cm) {};
			\node[draw,circle,fill] (7) at ($(2)!0.5!(3)$) {};
			
			\draw [line width=2pt,-] (1) -- (2);
			\draw [line width=2pt,-] (1) -- (3);
			\draw [line width=2pt,-] (1) -- (4);
			\draw [line width=2pt,-] (1) -- (7);
			\draw [line width=2pt,-] (2) -- (7);
			\draw [line width=2pt,-] (2) -- (5);
			\draw [line width=2pt,-] (3) -- (6);
			\draw [line width=2pt,-] (3) -- (7);
			\draw [line width=2pt,-] (4) -- (5);
			\draw [line width=2pt,-] (4) -- (6);
			\draw [line width=2pt,-] (5) -- (6);
			
			\end{tikzpicture}
			\caption{}
			\label{fig:3minConn7Vtx1}
		\end{subfigure}
		\hfill
		\begin{subfigure}{0.32\textwidth}
			\centering
			\begin{tikzpicture}[scale=0.8, every node/.style={scale=0.8}]
			
			\coordinate (a) at (0,1.5);
			\node[draw,circle,fill] (4) at ($(a) +(60+90: .75cm)$){};
			\node[draw,circle,fill] (7) at ($(a) +(-60+90: .75cm)$){};
			\node[draw,circle,fill] (1) at (0*360/3 +90: .75cm) {};
			\node[draw,circle,fill] (2) at (1*360/3 +90: .75cm) {};
			\node[draw,circle,fill] (3) at (2*360/3 +90: .75cm) {};
			\node[draw,circle,fill] (5) at (1*360/3 +90: 2.5cm) {};
			\node[draw,circle,fill] (6) at (2*360/3 +90: 2.5cm) {};
			
			\draw [line width=2pt,-] (1) -- (2);
			\draw [line width=2pt,-] (1) -- (3);
			\draw [line width=2pt,-] (1) -- (4);
			\draw [line width=2pt,-] (1) -- (7);
			\draw [line width=2pt,-] (2) -- (3);
			\draw [line width=2pt,-] (2) -- (5);
			\draw [line width=2pt,-] (3) -- (6);
			\draw [line width=2pt,-] (4) -- (5);
			\draw [line width=2pt,-] (4) -- (7);
			\draw [line width=2pt,-] (5) -- (6);
			\draw [line width=2pt,-] (6) -- (7);

			\end{tikzpicture}
			\caption{}
			\label{fig:3minConn7Vtx2}
		\end{subfigure}
		\caption{Minimally $3$-connected planar graphs on $7$ vertices.}
		\label{fig:min3ConnPlane7Vtxs}
	\end{figure}
	
	\begin{lemma}\label{lem:5wheel}
		A metrizable $2$-connected graph $G$ which contains a subdivision of $W_5$ has at most $7$ vertices. 
	\end{lemma}
	\begin{proof}
		Let $H$ be a subdivision of $W_5$ that is a subgraph of $G$, where $u$ is $H$'s pivot, and
		$v_1,v_2,v_3,v_4,v_5$ its degree $3$ vertices in cyclic order. Let $P_{i}$ be the path that is the possibly subdivided edge $v_{i}v_{i+1}$ and $Q_{i}$ the $uv_i$ path which is the possibly subdivided edge $uv_{i}$. We first prove that $G$ is not metrizable under certain conditions.
		\begin{enumerate}[(1)]
			\item $H$ has exactly two vertices of degree $2$.
			\begin{proof}
				We consider all the possible cases up to symmetries.
				\begin{itemize}
					\item $Q_{1}$ has length $3$. The graph $H - \set{uv_5,u_3}$ is a copy of \Cref{fig:graph8}.
					\item $Q_{1}$ and $Q_{2}$ have length $2$. The graph $H - \set{uv_3, uv_5}$ is a copy of \Cref{fig:graph9}.
					\item $Q_{1}$ and $Q_{3}$ have length $2$. The graph $H - \set{uv_4, uv_5}$ is a copy of \Cref{fig:graph6}.   
					\item $P_{1}$ has length $3$. The graph $H - \set{uv_3, uv_5}$ is a copy of \Cref{fig:graph6}.
					\item $P_{1}$ and $P_{2}$ have length $2$. The graph $H - \set{uv_4, uv_5}$ is a copy of \Cref{fig:graph6}.
					\item $P_{1}$ and $P_{3}$ have length $2$. The graph $H - \set{uv_3, uv_5}$ is a copy of \Cref{fig:graph6}.
					\item $P_{1}$ and $Q_{1}$ have length $2$. The graph $H - \set{uv_3, uv_5}$ is a copy of \Cref{fig:graph7}.
					\item $P_{1}$ and $Q_{2}$ have length $2$. The graph $H - \set{uv_3, uv_5}$ is a copy of \Cref{fig:graph7}.
					\item $P_{1}$ and $Q_{3}$ have length $2$. The graph $H - \set{uv_2, uv_4}$ is a copy of \Cref{fig:graph8}.
				\end{itemize}
			\end{proof}
			\item $H$ has exactly one vertex of degree $2$, $w$, and $wz \in E(G)$ for some $z\in V(G)\setminus V(H)$
			\begin{proof}
				Since $G$ is $2$-connected there exists another $z-H$ path not containing $w$. It suffices to consider the case where this path is a single edge. We go through all the possibilities up to symmetries.
				\begin{itemize}
					\item $w$ in on the path $P_{1}$
					\begin{itemize}
						\item $v_1 \in N(z)$. Redefine $P_{1}$ to be $v_1zwv_2$ a path of length $3$, and return to case $(1)$ above.
						\item $v_5 \in N(z)$. The paths $v_5zw, \ v_5v_1w, \ v_5uv_2, \ v_5v_4v_3v_2, \ v_2w$ form a subdivision of \Cref{fig:graph2}, \Cref{fig:W5NotMetFig1}.
						\item $v_4\in N(z)$. The paths $v_4zw, \ v_4v_5v_1w,\ v_4uv_2, \ v_4v_3v_2, \ wv_2$ form a subdivision of \Cref{fig:graph2}.
						\item $u \in N(z)$. The cycle $v_1wv_2v_3v_4v_5v_1$ along with the paths $uzw, \ uv_3, \  u v_5$ form a copy of \Cref{fig:graph7}, \Cref{fig:W5NotMetFig2}
					\end{itemize}
					\item $w$ in on the path $Q_{1}$
					\begin{itemize}
						\item $v_1 \in N(z)$. Redefine $Q_{1}$ to be the path of length $3$ $uwzv_1$, and return to case $(1)$ above.
						\item $u \in N(z)$. Again redefine 
						$Q_{1}$ to be the path of length $3$ $uzwv_1$.
						\item $v_2 \in N(z)$. The cycle $v_1v_2v_3v_4v_5v_1$, along with the paths $uwzv_2, \ uv_4, \ uv_5$ form a copy of \Cref{fig:graph8}, \Cref{fig:W5NotMetFig3}.
						\item $v_3 \in N(z)$.  Then the graph $(H- \set{uv_2, uv_3}) \cup \set{zw, zv_3}$ is a subdivision of \Cref{fig:graph5}, \Cref{fig:W5NotMetFig4}.
					\end{itemize}
				\end{itemize}
			\end{proof}
			\item $z\in V(G)\setminus V(H)$ and $zv_i, \ zv_j \in E(G)$ for some non-adjacent $v_i,\ v_j$, say $i=2$ and $j=5$.
			\begin{proof}
				The paths $v_5zv_2, \ P_5P_1, \ Q_5Q_3, \ P_4P_3, \ P_2$ form a subdivision of \Cref{fig:graph2}, \Cref{fig:W5NotMetFig5}.
			\end{proof}
			\item $P_1$ has length $2$ and $v_1, v_2$ are both neighbors of some $z\in V(G)\setminus V(H)$.
			\begin{proof}
				The path $P_1, \ v_1zv_2, \ Q_1Q_2, \ P_5P_4P_3P_2$ form a subdivsion of \Cref{fig:graph1}.
			\end{proof}
			\item $Q_1$ has length $2$ and $v_1, u$ are both neighbors of some $z\in V(G)\setminus V(H)$.
			\begin{proof}
				The paths $Q_1, \ uzv_1,\ P_5Q_5, \ P_1P_2P_3P_4Q_4$ form a subdivision of \Cref{fig:graph1}.
			\end{proof}
		\end{enumerate}
		\begin{figure}[h]
			\centering
			\begin{subfigure}{0.32\textwidth}
				\centering
				\begin{tikzpicture}[scale=0.85, every node/.style={scale=0.85}]
				
				\node[draw,circle,minimum size=.6cm,inner sep=0pt] (z) at (0*360/5 +90: 2.5cm){$z$};
				\node[draw,circle,minimum size=.6cm,inner sep=0pt] (v1) at (0*360/5 +90: 1.25cm) {$v_1$};
				\node[draw,circle,minimum size=.6cm,inner sep=0pt] (v2) at (-1*360/5 +90: 2.5cm) {$v_2$};
				\node[draw,circle,minimum size=.6cm,inner sep=0pt] (v3) at (-2*360/5 +90: 2.5cm) {$v_3$};
				\node[draw,circle,minimum size=.6cm,inner sep=0pt] (v4) at (-3*360/5 +90: 2.5cm) {$v_4$};
				\node[draw,circle,minimum size=.6cm,inner sep=0pt] (v5) at (-4*360/5 +90: 2.5cm) {$v_5$};
				\node[draw,circle,minimum size=.6cm,inner sep=0pt] (u) at (0,0) {$u$};
				\node[draw,circle,minimum size=.6cm,inner sep=0pt] (w) at ($(v2)!0.5!(z)$) {$w$};

				\draw [line width=2pt,-,red1] (v1) -- (w);
				\draw [line width=2pt,-,red1] (v1) -- (v5);
				\draw [line width=2pt,-] (v1) -- (u);
				\draw [line width=2pt,-,red1] (v2) -- (v3);
				\draw [line width=2pt,-,red1] (v2) -- (u);
				\draw [line width=2pt,-,red1] (v2) -- (w);
				\draw [line width=2pt,-,red1] (v3) -- (v4);
				\draw [line width=2pt,-] (v3) -- (u);
				\draw [line width=2pt,-,red1] (v4) -- (v5);
				\draw [line width=2pt,-] (v4) -- (u);
				\draw [line width=2pt,-,red1] (v5) -- (u);
				\draw [line width=2pt,-,red1] (v5) -- (z);
				\draw [line width=2pt,-,red1] (z) -- (w);

				\end{tikzpicture}
				\caption{This graph contains the non-metrizable \Cref{fig:graph2}.}
				\label{fig:W5NotMetFig1}
				
			\end{subfigure}
			\hfill
			\begin{subfigure}{0.32\textwidth}
				\centering
				\begin{tikzpicture}[scale=0.85, every node/.style={scale=0.85}]

				\node[draw,circle,minimum size=.6cm,inner sep=0pt] (v1) at (0*360/5 +90: 2.5cm) {$v_1$};
				\node[draw,circle,minimum size=.6cm,inner sep=0pt] (v2) at (-1*360/5 +90: 2.5cm) {$v_2$};
				\node[draw,circle,minimum size=.6cm,inner sep=0pt] (v3) at (-2*360/5 +90: 2.5cm) {$v_3$};
				\node[draw,circle,minimum size=.6cm,inner sep=0pt] (v4) at (-3*360/5 +90: 2.5cm) {$v_4$};
				\node[draw,circle,minimum size=.6cm,inner sep=0pt] (v5) at (-4*360/5 +90: 2.5cm) {$v_5$};
				\node[draw,circle,minimum size=.6cm,inner sep=0pt] (u) at (0,0) {$u$};
				\node[draw,circle,minimum size=.6cm,inner sep=0pt] (w) at ($(v1)!0.5!(v2)$) {$w$};
				\node[draw,circle,minimum size=.6cm,inner sep=0pt] (z) at ($(u)!0.5!(w)$) {$z$};

				\draw [line width=2pt,-,red1] (v1) -- (w);
				\draw [line width=2pt,-,red1] (v2) -- (w);
				\draw [line width=2pt,-,red1] (v1) -- (v5);
				\draw [line width=2pt,-] (v1) -- (u);
				\draw [line width=2pt,-,red1] (v2) -- (v3);
				\draw [line width=2pt,-] (v2) -- (u);
				\draw [line width=2pt,-,red1] (v3) -- (v4);
				\draw [line width=2pt,-,red1] (v3) -- (u);
				\draw [line width=2pt,-,red1] (v4) -- (v5);
				\draw [line width=2pt,-] (v4) -- (u);
				\draw [line width=2pt,-,red1] (v5) -- (u);
				\draw [line width=2pt,-,red1] (z) -- (u);
				\draw [line width=2pt,-,red1] (z) -- (w);

				\end{tikzpicture}
				\caption{The graph $W_7$ with a subdivided spoke contains \Cref{fig:graph2}.}
				\label{fig:W5NotMetFig2}
				
			\end{subfigure}
			\hfill
			\begin{subfigure}{0.32\textwidth}
				\centering
				\begin{tikzpicture}[scale=0.85, every node/.style={scale=0.85}]

				\node[draw,circle,minimum size=.6cm,inner sep=0pt] (v1) at (0*360/5 +90: 2.5cm) {$v_1$};
				\node[draw,circle,minimum size=.6cm,inner sep=0pt] (v2) at (-1*360/5 +90: 2.5cm) {$v_2$};
				\node[draw,circle,minimum size=.6cm,inner sep=0pt] (v3) at (-2*360/5 +90: 2.5cm) {$v_3$};
				\node[draw,circle,minimum size=.6cm,inner sep=0pt] (v4) at (-3*360/5 +90: 2.5cm) {$v_4$};
				\node[draw,circle,minimum size=.6cm,inner sep=0pt] (v5) at (-4*360/5 +90: 2.5cm) {$v_5$};
				\node[draw,circle,minimum size=.6cm,inner sep=0pt] (u) at (0,0) {$u$};
				\node[draw,circle,minimum size=.6cm,inner sep=0pt] (w) at ($(u)!0.5!(v1)$) {$w$};
				\node[draw,circle,minimum size=.6cm,inner sep=0pt] (z) at ($(w)!0.5!(v2)$) {$z$};

				\draw [line width=2pt,-,red1] (v1) -- (v2);
				\draw [line width=2pt,-,red1] (v1) -- (v5);
				\draw [line width=2pt,-] (v1) -- (w);
				\draw [line width=2pt,-,red1] (u) -- (w);
				\draw [line width=2pt,-,red1] (v2) -- (v3);
				\draw [line width=2pt,-] (v2) -- (u);
				\draw [line width=2pt,-,red1] (v3) -- (v4);
				\draw [line width=2pt,-] (v3) -- (u);
				\draw [line width=2pt,-,red1] (v4) -- (v5);
				\draw [line width=2pt,-,red1] (v4) -- (u);
				\draw [line width=2pt,-,red1] (v5) -- (u);
				\draw [line width=2pt,-,red1] (w) -- (z);
				\draw [line width=2pt,-,red1] (v2) -- (z);

				\end{tikzpicture}
				\caption{A graph which contains a copy of  \Cref{fig:graph8}.}
				\label{fig:W5NotMetFig3}
				
			\end{subfigure}
			\caption{}
		\end{figure}
		We show $G$ is not metrizable using the above conditions. Our analysis is split to cases according to the number of vertices in $H$. We show first that $G$ is not metrizable when $\abs{V(H)} = 8$, and derive the same conclusion  whenever $\abs{V(H)}\geq 8$, using \Cref{cor:subdiv}. When $\abs{V(H)} = 8$ there are exactly two vertices of degree $2$ in $H$ and this is dealt with in case $(1)$.
		
		Next assume $\abs{V(H)} =7$. Then there exists $w\in V(H)$, $\text{deg}_H(w) = 2$, and $z \in V(G) \setminus V(H)$. Since $G$ is $2$-connected there exist two openly disjoint $z-H$ paths. It is enough to assume that these paths are edges in $G$. If one of these edges is $zw$ then by condition $(2)$ $G$ is not metrizable.
		Suppose these edges are $zv_i, \ zv_j$. If $v_i$ and $v_j$ are not adjacent then $G$ is not metrizable by case $(3)$. If  $v_i$ and $v_j$ are adjacent with $j=i+1$, then either $P_i$ has length $2$ and  case $(4)$ applies or we can redefine $P_i$ to be $v_izv_{i+1}$ and case $(1)$ applies.
		Suppose these edges are $zv_i, \ zu$. If $Q_i$ has length $2$ then $G$ is not metrizable by case $(5)$. Otherwise we can redefine $Q_i$ to be $uzv_i$ and condition $(1)$ applies.\\
		Lastly assume $\abs{V(H)} =6$. Then there exists $z \in V(G) \setminus V(H)$. Since $G$ is $2$-connected there exists two openly disjoint $z-H$ paths, $R_1$ and $R_2$. If the ends of $R_1$ and $R_2$ are not adjacent then $G$ is not metrizable by condition $(3)$. So we can assume that ends of these paths are adjacent and we can replace one of the edges in $H$ by the path $R_1R_2$ so that $\abs{V(H)} \geq 7$. In this case $G$ was already shown to be not metrizable. 
	\end{proof}
	\begin{figure}[h]
		\centering
		\begin{subfigure}{0.28\textwidth}
			\centering
			\begin{tikzpicture}[scale=0.8, every node/.style={scale=0.8}]

			\node[draw,circle,minimum size=.6cm,inner sep=0pt] (v1) at (0*360/5 +90: 2.5cm) {$v_1$};
			\node[draw,circle,minimum size=.6cm,inner sep=0pt] (v2) at (-1*360/5 +90: 2.5cm) {$v_2$};
			\node[draw,circle,minimum size=.6cm,inner sep=0pt] (v3) at (-2*360/5 +90: 2.5cm) {$v_3$};
			\node[draw,circle,minimum size=.6cm,inner sep=0pt] (v4) at (-3*360/5 +90: 2.5cm) {$v_4$};
			\node[draw,circle,minimum size=.6cm,inner sep=0pt] (v5) at (-4*360/5 +90: 2.5cm) {$v_5$};
			\node[draw,circle,minimum size=.6cm,inner sep=0pt] (u) at (0,0) {$u$};
			
			\draw [line width=2pt,-] (v2) to [out=160,in=50, distance=1cm] (u);
			
			\node[draw,circle,minimum size=.6cm,inner sep=0pt] (w) at ($(u)!0.5!(v2)$) {$w$};

			\draw [line width=2pt,-,red1] (v1) -- (v2);
			\draw [line width=2pt,-,red1] (v1) -- (v5);
			\draw [line width=2pt,-,red1] (v1) -- (w);
			\draw [line width=2pt,-,red1] (u) -- (w);
			\draw [line width=2pt,-,red1] (v2) -- (v3);
			\draw [line width=2pt,-,red1] (v3) -- (v4);
			\draw [line width=2pt,-] (v3) -- (u);
			\draw [line width=2pt,-,red1] (v4) -- (v5);
			\draw [line width=2pt,-,red1] (v4) -- (u);
			\draw [line width=2pt,-,red1] (v5) -- (u);
			\node[draw,circle,minimum size=.6cm,inner sep=0pt] (z) at ($(w)!0.5!(v3)$) {$z$};
			\draw [line width=2pt,-,red1] (w) -- (z);
			\draw [line width=2pt,-,red1] (v3) -- (z);

			\end{tikzpicture}
			\caption{A graph which contains a subdivision of \Cref{fig:graph5}.}
			\label{fig:W5NotMetFig4}
			
		\end{subfigure}
		\hfill
		\begin{subfigure}{0.28\textwidth}
			\centering
			\begin{tikzpicture}[scale=0.8, every node/.style={scale=0.8}]

			\node[draw,circle,minimum size=.6cm,inner sep=0pt] (v1) at (0*360/5 +90: 1.25cm) {$v_1$};
			\node[draw,circle,minimum size=.6cm,inner sep=0pt] (v2) at (-1*360/5 +90: 2.5cm) {$v_2$};
			\node[draw,circle,minimum size=.6cm,inner sep=0pt] (v3) at (-2*360/5 +90: 2.5cm) {$v_3$};
			\node[draw,circle,minimum size=.6cm,inner sep=0pt] (v4) at (-3*360/5 +90: 2.5cm) {$v_4$};
			\node[draw,circle,minimum size=.6cm,inner sep=0pt] (v5) at (-4*360/5 +90: 2.5cm) {$v_5$};
			\node[draw,circle,minimum size=.6cm,inner sep=0pt] (u) at (0,0) {$u$};
			\node[draw,circle,minimum size=.6cm,inner sep=0pt] (z) at (0*360/5 +90: 2.5cm) {$z$};
			
			\draw [line width=2pt,-,red1] (v1) -- (v2);
			\draw [line width=2pt,-,red1] (v1) -- (v5);
			\draw [line width=2pt,-] (v1) -- (u);
			\draw [line width=2pt,-,red1] (v2) -- (v3);
			\draw [line width=2pt,-] (v2) -- (u);
			\draw [line width=2pt,-,red1] (v3) -- (v4);
			\draw [line width=2pt,-,red1] (v3) -- (u);
			\draw [line width=2pt,-,red1](v4) -- (v5);
			\draw [line width=2pt,-] (v4) -- (u);
			\draw [line width=2pt,-,red1] (v5) -- (u);
			\draw [line width=2pt,-,red1] (v5) -- (z);
			\draw [line width=2pt,-,red1] (v2) -- (z);

			\end{tikzpicture}
			\caption{The red edges form a subgraph isomorphic to \Cref{fig:graph2}.}
			\label{fig:W5NotMetFig5}
		\end{subfigure}
		\hfill
		\begin{subfigure}{0.28\textwidth}
			\centering
			\begin{tikzpicture}[scale=0.8, every node/.style={scale=0.8}]

			\node[draw,circle,fill] (1) at (0*360/3 +90: .75cm) {};
			\node[draw,circle,fill] (2) at (1*360/3 +90: .75cm) {};
			\node[draw,circle,fill] (3) at (2*360/3 +90: .75cm) {};
			\node[draw,circle,fill] (4) at (0*360/3 +90: 3cm) {};
			\node[draw,circle,fill] (5) at (1*360/3 +90: 3cm) {};
			\node[draw,circle,fill] (6) at (2*360/3 +90: 3cm) {};
			\node[draw,circle,fill] (7) at ($(2) ! 0.5 !(3)+ (-.5,-.75)$ ) {};
			
			\draw [line width=2pt,-,red1] (1) -- (2);
			\draw [line width=2pt,-,red1] (1) -- (3);
			\draw [line width=2pt,-,red1] (1) -- (4);
			\draw [line width=2pt,-] (2) -- (3);
			\draw [line width=2pt,-,red1] (2) -- (5);
			\draw [line width=2pt,-,red1] (3) -- (6);
			\draw [line width=2pt,-,red1] (3) -- (7);
			\draw [line width=2pt,-,red1] (4) -- (5);
			\draw [line width=2pt,-] (4) -- (6);
			\draw [line width=2pt,-,red1] (5) -- (6);
			\draw [line width=2pt,-,red1] (5) -- (7);
			
			\end{tikzpicture}
			\caption{A path which connects two non-adjacent vertices in $Y_3$ creates a subdivision of \Cref{fig:graph2}.}
			\label{fig:3PrismNotMet}
		\end{subfigure}
		\caption{}
	\end{figure}
	\begin{lemma}\label{lem:3prism}
		A $2$-connected graph which contains a subdivision of $Y_3$ and has at least $7$ vertices is not metrizable.
	\end{lemma}
	\begin{proof}
		Let $G$ be a graph with at least $7$ vertices containing a subdivision of $Y_3$.
		Subdividing any edge in $Y_3$ we either get a copy of \Cref{fig:graph4} or \Cref{fig:graph5}. So we need only consider the case where there is path of  length $2$ between two non-adjacent vertices of $Y_3.$ In this case $G$ contains a subdivision of \Cref{fig:graph2}, see \Cref{fig:3PrismNotMet}.
	\end{proof}
	We use the following lemma from \cite{Di}:
	\begin{lemma} \label{lem:3ConnEdgeRem}
		Every $3$-connected graph $G\neq K_4$ contains an edge $e$ such that the graph $G-e$ is $3$-connected after suppressing any vertices of degree $2$.
	\end{lemma}
	An immediate consequence of \Cref{lem:3ConnEdgeRem} is
	\begin{lemma}
		Every $3$-connected graph $G$ on $n>4$ vertices contains a subdivision of a $3$-connected graph $H$ on either $n-1$ or $n-2$ vertices. 
	\end{lemma}
	\begin{lemma}\label{lem:3ConnTM}
		Let $G$ be $3$-connected (planar) graph on $n>4$ vertices. Then $G$ contains a subdivision of a graph $H$, where $H$ is a $3$-connected (planar) graph which has either $n-1$ or $n-2$ vertices. 
	\end{lemma}
	We can now prove the following:
	\begin{theorem}\label{thm:2Conn3ConnMet}
		If $G$ is a $2$-connected graph of order at least $8$ which contains a subdivision of a $3$-connected graph other than $K_4, \ W_4$ and $K_5-e$, then $G$ is not metrizable.
	\end{theorem}
	
	\begin{proof}
		By \Cref{thm:NonPlamarMet} it suffices to prove the planar case. By \Cref{lem:3ConnTM} if the theorem holds for $3$-connected planar graphs on $6$ and $7$ vertices then it holds for all such graphs with at least $6$ vertices. Moreover, it suffices to consider minimally $3$-connected graphs as a subdivision of a $3$-connected graph always contains a subdivision of a minimally $3$-connected graph. It follows from \Cref{lem:list3minConnGraphs}, \Cref{lem:5wheel}, \Cref{lem:3prism} that the theorem holds for $3$-connected planar graphs on $6$ vertices. The theorem is also true for minimally $3$-connected planar graph on $7$ vertices as it is easy to see from \Cref{lem:list3minConnGraphs} that such graphs contain a subdivision of $W_5$ or $Y_3$. Lastly we note that the only $3$-connected planar graphs on less than $6$ vertices are $K_4, \ W_4$ and $K_5-e$. 
	\end{proof}
	
	\begin{corollary}\label{cor:3ConnMet}
		A $3$-connected graph with at least $8$ vertices is not metrizable.
	\end{corollary}

	\begin{figure}[h]
		\centering
		\begin{subfigure}{0.32\textwidth}
			\centering
			\begin{tikzpicture}
			
			\node[draw,circle,minimum size=.6cm,inner sep=0pt] (x1) at (0,2) {$x_1$};
			\node[draw,circle,minimum size=.6cm,inner sep=0pt] (y1) at (1.375,-2) {$y_1$};
			\node[draw,circle,minimum size=.6cm,inner sep=0pt] (y2) at (-1.375,-2) {$y_2$};
			\node[draw,circle,minimum size=.6cm,inner sep=0pt] (a2) at (-2,0) {$a_2$};
			\node[draw,circle,minimum size=.6cm,inner sep=0pt] (b2) at (-.75,0) {$b_2$};
			\node[draw,circle,minimum size=.6cm,inner sep=0pt] (a1) at (2,0) {$a_1$};
			\node[draw,circle,minimum size=.6cm,inner sep=0pt] (b1) at (.75,0) {$b_1$};
			
			\draw [line width=2pt,-] (x1) -- (a2);
			\draw [line width=2pt,-] (x1) -- (b2);
			\draw [line width=2pt,-] (y2) -- (a2);
			\draw [line width=2pt,-] (y2) -- (b2);
			
			\draw [line width=2pt,-] (x1) -- (a1);
			\draw [line width=2pt,-] (x1) -- (b1);
			\draw [line width=2pt,-] (y1) -- (a1);
			\draw [line width=2pt,-] (y1) -- (b1);
			\draw [line width=1pt,-, dashed] (y1) -- (y2) node [midway, below] {$R$};

			\end{tikzpicture}
			
			\caption{A subdivision of \Cref{fig:graph2}.}
			\label{fig:fourDisjointPaths1}
		\end{subfigure}
		\hfill
		\begin{subfigure}{0.33\textwidth}
			\centering
			\begin{tikzpicture}
			
			\node[draw,circle,minimum size=.6cm,inner sep=0pt] (x1) at (-1.5,2) {$x_1$};
			\node[draw,circle,minimum size=.6cm,inner sep=0pt] (y1) at (-1.5,-2) {$y_1$};
			\node[draw,circle,minimum size=.6cm,inner sep=0pt] (a1) at (-2.25,0) {$a_1$};
			\node[draw,circle,minimum size=.6cm,inner sep=0pt] (b1) at (-0.75,0) {$b_1$};
			\node[draw,circle,minimum size=.6cm,inner sep=0pt] (x2) at (1.5,2) {$x_2$};
			\node[draw,circle,minimum size=.6cm,inner sep=0pt] (y2) at (1.5,-2) {$y_2$};
			\node[draw,circle,minimum size=.6cm,inner sep=0pt] (a2) at (2.25,0) {$a_2$};
			\node[draw,circle,minimum size=.6cm,inner sep=0pt] (b2) at (0.75,0) {$b_2$};
			
			\draw [line width=2pt,-] (x2) -- (a2);
			\draw [line width=2pt,-] (x2) -- (b2);
			\draw [line width=2pt,-] (y2) -- (a2);
			\draw [line width=2pt,-] (y2) -- (b2);
			
			\draw [line width=2pt,-] (x1) -- (a1);
			\draw [line width=2pt,-] (x1) -- (b1);
			\draw [line width=2pt,-] (y1) -- (a1);
			\draw [line width=2pt,-] (y1) -- (b1);
			
			\draw [line width=1pt,-, dashed] (x1) -- (x2);
			\draw [line width=1pt,-, dashed] (y1) -- (y2);

			\end{tikzpicture}
			
			\caption{Disjoint paths between $x_1$ and $x_2$ and $y_1$ and $y_2$.}
			\label{fig:fourDisjointPaths2}
		\end{subfigure}
		\hfill	
		\begin{subfigure}{0.33\textwidth}
			\centering
			\begin{tikzpicture}
			
			\node[draw,circle,minimum size=.6cm,inner sep=0pt] (x1) at (-1.5,2) {$x_1$};
			\node[draw,circle,minimum size=.6cm,inner sep=0pt] (y1) at (-1.5,-2) {$y_1$};
			\node[draw,circle,minimum size=.6cm,inner sep=0pt] (a1) at (-2.15,0) {$a_1$};
			\node[draw,circle,minimum size=.6cm,inner sep=0pt] (b1) at (-0.85,0) {$b_1$};
			\node[draw,circle,minimum size=.6cm,inner sep=0pt] (x2) at (1.5,2) {$x_2$};
			\node[draw,circle,minimum size=.6cm,inner sep=0pt] (y2) at (1.5,-2) {$y_2$};
			\node[draw,circle,minimum size=.6cm,inner sep=0pt] (a2) at (2.15,0) {$a_2$};
			\node[draw,circle,minimum size=.6cm,inner sep=0pt] (b2) at (0.85,0) {$b_2$};
			\node[draw,circle,minimum size=.5cm,inner sep=0pt] (u) at (0,0) {$u$};
			
			\draw [line width=2pt,-] (x2) -- (a2);
			\draw [line width=2pt,-] (x2) -- (b2);
			\draw [line width=2pt,-] (y2) -- (a2);
			\draw [line width=2pt,-] (y2) -- (b2);
			
			\draw [line width=2pt,-] (x1) -- (a1);
			\draw [line width=2pt,-] (x1) -- (b1);
			\draw [line width=2pt,-] (y1) -- (a1);
			\draw [line width=2pt,-] (y1) -- (b1);
			
			\draw [line width=1pt,-, dashed] (x1) -- (x2);
			\draw [line width=1.5pt,-, dashed] (y1) -- (y2);
			\draw [line width=1.5pt,-, dashed] (x1) -- (u);
			\draw [line width=1.5pt,-, dashed] (u) -- (y2);

			\end{tikzpicture}
			
			\caption{This graph contains a subdivision of \Cref{fig:graph2}.}
			\label{fig:fourDisjointPaths3}
		\end{subfigure}

		\caption{}
	\end{figure}
	
	We now state and prove more results regarding the non-metrizability of certain graphs that will be useful to us later. A vertex in a graph is called {\em essential} if its degree is greater than $2$.
	
	\begin{lemma}\label{lem:fourDisjointPaths}
		Let $G=(V,E)$ be a $2$-connected graph which contains internally disjoint paths $P_1$, $Q_1$, $P_2$,  and $Q_2$, each of length at least $2$. Suppose that for each $i=1,2$, $P_i$ and $Q_i$ share endpoints $x_i, y_i \in V$  such that $x_i$ and $y_i$ separate the vertices of $P_i$ and $Q_i$ from the rest of $G$. If $G$ contains an essential vertex not contained in $P_1, \ Q_1, \ P_2$, or $Q_2$ then $G$ is not metrizable.
		
	\end{lemma}
	\begin{proof}
		It suffices to prove this when all these paths have length $2$. If some path is longer, we can suppress a vertex of degree $2$ in it and refer to \Cref{cor:subdiv}. We write $P_{i} = x_ia_iy_i$ and $Q_i = x_i b_i y_i$, and let $A=\{x_1,a_1,b_1,y_1,x_2,a_2,b_2,y_2\}$, and take $u\notin A$ a vertex of degree at least $3$. There are several cases to consider.
		\begin{itemize}
			\item $x_1 = x_2$ and $y_1 = y_2$.
			
			By the fan lemma there exist $u-A$ paths $R_1$ and $R_2$ which intersect only at $u$. As $x_1$ and $y_1$ separate $A$ from the rest of $G$ these paths must terminate at $x_1$ and $y_1$. If either $R_1$ or $R_2$ has length $2$ or more, then the graph comprised of the $x_1y_1$ paths $P_1$, $Q_1$, $P_2$ and $R_1R_2$ is a subdivision of \Cref{fig:graph1}. So we can assume these paths are the edges $ux_1$ and $uy_1$. Since $u$ has degree at least $3$ it has another neighbor $v$. We now apply to $v$ the same argument previously applied to $u$ and conclude that $v$ is a neighbour of both $x_1$ and $y_1$. But then the $x_1y_1$ paths $P_1$, $Q_1$, $P_2$  and $x_1uvy_2$ form a subdivision of \Cref{fig:graph1}.
			\item $x_1 = x_2$ and $y_1\neq y_2$.
			
			Let $R$ be a $y_1y_2$ path not containing $x_1$. Note that $R$ contains none of the vertices in $\{a_1, b_1, a_2,b_2\}$. This is because $x_1$ and $y_1$ separate  $\{a_1, b_1\}$ from $y_2$ while $x_1$ and $y_2$ separate $\{a_2, b_2\}$ from $y_1$. Therefore the paths $P_1$, $Q_1$, $P_2$ , $Q_2$ and $R$ form a subdivision of \Cref{fig:graph2}, see \Cref{fig:fourDisjointPaths1}
			\item $\{x_1, y_1\} \cap \{x_2,y_2\} = \emptyset$.
			
			By Menger's theorem there are disjoint paths between $\{x_1,y_1\}$ and $\{x_2,y_2\}$, say, from $x_1$ to $x_2$ and $y_1$ to $y_2$, see \Cref{fig:fourDisjointPaths3}. Again these paths contain none of the vertices in $\{a_1, b_1, a_2,b_2\}$. Therefore, if either the $x_1x_2$ path or the $y_1y_2$ path have length greater than $1$ we obtain a subdivision of \Cref{fig:graph10} by taking the union of these two paths along with $P_1,\ P_2, \ Q_1, \ Q_2$ . So we can assume these paths are the edges $x_1x_2$ and $y_1y_2$.
			
			Again, by the fan lemma there are two $u-A$ paths $R_1, \ R_2$  which only intersect at $u$. Since the vertices  $\{x_1,y_1,x_2,y_2\}$ separate $A$ from the rest of $G$ the endpoints of these paths must be in the set $\{x_1,y_1,x_2,y_2\}$. If the endpoint are $x_1$ and $x_2$ then replacing the edge $x_1x_2$ with the path $R_1R_2$ we get a subdivision of \Cref{fig:graph10}, as above. If the endpoints are $x_1$ and $y_1$ then the four $x_1y_1$ paths $P_1$, $Q_1$, $R_1R_2$ and $x_1x_2P_2y_2y_1$ form a subdivision of \Cref{fig:graph1}. If the endpoints are $x_1$ and $y_2$ then the paths $P_1$, $Q_1$, $R_1R_2$, $P_2$ and $y_1y_2$ form a subdivision of \Cref{fig:graph2}, see \Cref{fig:fourDisjointPaths3}. Every other case is identical to one of the previous ones.
		\end{itemize}
	\end{proof}
	
	As shown by Dirac \cite{Dir} a $2$-connected graph with minimum degree $3$ contains a subdivision of $K_4$. We need the
	following variation of this result.
	\begin{lemma}\label{lem:3MinDegree3Conn}
		Let $G=(V,E)$ be a graph of connectivity $2$ and $S \subseteq V$, $\abs{S}=2$, a separating set. If $C$ is a component of $G\setminus S$ such that that each vertex in $C$ is essential in $G$ then there is a set $V'\subseteq V(C) $ and two vertices $u,v \in (V(C) \cup S) \setminus V'$ such that  the graph $G[V'\cup \set{u,v}] + uv$ is $3$-connected and $G[V']$ is a component of the graph $G\setminus\set{u,v}.$
	\end{lemma}
	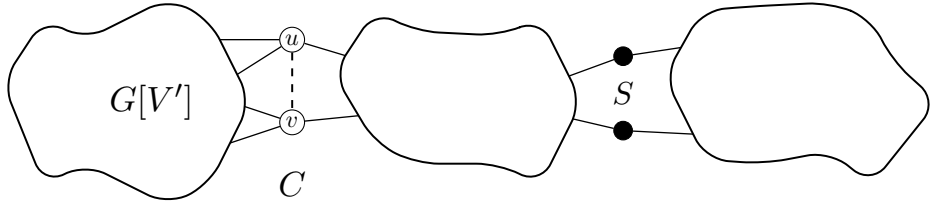
\begin{figure}[h]
		\centering
		\begin{tikzpicture}[scale=1.1, every node/.style={scale=1.1}]
		\node[scale = .75, draw,circle,minimum size=.4cm,inner sep=0pt] (u) at (3.5, .75) {$u$};
		\node[scale = .75, draw,circle,minimum size=.4cm,inner sep=0pt] (v) at (3.5, -.25) {$v$};
		\node[scale = .75, draw,circle,minimum size=.3cm,inner sep=0pt, fill] (x) at (7.5, .55) {};
		\node[scale = .75, draw,circle,minimum size=.3cm,inner sep=0pt, fill] (y) at (7.5, -.35) {};
		\draw [line width=.5pt,-] (u) -- (2,-.2);
		\draw [line width=.5pt,-] (u) -- (2,.75);
		\draw [line width=.5pt,-] (u) -- (6,0);
		\draw [line width=.5pt,-] (v) -- (2,.25);
		\draw [line width=.5pt,-] (v) -- (2,-.75);
		\draw [line width=.75pt,-,dashed] (u) -- (v);
		\draw [line width=.5pt,-] (v) -- (6,0);
		\draw [line width=.5pt,-] (x) -- (6,0);
		\draw [line width=.5pt,-] (y) -- (6,0);
		\draw [line width=.5pt,-] (x) -- (8.5,.7);
		\draw [line width=.5pt,-] (y) -- (8.5,-.4);
		\draw [line width = .75pt, scale =1,
		rounded corners=3mm, fill = white] 
		(0,.24) --(.5,1)-- (1,.75)-- (2,1.25)-- (2.5,1)-- (3, 0)-- (2.5, -1) -- (2,-1.25) -- (1,-.75) --(.5,-1) --cycle;
		
		\node[scale = 1.1] at (1.8,0) {$G[V']$};
		\node[scale = 1.1] at (3.5,-1) {$C$};
		\node[scale = 1] at ($(x)!0.5!(y)$) {$S$};

		\draw [line width = .75pt, scale =1,
		rounded corners=3mm,fill = white ]
		(4,.34) --(4.5,1.1)-- (5,.85)-- (6,.8)-- (6.5,1)-- (7, 0)-- (6.5, -1) -- (6,-.75) -- (5,-.75) --(4.5,-.5) --cycle;
		
		\draw [line width = .75pt, scale =1,
		rounded corners=3mm,fill = white ]
		(8,.3) --(8.45,1.1)-- (9,1.15)-- (10,1.2)-- (10.5,1)--(10.75,.5)-- (11.3, 0)-- (10.5, -1) -- (10,-.6) -- (9.2,-.75) --(8.5,-.7) --cycle;

		\end{tikzpicture}

		\caption{An illustration of a graph satisfying the conditions of \Cref{lem:3MinDegree3Conn}.}
	\end{figure}
	\begin{proof}
		We choose $u,v \in V(C) \cup S$ so as to minimize the size of the smallest component, $C'$, of $G\setminus\set{u,v}$ such that $V(C') \subseteq V(C)$. Such pairs of vertices exist, e.g., the two vertices of $S$. Set $V' = V(C')$ and $\tilde{V} = V' \cup \set{u,v}$. We argue that the graph $G[\tilde{V}] + uv$ is $3$-connected. First we show that $|\tilde{V}| \geq 4$. Indeed, both $u$ and $v$ have at least two neighbors in $C'$. To see this we note that since $G$ is $2$-connected both $u$ and $v$ must have at least one neighbor in $C'$ for otherwise the removal of only one vertex disconnects $G$. Suppose that $u$ has only one neighbor in $C'$, say $w$. Since $d_G(w) \geq 3$, it has at least one neighbor $z\in C'$ so that $C'- \set{w}$ is not empty. Then removing $w$ and $v$ from $G$ separates $C'- \set{w}$ from $u$. Moreover, $C'-\set{w}$ has a component that is strictly contained in $C'$, contradicting our choice of $u$ and $v$. Now suppose  $H=G[\tilde{V}] \cup  \set{uv}$ is not $3$-connected. Then there exists $x$ and $y$ whose removal disconnects $H$. If $\{x,y\} = \{u,v\}$ then $H \setminus \{x,y\} = C'$, which is a connected graph. If $\{x,y\} \cap \{u,v\} =\emptyset$ then both $u$ and $v$ are in the same connected component since they are connected by an edge. Let $C''$ be the component $H \setminus \{x,y\}$ which does not contain $u$ and $v$. Then any path from $u$ or $v$ to $C''$ must contain either $x$ or $y$. Since any path connecting a vertex in $ V\setminus \tilde{V}$ to $C'$ must contain either $u$ or $v$, it follows that $x$ and $y$ disconnects $C''$ from the rest of $G$. This implies $C''$ is a component of $G\setminus\{x,y\}$ which is strictly contained in $C'$, a contradiction. Lastly suppose that w.l.o.g. $x=u$ and $y\neq v$. Then again we have that removing $x$ and $y$ disconnects $G$. Indeed, if $C''$ is a component of $H\setminus\set{x,y}$ which does not contain $v$ then  removing $x$ and $y$ from $G$ we find there is no path connecting $C''$ to $ V\setminus \tilde{V}$. Therefore $C''$ is a component of $G \setminus \set{x,y}$ which is strictly contained in $C'$, a contradiction.
	\end{proof}
	
	Using \Cref{lem:3MinDegree3Conn} we prove this next metrizability lemma. 
	\begin{lemma}\label{lem:susPathMinDeg3}
		Let $G= (V,E)$ be a graph of connectivity $2$ containing two internally disjoint paths $P_1$ and $P_2$, each of length at least $2$, which share endpoints $x$ and $y$ such that $x$ and $y$ separate $P_1$ and $P_2$ from the rest of $G$.  If $|V \setminus (V(P_1) \cup V(P_2))|\geq 6$ and every vertex in $V \setminus (V(P_1) \cup V(P_2))$ is essential in $G$ then $G$ is not metrizable.
	\end{lemma}
	
	\begin{proof}
		This graph conforms with the assumptions of \Cref{lem:3MinDegree3Conn}. Indeed, the set $\{x,y\}$ separates $G$, and every vertex in every connected component $C$ of $G\setminus (V(P_1) \cup V(P_2))$ has degree at least $3$ in $G$. Using \Cref{lem:3MinDegree3Conn} we find some $V' \subseteq V(C)$ and vertices $u,v \in V(C) \cup \{x,y\}$ such that $H\coloneqq G[V' \cup \{u,v\}] + uv$ is $3$-connected and $G[V']$ is a component of $G \setminus \{u,v\}$.
		If $H\notin \{K_4,W_4, K_5-e\}$ then $G$ is non-metrizable by \Cref{thm:2Conn3ConnMet}. Otherwise $H\in \{K_4,W_4, K_5-e\}$ and $H$ has at most $5$ vertices. As $|V \setminus (V(P_1) \cup V(P_2))|\geq 6$, there exists at least one more vertex $u\notin P_1\cup P_2\cup H$ of degree at least $3$. Since $H$ is $3$-connected, there exists two disjoint $uv$ paths $Q_1$ and $Q_2$ of length at least $2$. These paths are also paths in $G$ since they do not contain the edge $uv$. 
		Thus $P_1$, $P_2$, $Q_1$ and $Q_2$ are internally disjoint paths of length at least $2$, where $P_1$, $P_2$ share endpoints, $Q_1$, $Q_2$ share endpoints and some vertex not in any of these paths has degree at least $3$. By \Cref{lem:fourDisjointPaths} $G$ is not metrizable .
	\end{proof}
	\begin{theorem}\label{thm:deg3Met}
		A metrizable $2$-connected graph with all vertex degrees at least $3$ has at most $12$ vertices. 
	\end{theorem}
	With some more work the upper bound on the number of vertices can be reduced from $12$ to $8$.
	\begin{proof}
		
		Let $G=(V,E)$ be a $2$-connected graph with $\delta(G) \geq 3$ and at least $13$ vertices. \Cref{cor:3ConnMet} allows us to assume that that $G$ has connectivity $2$. If $\{x,y\}$ is a cut set of $G$ and $C$ is the smallest connected component of $G \setminus\{x,y\}$, then $|V\setminus (\{x,y\} \cup V(C))| \geq 6$. Let $V' \subseteq V(C)$ and $u,v\in (V(C) \cup \{x,y\}) \setminus V'$ be as in \Cref{lem:3MinDegree3Conn}. Namely, $G[V'\cup  \{u,v\}] +uv$ is $3$-connected and $G[V']$ is a component of $G \setminus \{u,v\}$. Since $G[V' \cup \{u,v\}] +uv$ is $3$-connected there are two disjoint $uv$ paths of length at least $2$. Fix two such paths $P_1$ and $P_2$. Let $E'$ denote all the edges in $G[V'\cup  \{u,v\}] +uv$ which do not appear in $P_1$ or $P_2$.  After deleting every edge in $E'$ from $G$ we get a subgraph $G'$ which satisfies all the requirements of \Cref{lem:susPathMinDeg3}. Therefore $G'$ and hence $G$ is not metrizable.
	\end{proof}
	
	\section{Establishing Metrizability}\label{sec:est_met}
	
	Up until this point we have shown that metrizable graphs are rare. Can we find a large class graphs which {\em are} metrizable? Trivially, trees are strictly metrizable. As we show next:
	
	\begin{proposition}\label{prop:cycleStrictlyMet}
		Cycles are strictly metrizable.
	\end{proposition}
	\begin{proof}
		Let $\mathcal{P}$ be path system in $C_n$ for some $n\geq 3$. If $\mathcal{P}$ is trivial in the sense of \Cref{sec:basic}, it is strictly metrizable. As \Cref{prop:CyclePathSystem} shows $\mathcal{P}/F=\mathcal{S}_{m}$ for some odd $m\geq 3$, where $F$ is the set of persistent edges of $\mathcal{P}$. But $\mathcal{S}_{m}$ is strictly metrizable, and by \Cref{prop:StrictMetQuot} so is $\mathcal{P}$.
	\end{proof}
	
	Our next goal is to show, moreover that {\em all outerplanar graphs are strictly metrizable}. Recall that a graph is outerplanar if it can be drawn in the plane with all vertices residing in the outer face. Equivalently, a graph is outerplanar iff it contains no subdivision of $K_{2,3}$ or $K_{4}$. We start with some preparations before we can embark on the proof.
	
	A path in a graph $G$ is said to be {\em suspended} if all the vertices, except possibly the endpoints, have degree $2$ in $G$.
	\begin{theorem}\label{thm:SuspendedPath}
		If a metrizable graph $G$ has a suspended path with endpoints $x$ and $y$, then $G+xy$ is also metrizable. Similarly, if $G$ is strictly metrizable then so is $G+xy$.
	\end{theorem}
	We start with some easy lemmas.
	\begin{lemma}\label{lem:strictWeightPerturb}
		Let $\mathcal{P}$ be a path system in an $n$-vertex graph $G=(V,E)$. If $\mathcal{P}$ is strictly induced by a weight function $w:E(G) \to [0,\infty)$, then there exists $\varepsilon >0$ such that for any function $\delta : E(G) \to [0,\varepsilon]$, $w+\delta$ also strictly induces $\mathcal{P}$
	\end{lemma}
	\begin{proof}
		Since $w$ strictly induces $\mathcal{P}$ there is some $\varepsilon' > 0$ such that $w(Q)-w(P_{u,v})> \varepsilon'$ for every $u,v \in V$ and every $uv$ path $Q\neq P_{u,v}$. The claim clearly holds with $\varepsilon = \frac{\varepsilon'}{2n}$, since for $w' = w + \delta$, $w(P)\leq w'(P)< w(P) + n\varepsilon = w(P) +  \frac{\varepsilon'}{2}$ for any path $P$ in $G$. Consequently, $w'(Q) - w'(P_{u,v}) > w(Q) - w(P_{u,v}) - \frac{\varepsilon'}{2} \geq \frac{\varepsilon'}{2} > 0$ for every two vertices $u,v$ and every $uv$ path $Q\neq P_{u,v}$.
	\end{proof}
	
	\begin{lemma}\label{lem:cyclePathCharac}
		Let $f : V\to E$ be a crossing function corresponding to a path system $\mathcal{P}$ in the cycle $C_n= (V,E)$, $n\geq 3$, as in \Cref{lem:CrossingCycle}. Then $w:E \to [0,\infty)$ strictly induces $\mathcal{P}$ if and only if for every $x\in V$, $|w(P_{x,a_x}) -w(P_{x,b_x}) | < w(a_xb_x)$, where $f(x) = a_xb_x$.
	\end{lemma}
	\begin{proof}
		Suppose that $w$ strictly induces $\mathcal{P}$. Then for any $x\in V$,
		$$0 < w(P_{x,b_x} b_x a_x) - w(P_{x,a_x}) = w(P_{x,b_x}) - w(P_{x,a_x})  + w(a_xb_x)$$
		and
		$$0 < w(P_{x,a_x} a_x b_x) - w(P_{x,b_x}) = w(P_{x,a_x}) - w(P_{x,b_x})  + w(a_xb_x),$$
		implying 
		$$|w(P_{x,a_x}) -w(P_{x,b_x}) | < w(a_xb_x).$$
		Now suppose that $|w(P_{x,a_x}) -w(P_{x,b_x}) | < w(a_xb_x)$ for all $x\in V$. We need to show that $w(Q) - w(P_{u,v}) >0$, where $P_{u,v} \in \mathcal{P}$ and $Q$ the other $uv$ path in $C_n$. W.l.o.g.\ $v\in P_{u,a_u}$, where $f(u)=a_ub_u$. By definition of $f$, it must be that $P_{u,v}$ is a subpath of $P_{u,a_u}$ and that $P_{u,b_u}b_ua_u$ is a subpath of $Q$. As $|w(P_{u,a_u}) -w(P_{u,b_u}) | < w(a_ub_u)$, this implies
		$$w(Q) - w(P_{u,v}) \geq w(P_{u,b_u}b_ua_u) - w(P_{u,a_u}) = w(P_{u,b_u})- w(P_{u,a_u}) + w(a_ub_u)  >0.$$
	\end{proof}
	Let $\mathcal{P}$ a path system in a graph $G$ and let $H$ be a subgraph of $G$. We say $\mathcal{P}$ {\em restricts} to $H$ if for every two vertices $u,v\in H$ the path $P_{u,v}$ is contained in $H$. We now prove \Cref{thm:SuspendedPath}.
	\begin{proof}{[\Cref{thm:SuspendedPath}]}
		We only deal with the metrizable case, since essentially the same arguments applies to the strictly metrizable case as well.
		We can assume $xy \notin G$, otherwise there is nothing to show.
		Let $Q$ be a suspended path between $x$ and $y$ and let $C = Q + xy$ be the cycle formed by $Q$ and the edge $xy$. Let $H\coloneqq G\setminus E(Q) +xy$ be the graph obtained by removing $Q$'s edges from $G$ and adding the edge $xy$. By \Cref{prop:TopMinClosed} $H$ is metrizable, being a topological minor of $G$.

		We need to show that every path system $\mathcal{P}$ in $G + xy$ is metrizable. We first reduce the problem to the case where $\mathcal{P}$ includes every edge in $C$. If $xy\notin \mathcal{P}$ then $\mathcal{P}$ is just a path system of $G$ and is metrizable by assumption. On the other hand, if $e'\notin \mathcal{P}$ for some edge $e' \in Q$, then $\mathcal{P}$ is a path system of a graph whose biconnected components include those of $H$ and the rest of the edges of $Q$. The metrizability of $G$ then follows from \Cref{rem:2conn}.
		
		Next we claim that if $\mathcal{P}$ contains every edge in $C$, then it restricts to both $H$ and $C$. Indeed, if the path $P_{u,v}$ is not contained in $H$ for some $u,v \in H$, then it must contain vertices from $V(Q) \setminus \{x,y\}$. But $Q$ is a suspended path, and is separated from $H$ by $x$ and $y$. Therefore any path between two vertices in $H$ which meets $V(Q) \setminus \{x,y\}$ must contain all of $Q$. Since $\mathcal{P}$ is consistent, this implies $P_{x,y} = Q$, contradicting that $xy\in \mathcal{P}$. In the same way $\mathcal{P}$ restricts to a path system in $C$.
		
		Let the restrictions of $\mathcal{P}$ to $H$ and $C$ be called path systems $\mathcal{P}_H$ and $\mathcal{P}_C$, respectively.
		By \Cref{lem:CrossingCycle}, there corresponds to the path system $\mathcal{P}_C$ a crossing function $f:V(C) \to E(C)$. Namely, for each $z\in C$, $f(z) = uv$ is the unique edge such that $V(P_{z,u}) \cap V(P_{z,v}) = \{z\}.$  
		
		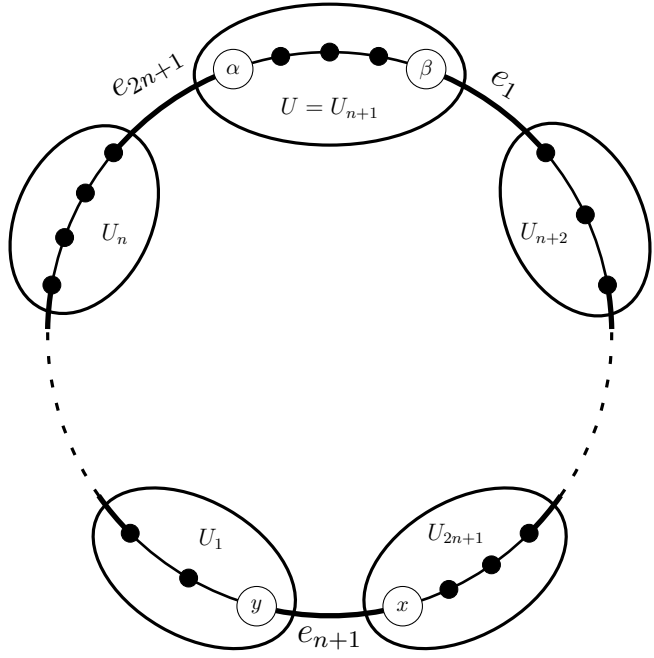
\begin{figure}[h]
		\centering
		\begin{tikzpicture}	[scale=0.75, every node/.style={scale=0.75}]
				
				\def \radius {5}
				\def \startAngle {70}
				\def \endAngle {110}
				\def \startAngleII {140}
				\def \endAngleII {170}
				\def \startAngleIII {10}
				\def \endAngleIII {40}
				\def \startAngleIV {205}
				\def \endAngleIV {225}
				\def \startAngleV {-45}
				\def \endAngleV {-25}
				\def \startAngleVI {-105}
				\def \endAngleVI {-75}

				\draw[black, line width = 1pt] ([shift={(0,0)}] \startAngle :\radius ) arc[radius=\radius , start angle= \startAngle , end angle= \endAngle];
				
				\draw[black, line width = 2pt] ([shift={(0,0)}] \endAngleIII :\radius ) arc[radius=\radius , start angle= \endAngleIII , end angle=  \startAngle];
				\draw[black, line width = 2pt] ([shift={(0,0)}]\endAngle:\radius ) arc[radius=\radius , start angle=\endAngle, end angle= \startAngleII];

				\node[draw,circle,minimum size=.7cm,inner sep=0pt,fill=white] at (\endAngle: \radius ) {$\alpha$};
				\node[draw,circle,minimum size=.7cm,inner sep=0pt,fill=white] at (\startAngle: \radius ) {$\beta$};
				\node[draw,circle,inner sep=0pt,minimum size=9pt,fill=black] at (\endAngle/4 - \startAngle/4+ \startAngle: \radius ) {};
				\node[draw,circle,inner sep=0pt,minimum size=9pt,fill=black] at  (\endAngle /4*2- \startAngle/4 *2+ \startAngle: \radius ) {};
				\node[draw,circle,inner sep=0pt,minimum size=9pt,fill=black] at  (\endAngle/4*3 - \startAngle/4*3 + \startAngle: \radius )  {};
				\node at (\endAngle/2 - \startAngle/2 + \startAngle: \radius-1 ) {\large$U=U_{n+1}$};
				\node[scale=1.5, rotate = 30] at (\startAngleII/2 - \endAngle/2 + \endAngle: \radius+.6 ) {\large$e_{2n+1}$};
				
				\draw[line width = 1.2] (\endAngle/2 - \startAngle/2 + \startAngle: \radius-.4) ellipse (2.4cm and 1.25cm);
				
				\draw[black, line width = 1pt] ([shift={(0,0)}] \startAngleIII :\radius ) arc[radius=\radius , start angle= \startAngleIII , end angle= \endAngleIII];
				
				\draw[black, line width = 1.2pt, loosely dashed] ([shift={(0,0)}]\endAngleII:\radius ) arc[radius=\radius , start angle=\endAngleII, end angle= \endAngleIV ];
				\draw[black, line width = 1.2pt, loosely dashed] ([shift={(0,0)}]\startAngleIII:\radius ) arc[radius=\radius , start angle=\startAngleIII, end angle= \startAngleV];
				
				\draw[black, line width = 2pt] ([shift={(0,0)}]\endAngleII:\radius ) arc[radius=\radius , start angle=\endAngleII, end angle= \endAngleII+9 ];
				\draw[black, line width = 2pt] ([shift={(0,0)}]\startAngleIII:\radius ) arc[radius=\radius , start angle=\startAngleIII, end angle= \startAngleIII-9];

				\node[draw,circle,inner sep=0pt,minimum size=9pt,fill=black] at (\endAngleIII: \radius ) {};
				\node[draw,circle,inner sep=0pt,minimum size=9pt,fill=black] at (\startAngleIII: \radius ) {};
				\node[draw,circle,inner sep=0pt,minimum size=9pt,fill=black] at (\endAngleIII/2 + \startAngleIII/2: \radius ) {};
				\node at (\endAngleIII/2 - \startAngleIII/2 + \startAngleIII: \radius-.8 ) {\large$U_{n+2}$};
				
				\node[scale=1.5, rotate = -30] at (\startAngle/2 - \endAngleIII/2 + \endAngleIII: \radius+.4 ) {\large$e_{1}$};
				
				\draw[line width = 1.2, rotate around={90+\endAngleIII/2 - \startAngleIII/2 + \startAngleIII:(\endAngleIII/2 - \startAngleIII/2 + \startAngleIII: \radius-.2) }] (\endAngleIII/2 - \startAngleIII/2 + \startAngleIII: \radius-.2) ellipse (1.8cm and 1.2cm);
				
				\draw[black, line width = 1pt] ([shift={(0,0)}] \startAngleII :\radius ) arc[radius=\radius , start angle= \startAngleII , end angle= \endAngleII];
				\node[draw,circle,inner sep=0pt,minimum size=9pt,fill=black] at (\startAngleII: \radius ) {};
				\node[draw,circle,inner sep=0pt,minimum size=9pt,fill=black] at (\endAngleII: \radius ) {};
				\node[draw,circle,inner sep=0pt,minimum size=9pt,fill=black] at (\endAngleII/3 - \startAngleII/3+ \startAngleII: \radius ) {};
				\node[draw,circle,inner sep=0pt,minimum size=9pt,fill=black] at ((\endAngleII/3*2 - \startAngleII/3*2+ \startAngleII: \radius ) {};
				\node at (\endAngleII/2 - \startAngleII/2 + \startAngleII: \radius-.8 ) {\large$U_{n}$};
				
				\draw[line width = 1.2, rotate around={90+\endAngleII/2 - \startAngleII/2 + \startAngleII:(\endAngleII/2 - \startAngleII/2 + \startAngleII: \radius-.2) }] (\endAngleII/2 - \startAngleII/2 + \startAngleII: \radius-.2) ellipse (1.75cm and 1.2cm);
				%
				
				\draw[black, line width = 2pt] ([shift={(0,0)}] \endAngleIV :\radius ) arc[radius=\radius , start angle= \endAngleIV , end angle=  \endAngleIV-10];
				\draw[black, line width = 2pt] ([shift={(0,0)}]\startAngleV:\radius ) arc[radius=\radius , start angle=\startAngleV, end angle= \startAngleV + 10];

				\draw[black, line width = 1pt] ([shift={(0,0)}] \startAngleV :\radius ) arc[radius=\radius , start angle= \startAngleV , end angle= \endAngleVI];
				\draw[black, line width = 1pt] ([shift={(0,0)}] \startAngleVI :\radius ) arc[radius=\radius , start angle= \startAngleVI , end angle= \endAngleIV-360];
				
				
				\draw[black, line width = 2pt] ([shift={(0,0)}] \startAngleVI :\radius ) arc[radius=\radius , start angle= \startAngleVI , end angle=  \endAngleVI];

				\node[draw,circle,inner sep=0pt,minimum size=9pt,fill=black] at (\endAngleVI/3*2 - \startAngleV/3*2+ \startAngleV: \radius ) {};
				\node[draw,circle,inner sep=0pt,minimum size=9pt,fill=black] at (\endAngleVI/3 - \startAngleV/3+ \startAngleV: \radius ) {};
				\node[draw,circle,inner sep=0pt,minimum size=9pt,fill=black] at (\endAngleIV/2 - \startAngleVI/2+ \startAngleVI + 180: \radius ) {};
				
				\node at (\endAngleVI/2 - \startAngleV/2+ \startAngleV +2.25: \radius-.8 ) {\large$U_{2n+1}$};
				\node at (\endAngleIV/2 - \startAngleVI/2+ \startAngleVI + 180: \radius-.8 ) {\large$U_{1}$};
				
				\draw[line width = 1.2, rotate around={90+\endAngleVI/2 - \startAngleV/2+ \startAngleV:(\endAngleVI/2 - \startAngleV/2+ \startAngleV: \radius-.2) }] (\endAngleVI/2 - \startAngleV/2+ \startAngleV: \radius-.2) ellipse (1.95cm and 1.2cm);
				\draw[line width = 1.2, rotate around={90+\endAngleIV/2 - \startAngleVI/2+ \startAngleVI + 180:(\endAngleIV/2 - \startAngleVI/2+ \startAngleVI + 180: \radius-.2) }] (\endAngleIV/2 - \startAngleVI/2+ \startAngleVI + 180: \radius-.2) ellipse (1.95cm and 1.2cm);
				\node[scale=1.5] at (-90 :\radius+.4 ) {\large$e_{n+1}$};
				\node[draw,circle,inner sep=0pt,minimum size=9pt,fill=black]  at (\startAngleV: \radius ) {};
				\node[draw,circle,inner sep=0pt,minimum size=9pt,fill=black] at (\endAngleIV: \radius ) {};
				\node[draw,circle,minimum size=.7cm,inner sep=0pt,fill=white] at (\endAngleVI: \radius ) {$x$};
				\node[draw,circle,minimum size=.7cm,inner sep=0pt,fill=white] at (\startAngleVI: \radius ) {$y$};

				\end{tikzpicture}
				
				\caption{The path system $\mathcal{P}_C$ partitions $C$.}
				\label{fig:SusPathFig1}
		\end{figure}
		
		We take a closer look at the path system $\mathcal{P}_C$. In \Cref{prop:CyclePathSystem}  we saw that by contracting each set $f^{-1}(e)$, $e\in \text{Im}f$, to a vertex we obtain an odd cycle equipped with its canonical path system. It follows that for some $n\geq 1$, $\abs{\text{Im}f} = 2n+1$ and that there exists an ordering of the edges $\text{Im}f = \{e_1,\cdots,e_{2n+1} \}$ so that $U_i = C[f^{-1} (e_i)]$ satisfy the following, for $1\leq i \leq 2n+1$,
		\begin{itemize}
			\item $U_i$ is a subpath of $C$
			\item $U_i$ and $U_{i+1}$ are connected by the edge $e_{n+i+1}$, with indices taken mod $2n+1$, and $e_{n+1} = xy$, with $y\in U_{1}$ and $x\in U_{2n+1}$.
			\item  For $1\leq i \leq 2n+1$, $f(u) = e_i$ for all $u \in U_i$
		\end{itemize}
		This means we can express $C$ as $C = yU_{1}U_2\cdots U_{2n}U_{2n+1}xy$, \Cref{fig:SusPathFig1}.
		Let $\alpha$ and $\beta$ denote the end vertices of the path $U_{n+1}$ so that $e_1$ is incident with $\beta$ and $e_{2n+1}$ is incident with $\alpha$. Also, we set $U\coloneqq U_{n+1}$.
		
		  Let $\tilde{w} : E(G+xy) \to [0,\infty)$ any non-negative weight function. If a $uv$ geodesic contains a vertex $z$, then clearly if  $R_1$ is a $uz$ geodesic and $R_2$ is a $zv$ geodesic then the path $R_1R_2$ is also a $uv$ geodesic. We make the following simple observation:
		\begin{idea}[$\ast_1$]
			\label{idea2}
			Suppose that $P_{u,x}$, $P_{u,y}$, $P_{x,v}$ and $P_{y,v}$ are $\tilde{w}$-geodesics, where $u\in C, v\in H$. If, w.r.t.\ $\tilde{w}$, some $uv$ geodesic contains $x$, resp. $y$,  then $P_{u,x}P_{x,v}$, resp.\ $P_{u,y}P_{y,v}$, is a $\tilde{w}$-geodesic. Consequently, either $P_{u,x}P_{x,v}$ or $P_{u,y}P_{y,v}$ is a $uv$ geodesic w.r.t.\ $\tilde{w}$.
		\end{idea}
		The last part follows from the fact that $x$ and $y$ separates $C$ from $H$ so that every $uv$ path contains either $x$ or $y$. Next we prove an important claim which roughly says that a weight function which induces both $\mathcal{P}_C$ and $\mathcal{P}_H$ {\em almost} induces $\mathcal{P}$.
		\begin{idea}[$\ast_2$]
			\label{idea1}
			Suppose that $\tilde{w}$ induces both $\mathcal{P}_C$ and $\mathcal{P}_H$. Then every path in $\mathcal{P}$ is a $\tilde{w}$-geodesic, with the possible exception of those $uv$ paths
			where $v\in H$ and $u\in U$, (whence $f(u)= xy$). 
		\end{idea}
	
		Paths in $\mathcal{P}$ with both endpoints in $C$ or both endpoints in $H$ are clearly $\tilde{w}$-geodetic, since $\tilde{w}$ induces both $\mathcal{P}_C$ and $\mathcal{P}_H$. 
		An exceptional path must connect some $u\in C$ to some $v\in H$.  Let $P_{u,v}$ be such an exceptional path. Since the paths $P_{u,x}$ and $P_{u,y}$ are in $\mathcal{P}_C$ these paths are  $\tilde{w}$-geodesics. Similarly, $P_{x,v}$ and  $P_{y,v}$ are path in $\mathcal{P}_H$ and therefore also $\tilde{w}$-geodesics. As $x$ and $y$ separate $C$ from $H$ we can assume w.l.o.g. $y\in P_{u,v}$. By consistency we get $P_{u,v} = P_{u,y}P_{y,v}$. Since $P_{u,v}= P_{u,y}P_{y,v}$ is not $\tilde{w}$-geodetic, by \eqref{idea2} it follows that $P_{u,x}P_{x,v}$ is a $\tilde{w}$-geodesic. Note that this implies $x\notin P_{u,v}$, for otherwise by consistency $P_{u,v} = P_{u,x}P_{x,v}$, which contradicts that $P_{u,v}$ is not a geodesic. Similarly, $y\notin P_{u,x}P_{x,v}$ or else $P_{u,y}P_{y,v}=P_{u,v}$ is a $\tilde{w}$-geodesic by $\eqref{idea2}$.  It follows that $y\notin P_{u,x}$ and $x \notin P_{u,y}$. Since these are two paths in $C$ it must be that $V(P_{u,y})\cap V(P_{u,x}) = \{u\}$, which, by definition of $f$, yields $f(u) = xy$.

		Since both $C$ and $H$ are metrizable, there exist weight functions $w_C:E(C)\to (0,\infty)$ and $w_H: E(H) \to (0,\infty)$ which induce $\mathcal{P}_C$ and $\mathcal{P}_H$, respectively. Clearly $E(H)\cap E(C) = \{xy\}$, and we rescale these weight functions if necessary to guarantee that $w_C(xy) = w_H(xy)$. With this normalization there is a uniquely defined $\tilde{w}:E(G + xy) \to (0,\infty)$ such that $\tilde{w}|_{E(C)} = w_C$ and $\tilde{w}|_{E(H)} = w_H$. If $f^{-1}(xy) = \emptyset$ then by \eqref{idea1} this weight function actually induces $\mathcal{P}$.
		Therefore, we can and will henceforth assume $f^{-1}(xy) \neq \emptyset$.
		
		To construct our desired weight function we first find a function which induces those paths in $\mathcal{P}$ which have endpoints in $H$ and $U$. After acquiring such a function, we see from \eqref{idea1} that it is then enough to focus all our attention more locally on $C$ and to adjust this weight function so that it also induces the path system $\mathcal{P}_C$.
		More concretely, we construct in three steps a weight function $w:E(G+xy) \to (0,\infty)$ that induces $\mathcal{P}$.
		\begin{enumerate}
			\item Find a weight function $w_1: E(G+xy) \to [0,\infty)$ that induces $\mathcal{P}_H$, as well as every $P_{u,v}\in \mathcal{P}$ with $u\in U$ and $v\in H$.
			\item Modify $w_1$ to a weight function $w_2: E(G+xy) \to [0,\infty)$ that induces $\mathcal{P}_C$ as well.
			\item Alter $w_2$ into a strictly positive $w: E(G+xy) \to (0,\infty)$ while maintaining the above properties.
		\end{enumerate}
		The resulting positive weight function $w$ induces both $\mathcal{P}_C$ and $\mathcal{P}_H$ and in addition every path $P_{u,v}\in \mathcal{P}$ with  $u\in U$ and $v\in H$ is $w$-geodetic. By \eqref{idea1} this weight function induces $\mathcal{P}$.

		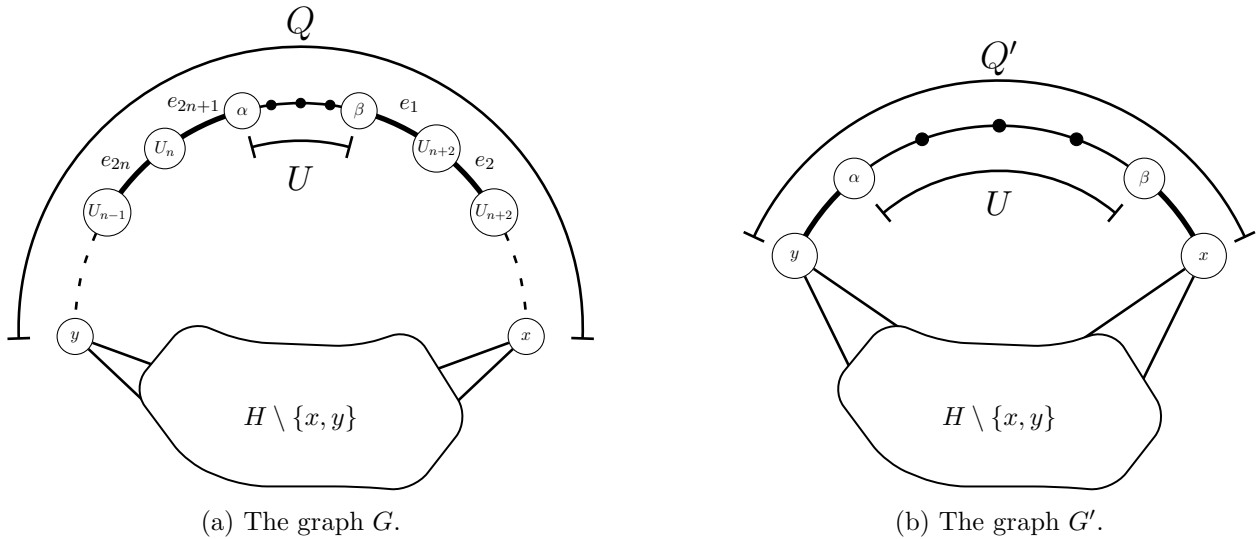
\begin{figure}[h]
			\centering
			\begin{subfigure}[t]{0.47\textwidth}
				\centering
				
				\begin{tikzpicture}	[scale=0.6, every node/.style={scale=0.6}]
				\def \inc {22}
				\def \radius {5}
				\def \rightAngle {75}
				\def \leftAngle {105}
				\pgfmathsetmacro\rightAngleII {\rightAngle- \inc}
				\pgfmathsetmacro \leftAngleII {\leftAngle + \inc}
				\pgfmathsetmacro\rightAngleIIHalf {\rightAngle- .8*\inc}
				\pgfmathsetmacro \leftAngleIIHalf {\leftAngle + .8*\inc}
				\pgfmathsetmacro\rightAngleIII {\rightAngle - 2 * \inc}
				\pgfmathsetmacro\leftAngleIII {\leftAngle + 2*\inc}
				\pgfmathsetmacro\rightAngleIV {\rightAngle - 3.5 * \inc}
				\pgfmathsetmacro\leftAngleIV {\leftAngle + 3.5*\inc}
				
				\pgfmathsetmacro\downRight {-90+ \inc/2}
				\pgfmathsetmacro\downLeft {-90- \inc/2}
				\pgfmathsetmacro\downRightII {-90+ \inc/2+ \inc}
				\pgfmathsetmacro\downLeftII {-90- \inc/2- \inc}
				\pgfmathsetmacro\downRightIII {-90+ \inc/2+ 2*\inc}
				\pgfmathsetmacro\downLeftIII {-90- \inc/2- 2*\inc}

				\draw[black, line width = 1pt] ([shift={(0,0)}] \rightAngleIV:\radius + \radius/4 ) arc[radius=\radius + \radius/4 , start angle= \rightAngleIV , end angle= \leftAngleIV];
				\draw[black, line width = 1pt] (\rightAngleIV :\radius + \radius/4 - \radius/20) -- (\rightAngleIV :\radius + \radius/4 +\radius/20);
				\draw[black, line width = 1pt] (\leftAngleIV :\radius + \radius/4 - \radius/20) -- (\leftAngleIV :\radius + \radius/4 +\radius/20);
				\node[scale =2] at (90: \radius+ \radius/4 +\radius/10 ) {\large$Q$};
				
				\draw[black, line width = 1pt] ([shift={(0,0)}] \rightAngle :\radius ) arc[radius=\radius , start angle= \rightAngle , end angle= \leftAngle];
				
				\draw[black, line width = 2pt] ([shift={(0,0)}] \leftAngle :\radius ) arc[radius=\radius , start angle= \leftAngle , end angle=  \leftAngleII];
				\draw[black, line width = 2pt] ([shift={(0,0)}]\rightAngleII:\radius ) arc[radius=\radius , start angle=\rightAngleII, end angle= \rightAngle];

				\draw[black, line width = 1pt] ([shift={(0,0)}] \rightAngle:\radius - \radius/6 ) arc[radius=\radius - \radius/6 , start angle= \rightAngle , end angle= \leftAngle];
				\draw[black, line width = 1pt] (\rightAngle :\radius - \radius/6 - \radius/20) -- (\rightAngle :\radius - \radius/6 +\radius/20);
				\draw[black, line width = 1pt] (\leftAngle :\radius - \radius/6 - \radius/20) -- (\leftAngle :\radius - \radius/6 +\radius/20);
				
				\node[scale =2] at (90: \radius - \radius/3 ) {$U$};
				
				\node[draw,circle,minimum size=.8cm,inner sep=0pt,fill=white] at (\leftAngle: \radius ) {$\alpha$};
				\node[draw,circle,minimum size=.8cm,inner sep=0pt,fill=white] at (\rightAngle: \radius ){$\beta$};
				\node[draw,circle,inner sep=0pt,minimum size=6pt,fill=black] at (\leftAngle/4 - \rightAngle/4+ \rightAngle: \radius ) {};
				\node[draw,circle,inner sep=0pt,minimum size=6pt,fill=black] at  (\leftAngle /4*2- \rightAngle/4 *2+ \rightAngle: \radius ) {};
				\node[draw,circle,inner sep=0pt,minimum size=6pt,fill=black] at  (\leftAngle/4*3 - \rightAngle/4*3 + \rightAngle: \radius )  {};

				\draw[black, line width = 2pt] ([shift={(0,0)}] \leftAngleII :\radius ) arc[radius=\radius , start angle= \leftAngleII , end angle=  \leftAngleIII];
				\draw[black, line width = 2pt] ([shift={(0,0)}]\rightAngleIII:\radius ) arc[radius=\radius , start angle=\rightAngleIII, end angle= \rightAngleII];
				
				\draw[black, line width = 1pt, loosely dashed] ([shift={(0,0)}] \leftAngleIII :\radius ) arc[radius=\radius , start angle= \leftAngleIII , end angle= \leftAngleIV];
				\draw[black, line width = 1pt, loosely dashed] ([shift={(0,0)}] \rightAngleIII :\radius ) arc[radius=\radius , start angle= \rightAngleIII , end angle= \rightAngleIV];
				
				\node[scale =1.3] at ({\rightAngle/2 - \rightAngleII/2 + \rightAngleII}: \radius + \radius/10 ) {$e_1$};
				\node[scale =1.3] at ({\leftAngleII/2-\leftAngle/2+\leftAngle -.5}: \radius +\radius/10 ) {$e_{2n+1}$};
				\node[scale =1.3] at ({\rightAngleII/2 - \rightAngleIII/2 + \rightAngleIII}: \radius + \radius/10 ) {$e_2$};
				\node[scale =1.3] at ({\leftAngleIII/2-\leftAngleII/2+\leftAngleII }: \radius +\radius/10 ) {$e_{2n}$};

				\node[draw,circle,minimum size=.9cm,inner sep=1.5pt,fill=white] at (\leftAngleII: \radius ) [scale =1]{$U_{n}$};
				\node[draw,circle,minimum size=.9cm,inner sep=1.5pt,fill=white] at (\rightAngleII: \radius )[scale =1] {$U_{n+2}$};
				
				\node[draw,circle,minimum size=.9cm,inner sep=1.5pt,fill=white] at (\leftAngleIII: \radius )[scale =1] {$U_{n-1}$};
				\node[draw,circle,minimum size=.9cm,inner sep=1.5pt,fill=white] at (\rightAngleIII: \radius )[scale =1] {$U_{n+2}$};

				\node[draw,circle,minimum size=.8cm,inner sep=0pt,fill=white] (y)at (\rightAngleIV: \radius ) {$x$};

				\node[draw,circle,minimum size=.8cm,inner sep=0pt,fill=white] (x) at (\leftAngleIV: \radius ) {$y$};
				
				\draw[black, line width = 1pt, -]  (x) --  (0,-2);
				\draw[black, line width = 1pt, -]  (x) --  (-2,-3);
				\draw[black, line width = 1pt, -]  (y) --  (0,-2);
				\draw[black, line width = 1pt, -]  (y) --  (2,-3);
				
				\begin{scope}[shift={(-13.75,-2)}, xscale = 2.5, yscale = 2 ]
				\draw [line width = .75pt,
				rounded corners=3mm,fill = white]
				(4,.34) --(4.5,1.1)-- (5,.85)-- (6,.8)-- (6.5,1)-- (7, 0)-- (6.5, -.8) -- (6,-.75) -- (5,-.75) --(4.5,-.5) --cycle;
				\end{scope}
				
				\node[scale = 1.5] at (0,-2) {$H\setminus\{x,y\}$};
				\end{tikzpicture}
				
				\caption{The graph $G$.}
				\label{fig:SusPathFig3}
				
			\end{subfigure}
			\hfill
			\begin{subfigure}[t]{.47\textwidth}
				\centering
				\begin{tikzpicture}	[scale=0.6, every node/.style={scale=0.6}]
				\def \inc {25}
				\def \radius {5}
				\def \rightAngle {75}
				\def \leftAngle {105}
				\pgfmathsetmacro\rightAngleII {\rightAngle- \inc}
				\pgfmathsetmacro \leftAngleII {\leftAngle + \inc}
				\pgfmathsetmacro\rightAngleIII {\rightAngle - 2 * \inc}
				\pgfmathsetmacro\leftAngleIII {\leftAngle + 2*\inc}
				\pgfmathsetmacro\rightAngleIV {\leftAngleIII- 180}
				
				\pgfmathsetmacro\downRight {-90+ \inc/2}
				\pgfmathsetmacro\downLeft {-90- \inc/2}
				\pgfmathsetmacro\downRightII {-90+ \inc/2+ \inc}
				\pgfmathsetmacro\downLeftII {-90- \inc/2- \inc}
				\pgfmathsetmacro\downRightIII {-90+ \inc/2+ 2*\inc}
				\pgfmathsetmacro\downLeftIII {-90- \inc/2- 2*\inc}

				\draw[black, line width = 1pt] ([shift={(0,0)}] \rightAngleIII :\radius + \radius/5 ) arc[radius=\radius + \radius/5 , start angle= \rightAngleIII , end angle= \leftAngleIII];
				\draw[black, line width = 1pt] (\rightAngleIII :\radius + \radius/5 - \radius/20) -- (\rightAngleIII :\radius + \radius/5 +\radius/20);
				\draw[black, line width = 1pt] (\leftAngleIII :\radius + \radius/5 - \radius/20) -- (\leftAngleIII :\radius + \radius/5 +\radius/20);
				\node[scale =2] at (90: \radius+ \radius/5 +\radius/10 ) {\large$Q'$};

				\draw[black, line width = 1pt] ([shift={(0,0)}] \rightAngleII :\radius ) arc[radius=\radius , start angle= \rightAngleII , end angle= \leftAngleII];
				
				\draw[black, line width = 2pt] ([shift={(0,0)}] \leftAngleII :\radius ) arc[radius=\radius , start angle= \leftAngleII , end angle=  \leftAngleIII];
				\draw[black, line width = 2pt] ([shift={(0,0)}]\rightAngleII:\radius ) arc[radius=\radius , start angle=\rightAngleII, end angle= \rightAngleIII];

				\draw[black, line width = 1pt] ([shift={(0,0)}] \rightAngleII :\radius - \radius/5 ) arc[radius=\radius - \radius/5 , start angle= \rightAngleII , end angle= \leftAngleII];
				\draw[black, line width = 1pt] (\rightAngleII :\radius - \radius/5 - \radius/20) -- (\rightAngleII :\radius - \radius/5 +\radius/20);
				\draw[black, line width = 1pt] (\leftAngleII :\radius - \radius/5 - \radius/20) -- (\leftAngleII :\radius - \radius/5 +\radius/20);
				\node[scale =2] at (90: \radius - \radius/5 -\radius/7 ) {$U$};

				\node[draw,circle,minimum size=.9cm,inner sep=0pt,fill=white] at (\leftAngleII: \radius ) {$\alpha$};
				\node[draw,circle,minimum size=.9cm,inner sep=0pt,fill=white] at (\rightAngleII: \radius ) {$\beta$};
				\node[draw,circle,inner sep=0pt,minimum size=8pt,fill=black] at (\leftAngleII/4 - \rightAngleII/4+ \rightAngleII: \radius ) {};
				\node[draw,circle,inner sep=0pt,minimum size=8pt,fill=black] at  (\leftAngleII /4*2- \rightAngleII/4 *2+ \rightAngleII: \radius ) {};
				\node[draw,circle,inner sep=0pt,minimum size=8pt,fill=black] at  (\leftAngleII/4*3 - \rightAngleII/4*3 + \rightAngleII: \radius )  {};

				\node[draw,circle,minimum size=1cm,inner sep=0pt,fill=white] (x)at (\leftAngleIII: \radius ) {$y$};
				\node[draw,circle,minimum size=1cm,inner sep=0pt,fill=white] (y) at (\rightAngleIII: \radius ) {$x$};

				\draw[black, line width = 1pt, -]  (x) --  (0,-1);
				\draw[black, line width = 1pt, -]  (x) --  (-3,-1);
				\draw[black, line width = 1pt, -]  (y) --  (0,-1);
				\draw[black, line width = 1pt, -]  (y) --  (3,-1);
				
				\begin{scope}[shift={(-13.75,-1.5)}, xscale = 2.5, yscale = 2 ]
				\draw [line width = .75pt,
				rounded corners=3mm,fill = white]
				(4,.34) --(4.5,1.1)-- (5,.85)-- (6,.8)-- (6.5,1)-- (7, 0)-- (6.5, -.8) -- (6,-.75) -- (5,-.75) --(4.5,-.5) --cycle;
				\end{scope}
				
				\node[scale = 1.5] at (0,-1.5) {$H\setminus\{x,y\}$};
				
				\end{tikzpicture}

				\caption{The graph $G'$.}
				\label{fig:SusPathFig4}
			\end{subfigure}
			
			\caption{The graph $G$ vs.\ $G'$.}
			\label{fig:SusPathFig5}
		\end{figure}
		Before getting into the construction of $w_1$, we make the following simple observation concerning paths between $U$ and $H$
		\begin{idea}[$\ast_3 $]
		\label{idea3}
		If $u\in U$ and $v\in H$ then exactly one of the vertices $x$ and $y$ is in $P_{u,v}$.   
		\end{idea}
		
		It is clear that such a path $P_{u,v}$ contains at least one of $x$ or $y$ since these vertices separate $U$ from $H$. We argue that if $P \in \mathcal{P}$ is a path containing some $u\in U$ then it cannot also contain both $x$ and $y$. Otherwise, by the consistency of $\mathcal{P}$, the $xy$ subpath of $P$ is the edge $xy$, while the $ux$ and $uy$ subpaths of $P$ are $P_{u,x}$ and $P_{u,y}$, respectively. Therefore one of the paths $P_{u,x}$ or $P_{u,y}$ contains the edge $xy$, which contradicts that $u\in U$, with $f(u) =xy$. 
    	 
From \eqref{idea3} we see that the paths $P_{u,v}$, with $u\in U$ and $v\in H$, actually form a partial path system in $G$, since none of these paths use the edge $xy$. This may suggest that we extend this partial system to a full system in $G$ in a way that mirrors $\mathcal{P}$. Our approach is similar, but rather than working with a path system in $G$ we construct a path system in $G'$, a topological minor of $G$. 
    	
We wish for $w_1$ to induce $\mathcal{P}_H$, as well as the paths $P_{u,v}\in \mathcal{P}$ with $u\in U$ and $v\in H$. But since the other vertices in $Q$ do not concern us at the moment, we take $G'$ to be the graph obtained from $G$ by suppressing the vertices in $\big(\bigcup_{i=1, i\neq n+1}^{2n+1}U_{i} \big)\setminus \{x,y\} $, \Cref{fig:SusPathFig5}. Equivalently, $G'$ is obtained from $G$ by replacing the path $Q$ with the path $Q' \coloneqq y\alpha U \beta x$.  
    	 
To obtain $w_1$ we equip $G'$ with a path system $\mathcal{P}'$. Since $G'$ is metrizable (\Cref{prop:TopMinClosed}), there exists some $w':E(G')\to (0,\infty)$ which induces $\mathcal{P}'$. We then extend $w'$ to a function in $G+xy$ to get our desired $w_1$.
		
The path system $\mathcal{P}'$ in $G'$, that we seek to define should be analogous to $\mathcal{P}$. The analog in $G'$ of $P_{u,v}$ for $u \in U$ and $v\in H$ naturally suggests itself. Namely, the path obtained from $P_{u,v}$ by suppressing the vertices in $\big(\bigcup_{i=1, i\neq n+1}^{2n+1}U_{i} \big)\setminus \{x,y\}$. For $u,v \in U$ the path $P_{u,v}$ is just a subpath of $U$ and is therefore a path in $G'$. For $u,v \in H$ then there are two cases to consider. If $xy\notin P_{u,v}$ then $P_{u,v}$ is also a path in $G'$. If $xy\in P_{u,v}$ the $G'$-analog of $P_{u,v}$ is the path obtained from $P_{u,v}$ by replacing the edge $xy$ by $Q'$. In particular, whereas the $xy$ geodesic in $\mathcal{P}$ is the edge $xy$, in $\mathcal{P}'$ this geodesic is the path $Q'$. 
Thus in constructing $\mathcal{P}'$ we take every path in $\mathcal{P}$ with endpoints in $G'$, suppress all vertices in $\big(\bigcup_{i=1, i\neq n+1}^{2n+1}U_{i} \big)\setminus \{x,y\}$ and replace every occurrence of the edge $xy$ with the path $Q'$. The formal definition of  $\mathcal{P}'$ follows: 
		\begin{itemize}
			\setlength{\itemindent}{2em}
			\item[Case $1.$] $u,v \in G' \setminus U$.
			\begin{enumerate}[$(i)$]
				\item If $xy \in P_{u,v}$, i.e., $P_{u,v} = P_{u,x}xyP_{y,v}$, we set $P_{u,v}' \coloneqq P_{u,x}Q'P_{y,v}$. 
				\item If $xy \notin P_{u,v}$ then $P_{u,v}' \coloneqq P_{u,v}$
			\end{enumerate}
			\item[Case $2.$] $u,v \in  U$. We set $P_{u,v}' \coloneqq P_{u,v}$.
			\item[Case $3.$] $u\in U$ and  $v\in  G' \setminus U$.
			\begin{enumerate}[$(i)$]
				\item If $x\in P_{u,v}$ then $P_{u,v}' \coloneqq  P_{u,\beta}\beta x P_{x,v}$.
				\item If $y\in P_{u,v}$ then $P_{u,v}' \coloneqq  P_{u,\alpha}\alpha yP_{y,v}$.
			\end{enumerate}
		\end{itemize}
We claim that this definition is valid. We need only elaborate on paths from case $3$, whose definition is valid by \eqref{idea3}. Namely, $x\in P_{u,v}$ iff $y \notin P_{u,v}$ for all $u\in U_{n+1}$ and  $v\in  G' \setminus U$. 
		
We now prove that $\mathcal{P}'$ is consistent, i.e., that if $a,b\in P_{u,v}'$ for some path $P_{u,v}' \in \mathcal{P}'$, then $P_{a,b}'$ is a subpath of $P_{u,v}'$. Our analysis follows the various cases in the definition of $\mathcal{P}'$. In most cases this property is a simple consequence of the definition of $\mathcal{P}'$ and the consistency of $\mathcal{P}$. We only need to elaborate on case $1(i)$ where $u,v \in G' \setminus U$ and $P_{u,v}$ traverses the edge $xy$, namely $P_{u,v} =P_{u,x}xyP_{y,v}$, whereas by construction, $P_{u,v}' =P_{u,x}Q'P_{y,v}$.
Let $a,b\in P_{u,v}'$. If either $a,b \in G' \setminus U$ or $a,b \in U$, then the definition of $\mathcal{P}'$ and the consistency of $\mathcal{P}$ readily imply that $P_{a,b}'$ is subpath of $P_{u,v}'$. So let us assume that $a\in U$ and $b\in G' \setminus U$.
W.l.o.g.\ $b \in P_{u,x}$ so that $P_{u,v} = P_{u,b}P_{b,x}xyP_{y,v}$ and $P_{u,v}' = P_{u,b}P_{b,x}Q'P_{y,v}$. By construction of $\mathcal{P}'$, the path $P_{a, b}'$ is either $P_{a,\beta}\beta xP_{x,b}$ or $P_{a,\alpha}\alpha yP_{y,b}$.  We claim that $P_{a,b}'$ is in fact $P_{a,\beta}\beta xP_{x,b}$, which is the $ab$ subpath of $P_{u,v}'$. Otherwise $P_{a,b}' =P_{a,\alpha}\alpha yP_{y,b}$ which, by the construction of $\mathcal{P}',$ implies $y\in P_{a,b}$ and $P_{a,b} = P_{a,y}P_{y,b}$. As $P_{u,v} = P_{u,b}P_{b,x}xyP_{y,v}$ and $\mathcal{P}$ is consistent, the $yb$ subpath $yxP_{x,b}$ of $P_{u,v}$ coincides with the path $P_{y,b},$ \Cref{fig:SusPathFig6}. It follows that  $P_{a,b} = P_{a,y}P_{y,b} = P_{a,y}yxP_{x,b}$, and the path $P_{a,b}$ contains the edge $xy$. From \eqref{idea3} we see this contradicts the fact that $a\in U$, with $f(a) = xy$.
		It follows that $\mathcal{P}'$ is consistent and there exists some $w':E(G') \to (0,\infty)$ which induces it.
		
		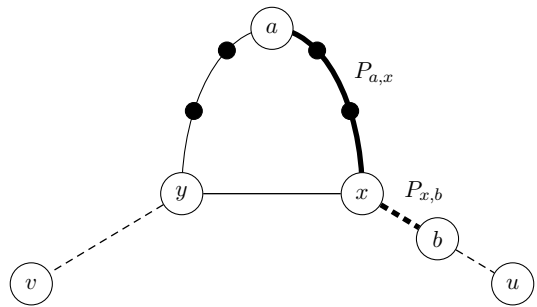
\begin{figure}[h]
			\centering
			\begin{tikzpicture}	[scale=.8, every node/.style={scale=.8}]

			\draw [line width = 2pt] (1.5,0) arc [x radius=1.5, y radius=2.75, start angle=0, end angle=90];
			\draw (1.5,0) arc [x radius=1.5, y radius=2.75, start angle=0, end angle=180]
			\foreach \t in {0.1666, .3333, 0.6666, .8333} { node[draw,circle,inner sep=0pt,minimum size=8pt,fill=black, pos=\t] {} };
			
			\node[draw,circle,minimum size=.7cm,inner sep=2pt,fill=white] at  (0, 2.75 ) {$a$};
			\node[draw,circle,minimum size=.7cm,inner sep=2pt,fill=white] (x) at (1.5,0 ) {$x$};
			\node[draw,circle,minimum size=.7cm,inner sep=2pt,fill=white] (y) at (-1.5,0 ) {$y$};
			\draw[black, line width = .5pt, -] (x)  --  (y) {}; 
			
			\node[draw,circle,minimum size=.7cm,inner sep=2pt,fill=white] (u) at (4,-1.5 ) {$u$};
			\node[draw,circle,minimum size=.7cm,inner sep=2pt,fill=white] (b) at ($(x) !0.5!(u)$) {$b$};
			\node[draw,circle,minimum size=.7cm,inner sep=2pt,fill=white] (v) at (-4,-1.5 ) {$v$};
			
			\draw[black, line width = 2pt, densely dashed,  -] (x)  --  (b); 
			\draw[black, line width = .5pt, densely dashed,  -] (b)  --  (u); 
			\draw[black, line width = .5pt, densely dashed, -] (y)  --  (v); 
			
			\node at ($(x) !0.5!(b) + (.4,.4)$) {$P_{x,b}$};
			\node at (1.7,2 ) {$P_{a,x}$};
			
			\end{tikzpicture}
			
			\caption{The path $P_{a,b}=P_{a,x}P_{x,b}$ contains $x$ but not $y$.}
			\label{fig:SusPathFig6}
		\end{figure}

Using $w'$ we wish to construct a weight function $w_1 : E(G+xy) \to [0,\infty)$ that induces $\mathcal{P}_H$ and every path $P_{u,v} \in \mathcal{P}$, with $u\in U$, $v\in H$. To motivate our definition of $w_1$, consider what happens when $\big(\bigcup_{i=1, i\neq n+1}^{2n+1}U_{i} \big)\setminus \{x,y\} =\emptyset$ so that $G'$ coincides with $G$ and $Q'$ coincides with $Q$. In this case, $w'$ is a weight function on $G$ and to define $w_1$ we need only extend this function to the edge $xy$. Also, in this scenario, the path system $\mathcal{P}'$ is obtained by taking the paths in $\mathcal{P}$ and replacing any occurrence of the edge $xy$ with $Q$.  It is not too difficult to see that if we take  $w_1(xy)\coloneqq w'(Q)$ and $w_1(e) \coloneqq w'(e)$, $e\neq xy$, then this weight function $w_1$ induces $\mathcal{P}$. Indeed, since $Q$ is a geodesic w.r.t.\ $w'$ the edge $xy$ is also a geodesic w.r.t.\ $w_1$. Thus, taking a $w_1$-geodesic and replacing any occurrence of the path $Q$ with the edge $xy$ yields another $w_1$-geodesic. The general definition of $w_1$ is similar, but slightly more care is needed.  
		
		Set $K=w'(Q)$ and fix some $0 < r < \min(w'(x\beta), w'(y\alpha) ) \leq \frac{K}{2}$. 
		We define $w_1 : E(G+xy) \to [0,\infty)$ as follows:
		\begin{itemize}
			\item[-] For $e\in H$, $e\neq xy$, $w_1(e)\coloneqq w'(e)$
			\item[-] $w_1(xy) \coloneqq K$
			\item[-] $w_1(e_1) \coloneqq w'(x\beta) + r$
			\item[-] $w_1(e_{2n+1}) \coloneqq  w'(y\alpha) +r$
			\item[-] For $e\in U$, $w_1(e) \coloneqq  w'(e)$
			\item[-] For all other $e\in C$, $w_1(e)\coloneqq 0$
		\end{itemize}
		Consider the subpath $L \coloneqq e_{2n+1}U e_1$ of $Q$.  Observe that 
		\begin{equation*}
		\begin{split}
		w_1(L) &= w_1(e_{2n+1}) + w_1(U) + w_1(e_1) \\
		&= (w'(y\alpha) +r) + w'(U) +(w'(x\beta) +r)\\
		&=  w'(Q) +2r\\
		&= K +2r.
		\end{split}
		\end{equation*}
		In some sense, we identify the path $L$ in $G$ with the path $Q'$ in $G'$. Note that the edges in $L$ and $xy$ are the only edges of $C$ with positive weight. Also, we could take $r=0$ in the definition of $w_1$ and maintain its desired properties, but then we would have $w_1(Q) = w_1(xy)$. The choice of $r>0$ will make it so that $w_1(xy)<w_1(Q)$ and will allow us to construct a weight function where the paths in $\mathcal{P}_C$ are {\em strictly} induced. This will in turn give us more flexibility in defining $w$.
		
We show that $w_1$ induces $\mathcal{P}_H$. 
Any $w_1$-geodesic between two vertices in $H$ contains no vertices in $Q$. Otherwise, it contains all of $Q$ and, since $w_1(xy) =K< K+2r = w_1(L) = w_1(Q)$, replacing $Q$ with the edge $xy$ gives a path of strictly smaller weight. For a path $P$ in $H$ let $P'$ be the path in $G'$ obtained by replacing any occurrence of the edge $xy$ with the path $Q'$. The mapping $P\mapsto P'$ is a one-to-one correspondence between the paths in $H$ and the paths in $G'$ with endpoints in $V(H)$. Since $w_1(xy) = w'(Q)$ these two paths have the same weight, i.e. $w_1(P) = w'(P')$. It follows that $P$ is a $w_1$-geodesic iff $P'$ is a $w'$-geodesic. By construction of $\mathcal{P}$, for any path $P$ in $H$ we have $P \in \mathcal{P}$ iff $P'\in \mathcal{P}'$. Since $w'$ induces $\mathcal{P}'$ the claim follows.
		
Next we show that for $u\in U$ and $v\in H$, the path $P_{u,v} \in \mathcal{P}$ is a $w_1$-geodesic. Since $\mathcal{P}$ is consistent, then either $P_{u,v} = P_{u,x} P_{x,v}$ or $P_{u,v} = P_{u,y} P_{y,v}$. We first check that both $P_{u,y}$ and $P_{u,x}$ are $w_1$-geodesics. As $u$, $x$ and $y$ are all in the cycle $C$ it suffices to show the inequality \mbox{ $|w_1(P_{u,y})  -  w_1(P_{u,x}) |<  w_1(xy)$} holds.
The only edges in $C$ with non-zero weight  are the  edges in $L =e_{2n+1}P_{\alpha, u}P_{u,\beta} e_1$ and $xy$ so that,
\begin{equation*}
\label{eg:ref1}
\tag{$\ast_4 $}
w_1(P_{u,y}) =w_1(P_{u,\alpha} \  e_{2n+1}) = w'(P_{u,y}')  + r, \hspace{15mm} w_1(P_{u,x})=w_1(P_{u,\beta} \  e_{1}) = w'(P_{u,x}') + r,
\end{equation*}
where $w'(P_{u,y}') + w'(P_{u,x}')  = w(Q') = K. $ Therefore,
		\begin{equation*}
		\label{eg:ineq1}
		\tag{$\ast_5$}
		|w_1(P_{u,y})  -  w_1(P_{u,x}) |=|w_1(P_{u,\alpha} \  e_{2n+1}) -  w_1(P_{u,\beta} \ e_{1}) | = |w'(P_{u,y}')  - w'(P_{u,x}') | <K = w_1(xy)
		\end{equation*}
		and both $P_{u,y}$ and $P_{u,x}$ are $w_1$-geodetic.  Since the paths $P_{x,v}$ and $P_{y,v}$ are in $\mathcal{P}_H$ they are also $w_1$-geodesics. So by \eqref{idea2} either $P_{u,x} P_{x,v}$ or $P_{u,y} P_{y,v}$ is a $uv$ geodesic w.r.t.\ $w_1$.
		Also, as $P_{x,v} \in \mathcal{P}_H$,
		from what we saw before, $w_1(P_{x,v}) = w'(P_{x,v}')$ and so
		$$w_1(P_{u,x} P_{x,v}) =w_1(P_{u,x}) + w_1(P_{x,v})  = w'(P_{u,x}')+ w'(P_{x,v}') +r = w'(P_{u,x}'P_{x,v}') +r.$$
		Similarly, $w_1(P_{y,v}) = w'(P_{y,v}')$
		and
		$$w_1(P_{u,y} P_{y,v}) =w_1(P_{u,y}) + w_1(P_{y,v})  = w'(P_{u,y}')+ w'(P_{y,v}') +r= w'(P_{u,y}'P_{y,v}') +r.$$
		Suppose $P_{u,v} = P_{u,x}P_{x,v}$.
		By construction of $\mathcal{P}'$,  $P_{u,v}' = P_{u,x}' P_{x,v}'$ and since $P_{u,v}'$ is a $w'$-geodesic,  $w'(P_{u,x}' P_{x,v}')=w'(P_{u,v})\leq w'(P_{u,y}' P_{y,v}')$, which implies
		$$w_1(P_{u,v})=w_1(P_{u,x} P_{x,v}) = w'(P_{u,x}' P_{x,v}') +r\leq w'(P_{u,y}' P_{y,v}') + r\leq w_1(P_{u,y} P_{y,v})$$
		and $P_{u,v}$ is a $w_1$-geodesic. The argument when $P_{u,v} = P_{u,y}P_{y,v}$ is identical.
		
We next establish some inequalities that we will need below. By construction
$r < \min(w'(x\beta), w'(y\alpha) )$, and $w_1(e_1) =w'(x\beta) +r$ and $w_1(e_{2n+1}) =w'(y\alpha) +r$ so that
\begin{equation*}
\label{eg:ineq2}
\tag{$\ast_6$}
|w_1(e_1) - 2r| = w'(x\beta) - r < w_1(e_1), \hspace{20mm}|w_1(e_{2n+1}) - 2r| = w'(y\alpha) -r < w_1(e_{2n+1})
\end{equation*}
		
		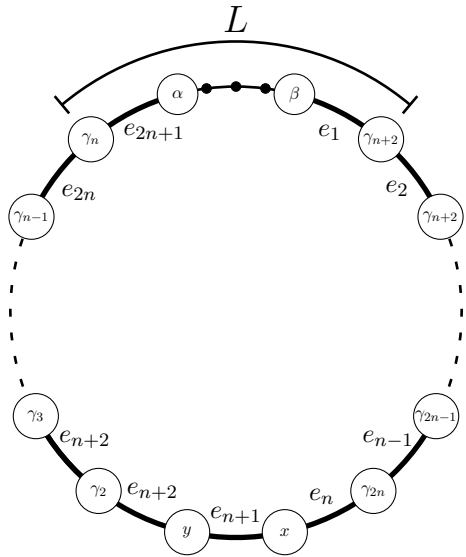
\begin{figure}[h]
				\centering
				\begin{tikzpicture}	[scale=0.6, every node/.style={scale=0.6}]
				\def \inc {25}
				\def \radius {5}
				\def \rightAngle {75}
				\def \leftAngle {105}
				\pgfmathsetmacro\rightAngleII {\rightAngle- \inc}
				\pgfmathsetmacro \leftAngleII {\leftAngle + \inc}
				\pgfmathsetmacro\rightAngleIII {\rightAngle - 2 * \inc}
				\pgfmathsetmacro\leftAngleIII {\leftAngle + 2*\inc}
				\pgfmathsetmacro\rightAngleIV {\leftAngleIII- 180}
				
				\pgfmathsetmacro\downRight {-90+ \inc/2}
				\pgfmathsetmacro\downLeft {-90- \inc/2}
				\pgfmathsetmacro\downRightII {-90+ \inc/2+ \inc}
				\pgfmathsetmacro\downLeftII {-90- \inc/2- \inc}
				\pgfmathsetmacro\downRightIII {-90+ \inc/2+ 2*\inc}
				\pgfmathsetmacro\downLeftIII {-90- \inc/2- 2*\inc}

				\draw[black, line width = 1pt] ([shift={(0,0)}] \rightAngleII :\radius + \radius/5 ) arc[radius=\radius + \radius/5 , start angle= \rightAngleII , end angle= \leftAngleII];
				\draw[black, line width = 1pt] (\rightAngleII :\radius + \radius/5 - \radius/20) -- (\rightAngleII :\radius + \radius/5 +\radius/20);
				\draw[black, line width = 1pt] (\leftAngleII :\radius + \radius/5 - \radius/20) -- (\leftAngleII :\radius + \radius/5 +\radius/20);
				\node[scale =2] at (90: \radius+ \radius/5 +\radius/10 ) {\large$L$};
				
				\node[scale =1.5] at ({\rightAngle/2 - \rightAngleII/2 + \rightAngleII}: \radius - \radius/10 ) {$e_1$};
				\node[scale =1.5] at ({\leftAngleII/2-\leftAngle/2+\leftAngle -3}: \radius -\radius/8 ) {$e_{2n+1}$};
				\node[scale =1.5] at ({\rightAngleII/2 - \rightAngleIII/2 + \rightAngleIII}: \radius - \radius/10 ) {$e_2$};
				\node[scale =1.5] at ({\leftAngleIII/2-\leftAngleII/2+\leftAngleII }: \radius -\radius/8 ) {$e_{2n}$};
				
				\draw[black, line width = 1pt] ([shift={(0,0)}] \rightAngle :\radius ) arc[radius=\radius , start angle= \rightAngle , end angle= \leftAngle];
				
				\draw[black, line width = 2pt] ([shift={(0,0)}] \leftAngle :\radius ) arc[radius=\radius , start angle= \leftAngle , end angle=  \leftAngleII];
				\draw[black, line width = 2pt] ([shift={(0,0)}]\rightAngleII:\radius ) arc[radius=\radius , start angle=\rightAngleII, end angle= \rightAngle];

				\node[draw,circle,minimum size=.9cm,inner sep=0pt,fill=white] at (\leftAngle: \radius ) {$\alpha$};
				\node[draw,circle,minimum size=.9cm,inner sep=0pt,fill=white] at (\rightAngle: \radius ) {$\beta$};
				\node[draw,circle,inner sep=0pt,minimum size=6pt,fill=black] at (\leftAngle/4 - \rightAngle/4+ \rightAngle: \radius ) {};
				\node[draw,circle,inner sep=0pt,minimum size=6pt,fill=black] at  (\leftAngle /4*2- \rightAngle/4 *2+ \rightAngle: \radius ) {};
				\node[draw,circle,inner sep=0pt,minimum size=6pt,fill=black] at  (\leftAngle/4*3 - \rightAngle/4*3 + \rightAngle: \radius )  {};

				\draw[black, line width = 2pt] ([shift={(0,0)}] \leftAngleII :\radius ) arc[radius=\radius , start angle= \leftAngleII , end angle=  \leftAngleIII];
				\draw[black, line width = 2pt] ([shift={(0,0)}]\rightAngleIII:\radius ) arc[radius=\radius , start angle=\rightAngleIII, end angle= \rightAngleII];
				
				\draw[black, line width = 1pt, loosely dashed] ([shift={(0,0)}] \leftAngleIII :\radius ) arc[radius=\radius , start angle= \leftAngleIII , end angle=  360+\downLeftIII];
				\draw[black, line width = 1pt, loosely dashed] ([shift={(0,0)}] \rightAngleIII :\radius ) arc[radius=\radius , start angle= \rightAngleIII , end angle=  \downRightIII];
				
				\draw[black, line width = 2pt] ([shift={(0,0)}] \downLeftIII :\radius ) arc[radius=\radius , start angle=\downLeftIII , end angle=  \downLeftII];
				\draw[black, line width = 2pt] ([shift={(0,0)}] \downRightIII :\radius ) arc[radius=\radius , start angle=\downRightIII , end angle=  \downRightII];
				\draw[black, line width = 2pt] ([shift={(0,0)}] \downLeftII :\radius ) arc[radius=\radius , start angle=\downLeftII , end angle=  \downLeft];
				\draw[black, line width = 2pt] ([shift={(0,0)}] \downRightII :\radius ) arc[radius=\radius , start angle=\downRightII , end angle=  \downRight];
				\draw[black, line width = 2pt] ([shift={(0,0)}] \downLeftIII :\radius ) arc[radius=\radius , start angle=\downLeftIII , end angle=  \downLeftII];
				\draw[black, line width = 2pt] ([shift={(0,0)}] \downRight :\radius ) arc[radius=\radius , start angle=\downRight , end angle=  \downLeft];
				
				\node[draw,circle,minimum size=1cm,inner sep=0pt,fill=white] at (\leftAngleII: \radius ) {$\gamma_{n}$};
				\node[draw,circle,minimum size=1cm,inner sep=0pt,fill=white] at (\rightAngleII: \radius ) {$\gamma_{n+2}$};
				
				\node[draw,circle,minimum size=1cm,inner sep=0pt,fill=white] at (\leftAngleIII: \radius ) {$\gamma_{n-1}$};
				\node[draw,circle,minimum size=1cm,inner sep=0pt,fill=white] at (\rightAngleIII: \radius ) {$\gamma_{n+2}$};

				\node[draw,circle,minimum size=1cm,inner sep=0pt,fill=white] at (\downRightIII: \radius ) {$\gamma_{2n-1}$};
				\node[draw,circle,minimum size=1cm,inner sep=0pt,fill=white] at (\downRightII: \radius ) {$\gamma_{2n}$};
				\node[draw,circle,minimum size=1cm,inner sep=0pt,fill=white] at (\downRight: \radius ) {$x$};
				
				\node[draw,circle,minimum size=1cm,inner sep=0pt,fill=white] at (\downLeftIII: \radius ) {$\gamma_{3}$};
				\node[draw,circle,minimum size=1cm,inner sep=0pt,fill=white] at (\downLeftII: \radius ) {$\gamma_{2}$};
				\node[draw,circle,minimum size=1cm,inner sep=0pt,fill=white] at (\downLeft: \radius ) {$y$};
				
				\node[scale =1.5] at ({\downRightII/2 - \downRightIII/2 + \downRightIII}: \radius - \radius/10 -.1) {$e_{n-1}$};
				\node[scale =1.5] at ({\downLeftIII/2-\downLeftII/2+\downLeftII }: \radius -\radius/8 ) {$e_{n+2}$};
				\node[scale =1.5] at ({\downRight/2 - \downRightII/2 + \downRightII}: \radius - \radius/10 ) {$e_{n}$};
				\node[scale =1.5] at ({\downLeftII/2-\downLeft/2+\downLeft }: \radius -\radius/8 ) {$e_{n+2}$};
				\node[scale =1.5] at ({\downRight/2 - \downLeft/2 + \downLeft}: \radius - \radius/10 ) {$e_{n+1}$};
				
				\end{tikzpicture}
				
				\caption{The cycle $C$, with $U_i=\{\gamma_i\}$, $i\neq n+1$.}
				\label{fig:SusPathFig2}
		\end{figure}
		
The next step is to modify $w_1$ to $w_2: E(G+xy) \to [0,\infty)$. At this point we only need to find appropriate weights for each $e_i$, $i\neq 1,n+1,2n+1$, so that $w_2$ strictly induces $\mathcal{P}_C$. The canonical path system of an odd cycle is induced by uniform edge weights. So, in search of a weight function which strictly induces $\mathcal{P}_C$, it is suggestive to assign the same positive weight to every edge $e_i$, and zero weight to all edges in $U_i$ for $1\leq i \leq 2n+1$. However, we need to deal with the fact that $w_1(e_1)$, $w_1(e_{n+1})$, and $w_1(e_{2n+1})$ are already defined and not necessarily equal. To solve the problem we use the fact that these three edge weights roughly equal $\frac{K}{2} +r$.
		
		We define $w_2$ as follows:
		\begin{itemize}
			\item[-] For $i\neq 1,n+1, 2n+1$, $w_2(e_i) \coloneqq  \frac{K}{2} + r$
			\item[-] Otherwise, $w_2(e) \coloneqq w_1(e)$
		\end{itemize}
		We recall that for $u\in U_i$, $f(u) = e_i$. To prove that $w_2$ strictly induces $\mathcal{P}_C$, by \Cref{lem:cyclePathCharac}, it suffices to show that \mbox{$|w_2(P_{u,\alpha_i}) - w_2(P_{u,\beta_i}) | < w_2(e_{i})$} for every $u\in U_{i}$, where $e_i = \alpha_i \beta_i $. Moreover, for $i\neq n+1$ the edges in $U_i$ have zero $w_2$-weight. So if the inequality holds for some vertex in $U_{i}$, it holds for all of them. It therefore suffices to show that $\mathcal{P}_C$ is strictly induced by $w_2$ in the case where $U_i = \{\gamma_i\}$ is a singleton, $i\neq n+1$ with $\gamma_{1}=y$ and $\gamma_{2n+1}=x$, \Cref{fig:SusPathFig2}. In this scenario, $L =e_{2n+1} U_{n+1} e_{1}= \gamma_{n}\alpha U_{n+1} \beta\gamma_{n+2}=  \gamma_{n}P_{\alpha,\beta}\gamma_{n+2}$.
		
		We start with the range $1<i \leq n$, so $e_i = \gamma_{n+i}\gamma_{n+i+1}$. To calculate $w_2(P_{\gamma_i, \gamma_{n+i}})$ and $w_2(P_{\gamma_i,\gamma_{n+i+1}})$ we write
		$$P_{\gamma_i,\gamma_{n+i}} = \gamma_{i} \gamma_{i+1}\cdots \gamma_{n}L\gamma_{n+2}\cdots \gamma_{n+i},\hspace{15mm} P_{\gamma_{i},\gamma_{n+i+1}} = \gamma_{i}\gamma_{i-1}\cdots \gamma_{1}\gamma_{2n+1}\cdots \gamma_{n+i+2}\gamma_{n+i+1}.$$
		In $E(P_{\gamma_i,\gamma_{n+i}}) \setminus E(L)$ there are $n-2$ edges, and each such edge $e$ satisfies $w_2(e) = \frac{K}{2} + r$. Therefore, $w_2(P_{\gamma_i,\gamma_{n+i}}) = (n-2) \big(\frac{K}{2} + r\big) + w_2(L)$. Similarly, $w_2(e) = \frac{K}{2} + r$ for every edge $e\in P_{\gamma_{i},\gamma_{n+i+1}} $, where $e\neq\gamma_{1}\gamma_{2n+1}= xy$.
		The number of such edges is $n-1$, and therefore $w_2(P_{\gamma_{i},\gamma_{n+i+1}}) = (n-1) \big(\frac{K}{2} + r\big) + w_2(xy)$. Also, by construction $w_2(L) = w_1(L) =K + 2r$ and $w_2(xy)=w_1(xy) = K$. Since $0 < r < \frac{K}{2}$, there holds
		$$|w_2(P_{\gamma_i,\gamma_{n+i}})  - w_2(P_{\gamma_i,\gamma_{n+i+1}})  | = \frac{K}{2} - r < \frac{K}{2} + r = w_2(e_i).$$
		Now we consider the case $i=1$,  $e_1 = \beta\gamma_{n+2}$. Since $\gamma_1 = y$ we can write
		$$P_{y,\beta} =y \gamma_{2}\cdots \gamma_{n}P_{\alpha,\beta},  \hspace{20mm} P_{y,\gamma_{n+2}} = yx\gamma_{2n}\cdots \gamma_{n+3}\gamma_{n+2}.$$
		Clearly, $$w_2(P_{y,\gamma_{n+2}} ) = (n+1) \frac{K}{2} + (n-1)r,$$ since every edge in $P_{y,\gamma_{n+2}} $ other than $xy$ weighs $\frac{K}{2} + r$.\\
		We turn to calculate $w_2(P_{y,\beta})$. We first observe:
		$$w_2(\gamma_{n}P_{\alpha,\beta} ) = w_2(\gamma_{n}P_{\alpha,\beta}\gamma_{n+2}) - w_2(\beta\gamma_{n+2}) =  w_2(L) - w_2(\beta\gamma_{n+2}) = K+2r- w_2(e_1).$$ 
		But aside of the edges in $\gamma_{n}P_{\alpha,\beta}$ there are $n-1$ edges in $P_{y,\beta}$ with weight $\frac{K}{2} + r$. Therefore:
		$$w_2(P_{y,\beta} ) = (n+1) \frac{K}{2} + (n+1)r - w_2(e_1).$$
		Since $w_2(e_1) = w_1(e_1)$, from \eqref{eg:ineq2} we get
		$$|w_2(P_{y,\beta} ) - w_2(P_{y,\gamma_{n+2}} )| =|w_2(e_1) -2r| =|w_1(e_1) -2r|  <w_1(e_1) = w_2(e_1).$$
		The argument when $n+2\leq i \leq 2n+1$ is similar. There remains the case $i= n+1$, $e_{n+1} =xy$. For $u \in U_{n+1}$ we can write 
		$$P_{u,x} = P_{u,\beta} \gamma_{n+2} \gamma_{n+3}\cdots \gamma_{2n}x \hspace{10mm} \text{ and } \hspace{10mm} P_{u,y} = P_{u,\alpha} \gamma_{n}\gamma_{n-1}\cdots \gamma_{2}y.$$
		Note that
		$$w_2(P_{u,x}) = (n-1) \big(\frac{K}{2} + r\big) + w_2( P_{u,\beta} \gamma_{n+2} ),$$
		since $P_{u,x}$ has $n-1$ edges of weight $\frac{K}{2} + r$ in addition
		to the edges contains in $P_{u,\beta} \gamma_{n+2} $. Similarly,   $w_2(P_{u,y}) = (n-1) \big(\frac{K}{2} + r\big) + w_2(P_{u,\alpha} \gamma_{n})$. Since $w_2(P_{u,\beta} \gamma_{n+2}) = w_1(P_{u,\beta}e_1)$ and $w_2(P_{u,\alpha} \gamma_{n})   = w_1(P_{u,\alpha} e_{2n+1})  ,$
		by \eqref{eg:ineq1} it follows
		$$|w_2(P_{u,x}) - w_2(P_{u,y}) | = |w_1(P_{u,\alpha} e_{2n+1})- w_1(P_{u,\beta}e_1)|  < w_1(xy) = w_2(xy).$$

Finally, we construct a strictly positive $w:E(G+xy) \to (0,\infty)$ which induces $\mathcal{P}$. This is accomplished by perturbing $w_2$. Using \eqref{idea1}, we see that $w_2$ is a non-negative weight function which induces $\mathcal{P}$. The strictly metrizable version of \Cref{thm:SuspendedPath} already follows now, using \Cref{lem:strictWeightPerturb}. However, the non-strictly metrizable statement requires a bit more work. 
		
		Set $N_1 \coloneqq \sum_{i=1}^{n} |E(U_i) |$ and $N_2 \coloneqq \sum_{i=n+2}^{2n+1} |E(U_i) |$ and fix some small $\delta >0$. We construct $w$ as follows:
		\begin{itemize}
			\item[-] For $e\in U_i$, $i\neq n+1$, $w(e) \coloneqq \delta$
			\item[-] $w(e_{1}) \coloneqq w_2(e_{1})  + N_1\delta$
			\item[-] $w(e_{2n+1}) \coloneqq w_2(e_{2n+1}) +N_2\delta$
			\item[-] Otherwise, $w(e) \coloneqq w_2(e) $
		\end{itemize}
		Since $w_2$ strictly induces $\mathcal{P}_C$, by \Cref{lem:strictWeightPerturb}, $w$ also strictly induces $\mathcal{P}_C$. Moreover, as $w_1$ induces $\mathcal{P}_H$ and $w$ and $w_1$ agree over $H$, this implies that $w$ also induces $\mathcal{P}_H$. If every path in $\mathcal{P}$ between $U_i$ and $H$ is $w$-geodetic then by \eqref{idea1}, $w$ induces $\mathcal{P}$. Let $u \in U_i$ and $v \in H$. By \eqref{idea2}, either  $P_{u,x} P_{x,v}$ or $P_{u,y} P_{y,v}$ is a $w$-geodesic. From \eqref{eg:ref1}, $w_2(P_{u,\beta} \ e_{1})= w_1(P_{u,\beta} \ e_{1}) =  w_1(P_{u,x})$.  Also, we saw that
		\mbox{$w_2(P_{u,x})  =  (n-1) \big(\frac{K}{2} + r\big) + w_2(P_{u,\beta} \  e_{1}).$}
		The definition of $w$ yields 
		$$w(P_{u,x}) = w_2(P_{u,x}) + N_1\delta +N_2\delta = w_1(P_{u,x}) + (n-1) \big(\frac{K}{2} + r\big) + N_1\delta +N_2\delta=  w_1(P_{u,x}) +\tilde{K},$$
		where $\tilde{K} =(n-1) \big(\frac{K}{2} + r\big) + N_1\delta +N_2\delta$.
		Since $w$ and $w_1$ agree on $H$ 
		$$ w(P_{u,x} P_{x,v}) = w(P_{u,x})  + w(P_{x,v}) = w_1(P_{u,x})+ \tilde{K} +w_1(P_{x,v}) = w_1(P_{u,x} P_{x,v}) + \tilde{K}.$$
		In the same way,
		$$w(P_{u,y} P_{y,v}) = w_1(P_{u,y} P_{y,v}) + \tilde{K}.$$
		Since either $P_{u,v} = P_{u,x} P_{x,v}$ or $P_{u,v} = P_{u,y} P_{y,v}$, this implies $P_{u,v}$ is $w$-geodetic if it is $w_1$-geodetic. But it was already shown that $P_{u,v}$ is $w_1$-geodetic.

	\end{proof}

	Based on \Cref{thm:SuspendedPath} we conclude:
	\begin{corollary}\label{cor:outerPlanarMet}
		Every outerplanar graph $G$ is strictly metrizable.
	\end{corollary}
	
	\begin{proof}
		As usual we can assume w.l.o.g.\ that $G$ is $2$-connected. We argue by induction on $|E(G)|-|V(G)|$. In the base case $G$ is a cycle. For the induction step we apply \Cref{thm:SuspendedPath} to $G\setminus e$, where $e=uv$ is an internal edge such that $u$ and $v$ are connected by a suspended path.
	\end{proof}
	
	Next we show that all small graphs are metrizable.
	\begin{proposition}
		Every graph of order at most $4$ is strictly metrizable.
	\end{proposition}
	\begin{proof}
		Every graph other than $K_4$ with at most four vertices is outerplanar and therefore strictly metrizable by \Cref{cor:outerPlanarMet}. There remains the neighborly path systems of $K_4$ which is strictly induced by unit edge weights. 
	\end{proof}
	\begin{corollary}
		Every (strictly) non-metrizable graph contains a subdivision of $K_{2,3}$.
	\end{corollary}
	\begin{proof}
		Every $2$-connected non-outerplanar graph other than $K_4$ contains a subdivision of $K_{2,3}$.
	\end{proof}
	We next introduce infinitely many graphs which are metrizable but not strictly metrizable.
	\begin{proposition}
		The graph $K_{2,n}$, $n\geq 4$, is metrizable but not strictly metrizable.
	\end{proposition}
	\begin{proof}
		Let $a_1,a_2 ~;~ b_1,\dots, b_n$ be the vertex bipartition of $K_{2,n}$. We show that $K_{2,n}$ is metrizable by induction on $n$. The base case $K_{2,2}$ coincides with the $4$-cycle which we already know to be metrizable. Let $\mathcal{P}$ be a path system in $K_{2,n+1}$, $n\geq 2$. If $\mathcal{P}$ is not neighborly, then some edge, say $a_1b_{n+1}$ is not in $\mathcal{P}$, and $\mathcal{P}$ can be considered a path system in $K_{2,n+1} \setminus a_1b_{n+1}$. But $K_{2,n+1} \setminus a_1b_{n+1}$ has two biconnected components, namely, $K_{2,n}$ and the edge $a_2b_{n+1}$. By induction $K_{2,n}$ is metrizable, and since both its biconnected components are metrizable so is $K_{2,n+1} \setminus a_1b_{n+1}$. If $\mathcal{P}$ is neighborly it is easily verified that in this case every path in $\mathcal{P}$ has length $1$ or $2$ and is induced by constant edge weights.\\
		Next we show that $K_{2,4}$ is not strictly metrizable.	Again we write $V(K_{2,4})=(a_1, a_2~;~ b_1, b_2, b_3, b_4)$ and consider the path system that includes all the edges of $K_{2,4}$ and the paths
		\begin{alignat*}{3}
		&    P_{b_1,b_2} = b_1a_2b_2             \qquad \qquad  & P_{b_3,b_4} = b_3a_2b_4                 \qquad \qquad &P_{b_1,b_3} = b_1a_1b_3    \qquad\qquad & P_{b_2,b_4} = b_2a_1b_4    \qquad\qquad \\
		& P_{b_2,b_3} = b_2a_2b_3           \qquad \qquad  &   P_{b_1,b_4} = b_1a_1b_4   \qquad\qquad          & P_{a_1,a_2} = a_1b_1a_2   \qquad\qquad &\\
		\end{alignat*}
		Assume towards a contradiction that this system is strictly metrizable under some weight function $w$. The first row of paths yields the following inequalities
		\begin{equation*}
		\begin{split}
		w(a_2b_1) + w(a_2b_2) & < w(a_1b_1) + w(a_1b_2)\\
		w(a_2b_3) + w(a_2b_4) & < w(a_1b_3) + w(a_1b_4)\\
		w(a_1b_1) + w(a_1b_3) & < w(a_2b_1) + w(a_2b_3)\\
		w(a_1b_2) + w(a_1b_4) & < w(a_2b_2) + w(a_2b_4)\\
		\end{split}
		\end{equation*}
		Adding these inequalities and canceling terms we find $0<0$, a contradiction.
	\end{proof}
	
	\section{Structural Description of Metrizability}\label{sec:std}
	
	As shown above \Cref{prop:TopMinClosed}, the sets of metrizable and strictly metrizable graphs are closed under taking a topological minor. It follows that there is a minimal set of graphs $\mathcal{F}_{M}$ such that a graph $G$ is metrizable if and only if no graph in $\mathcal{F}_{M}$ is a topological minor of $G$. This set is minimal in the sense that if $H_1\neq H_2$ belong to $\mathcal{F}_{M}$, then $H_1$ is not a topological minor of $H_2$. Likewise, there is a minimal set $\mathcal{F}_{SM}$ such that $G$ is strictly if and only if no graph from $\mathcal{F}_{SM}$ is a topological minor of $G$.
	\begin{theorem}\label{thm:finiteMinGraphs}
		Both $\mathcal{F}_{M}$ and $\mathcal{F}_{SM}$ are finite.
	\end{theorem}

We start with some preliminary comments. In {\em each} of our proofs in  \Cref{sec:metRare} where we show that some graph $G$ is non-metrizable we find some subgraph of $G$ that is a subdivision of a graph in \Cref{fig:MinimumGraphs}. But all graphs in \Cref{fig:MinimumGraphs} have $9$ or fewer vertices, so that if $G$ is a graph of order $10$ or above, and $G$ is shown to be non-metrizable by invoking one these theorems, it necessarily follows that $G$ is not minimal, i.e. $G\notin \mathcal{F}_{M}.$ Similarly, it follows that $G\notin \mathcal{F}_{SM}$ since a graph which is not metrizable is clearly also not strictly metrizable.
We make the following crucial observation which is an easy consequence of \Cref{thm:SuspendedPath}:
\begin{observation}\label{obs:susPathMin}
If $G\in \mathcal{F}_{M}$ or $G\in \mathcal{F}_{SM}$, and if $xy \in E(G)$, then $G$ cannot contain a suspended path of length greater than $1$ with endpoints $x$ and $y$. 
\end{observation}
\begin{proof}
If $G\in \mathcal{F}_{M}$, then $G$ is non-metrizable. By \Cref{thm:SuspendedPath} the graph $G\setminus xy$ is non-metrizable as well. An identical argument works when $G\in \mathcal{F}_{SM}$
\end{proof}
	
Recall that we call a vertex essential if it has degree at least $3$. We prove next an upper bound on the number of essential vertices in a minimal graph. 
\begin{proposition}\label{prop:minMetDeg3}
Every graph $G$ in $\mathcal{F}_{M}$ or in $\mathcal{F}_{SM}$ has at most $12$ essential vertices.
\end{proposition}
\begin{proof}
Let $G$ be a $2$-connected graph with at least $13$ essential vertices and let $G'$ be the multigraph obtained by suppressing all vertices of degree $2$. Clearly, $G'$ has at least $13$ vertices and it contains no loops, since $G$ is $2$-connected. If $G'$ is actually a simple graph then it is not metrizable by \Cref{thm:deg3Met} and therefore not minimal, implying that neither is $G$.
		
Suppose that $G'$ has a set of four distinct edges $\{e_1, e_1', e_2,e_2'\}$, where $e_1, e_1'$ connect the same two vertices $u_1, v_1$ and likewise for $e_2,e_2'$ and $u_2, v_2$. These edges correspond to suspended $u_1v_1$ paths $P_1$, $P_1'$ and suspended $u_2v_2$ paths $P_2$, $P_2'$, respectively. By \Cref{obs:susPathMin} all these paths have length bigger than $1$. These paths are suspended in $G$ which contains at least $13$ essential vertices, so there is an essential vertex outside of $P_1$, $P_1'$, $Q_1$ or $Q_1'$. By \Cref{lem:fourDisjointPaths} $G$ is not metrizable and therefore not minimal.
		
		In the only remaining case there is exactly one pair of essential vertices $u$ and $v$ in $G'$ with two or three parallel edges between them. These edges correspond to suspended paths in $G$ between $u$ and $v$, of length bigger than $1$ (again, by \Cref{obs:susPathMin}). By eliminating one of these paths, if necessary, we can assume that there are precisely two suspended paths between $u$ and $v$. Replacing these parallel edges in $G'$ by the corresponding suspended paths yields a graph $G''$ which is a subdivision of $G$. It is easily verified that $G''$ has at least $11$ essential vertices outside these two suspended paths. By \Cref{lem:susPathMinDeg3} $G''$ is not metrizable. Consequently $G''$ and hence $G$ is not minimal.
	\end{proof}
	
	\begin{figure}[H]
		\centering
		\begin{tikzpicture}	
		\def \radius {3}
		\def \radiusII {3.5}
		\def \n {20}
		\def \sz {2}
		\draw[line width =1pt] circle(\radius);
		
		\foreach \x in {1,..., \n } {
			\node [draw,circle,minimum size=\sz pt,fill,inner sep=\sz pt] at ( 360 * \x / \n:\radius) {}; 
			\node [draw,circle,minimum size=\sz pt,fill,inner sep=\sz pt] at ( 360 * \x / \n + 180 / \n :\radiusII) {} ; 
			\draw [line width= 1pt,-] ( 360 * \x / \n:\radius)-- ( 360 * \x / \n + 180 / \n :\radiusII);
			\draw [line width=1 pt,-] ( 360 * \x / \n:\radius)-- ( 360 * \x / \n - 180 / \n :\radiusII);
		}
		
		\draw [line width=1pt,-] (-\radius,0)-- ( \radius, 0);
		
		\node [draw,circle,minimum size=\sz pt,fill,inner sep=\sz pt] at (0,0) {} ;

		\end{tikzpicture}
		\caption{This graph is not in $\mathcal{F}_M$ by \Cref{obs:susPathMin}.}
		\label{fig:sunGraph}
	\end{figure}
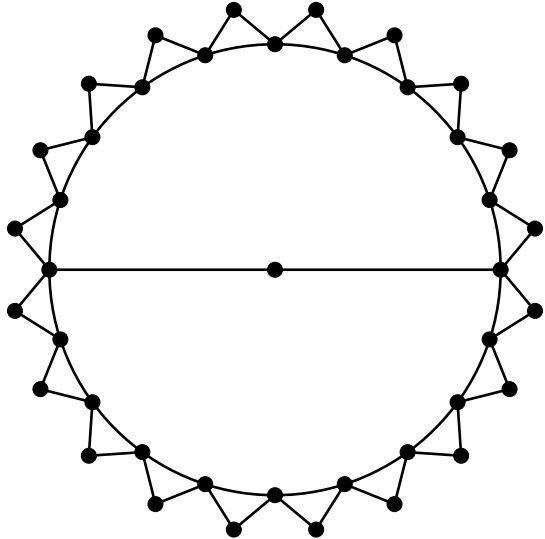
	We note that \Cref{obs:susPathMin}, and therefore \Cref{thm:SuspendedPath}, is crucial in the proof of \Cref{prop:minMetDeg3}. It allows us discount graphs we would otherwise be unsure how to deal with, see for example \Cref{fig:sunGraph}.

	A {\em quasi-order} is a binary relation which is reflexive and transitive. We say that $X$ is a {\em well-quasi-ordered set} (wqo for short) if there a quasi-order $\leq$ on $X$ such that for any sequence $x_1,x_2,x_3,\dots$ in $X$ there exists $i<j$ satisfying $x_i \leq x_j$. The set of all finite graphs is well-quasi-ordered with respect to the topological minor relation.
	We note that \Cref{thm:finiteMinGraphs} can be derived from deep results of Thomas and Liu, \Cite{LT}, who characterized collections of graphs which are well-quasi-ordered with respect to the topological minor relation. However, there is a much a more elementary route to the same goal.
	\begin{proposition}\label{prop:topologicalMinorBoundedOrder}
		Fix $n\geq 0$. If $G_1,G_2,G_2,\dots$ is an infinite sequence of graphs, each with at most $n$ essential vertices, then there exists $1\leq i' < i$ such that $G_i$ contains a subdivision of $G_{i'}$.
	\end{proposition}

	\begin{proof}
		We start with some general observations: Let $M=(a_k\ge a_{k-1}\ge\ldots a_1\ge 1)$ and $M'=(b_{\ell}\ge b_{\ell-1}\ge\ldots b_1\ge 1)$ be multisets of positive integers. We say that $M\succ M'$ if there is an injection $\sigma:[\ell]\to[k]$ such that $b_i\le a_{\sigma(i)}$ for all $i$. We recall that multisets of positive integers are well-quasi-ordered by the relation $\succ$, and that wqo sets are closed under Cartesian product, \Cite{Hig}. 
		
		We turn to prove the proposition.
		By passing to a subsequence, if necessary we can assume that each $G_i$, $i\geq 1$, has precisely $n'\le n$ essential vertices. Assume w.l.o.g.\ that $n'=n$. For each $i$ we label $G_i$'s essential vertices i.e., we fix a bijection from these $n$ vertices to $\{1,2,\dots, n\}$. For $1\le j<k \le n$ let $N_{i,j,k}$ be the multiset consisting of the lengths of all the suspended $jk$ paths in $G_{i}$.
		Let $N_i = \prod_{1\leq j<k \leq n}  N_{i,j,k}$. Clearly $G_{i}$ contains a subdivision of $G_{i'}$ iff $N_{i} \succ N_{i'}$. The conclusion follows from the general comments on wqo sets. 
	\end{proof}
	
	With these results the proof of \Cref{thm:finiteMinGraphs} is clear.
	\begin{proof}{[\Cref{thm:finiteMinGraphs}]}
		By \Cref{prop:minMetDeg3} the set $\mathcal{F}_M$ consists of graphs with at most $12$ essential vertices. Since each graph in $\mathcal{F}_M$ is a minimal element with respect to the topological minor relation, from \Cref{prop:topologicalMinorBoundedOrder} this implies that there can only be finitely many such graphs. The proof dealing with $\mathcal{F}_{SM}$ is identical.
	\end{proof}

	\section{A Continuous Perspective }\label{sec:continuous}
	The metrizability problem for the cycle can be stated in an appealing continuous form. We remark that it is very possible that the results in this section have been previously studied under a different setting, but nonetheless, we think it is worthwhile to examine them through the lens of metrizability. We say that a map $T:S^1\to S^1$ is {\em crossing} if for any $x, y \in S^1$ the segments $[x,T(x)],[y,T(y)]$ intersect. We say a crossing map $T$ is {\em metrizable} if $T$ has a {\em compatible} non-atomic probability measure $\mu$ over $S^1$. Namely for every $x\in S^1$, the points $x$ and $T(x)$ split $S^1$ into two arcs of $\mu$-measure of $\frac 12$ each. 
	\begin{proposition}\label{prop:MeasPresCrossing}
		A crossing map $T:S^1\to S^1$ is metrizable if and only if there exists a $T$-invariant non-atomic probability measure. 
	\end{proposition}
	We start with two simple observations.
	\begin{observation}\label{Obs:CrossInvImgArc}
		If $T:S^1\to S^1$ is crossing, then for any $w\in S^{1}$ the set $T^{-1}(w)$ is either empty, a single point or a connected arc of $S^{1}$.
	\end{observation}
	\begin{proof}
		Say that $T(x) = T(y)=w$ for some $x\neq y$. The points $x,y,w$ split $S^1$ into three arcs $A_{xy}, A_{yw}, A_{wx}$. We claim that $f(z)=w$ for every $z \in A_{xy}$. Indeed, if $f(z)\neq w$ then $f(z)$ is in either $A_{xy}, A_{yw}, or A_{wx}$. But then $y, z$ (resp.\ $x, z$, resp. both) fails the crossing condition.
	\end{proof}
	\begin{observation}\label{Obs:CrossInvImgMeasZero}
		If $\mu$ is compatible with $T$ then $\mu(T^{-1} (x)) = 0$ for all $x\in S^1.$
	\end{observation}
	\begin{proof}
		in view of \Cref{Obs:CrossInvImgArc} and the fact that $\mu$ is atom-free it suffices to consider the case where $T^{-1}(x))=A$ is an arc. But if $\mu(A) > 0$, we can find some $x'\neq x''$ in $A$ so that the $(x',x'')$ sub-arc of $A$ has a positive $\mu$ measure. The partitions of $S^1$ that $x, x'$ induce and the one induced by $x, x''$ cannot both satisfy the compatibility requirement.
	\end{proof}
	\begin{proof}{[\Cref{prop:MeasPresCrossing}]}
		We first show that every $T$-invariant measure $\mu$ is compatible. We need to show that $\mu(A_1)=\mu(A_2)$, where $A_1$ and $A_2$ be the two arcs defined by $x$ and $T(x)$ for an arbitrary $x\in S^{1}$. Since $T$ is crossing, for every $y\in A_2$ there holds $T(y) \in A_1\cup \{T(x), x\}$, whence $T(A_2) \subseteq A_1\cup \{T(x), x\}$. By \Cref{Obs:CrossInvImgMeasZero} both $T^{-1}(T(x))$ and $T^{-1}(x)$ are $\mu$ zero-sets, implying
		$$\mu(A_2) \leq \mu(T^{-1}A_1) + \mu(T^{-1}T(x)) + \mu (T^{-1}(x)) = \mu(T^{-1}A_1) = \mu (A_1).$$
		By symmetry we also have $\mu(A_1)\leq \mu(A_2)$.
		
		The reverse statement says that if $\mu$ is compatible with $T$, then for every arc $A$ there holds
		\[\label{eqn:mu_invariant}
		\mu(T^{-1}A) = \mu(A).    
		\]
		This statement for $A$ and for its complement $A^c$ are equivalent, since $\mu(T^{-1}(A^c)) = \mu((T^{-1}A)^c)= 1 - \mu(T^{-1}A)$ and  $\mu(A^c) = 1 - \mu(A)$. Also by \Cref{Obs:CrossInvImgMeasZero} this holds for $A$ a singleton. So consider an arc with endpoints $x\neq y$, where $T(x) = w_x,  T(y) = w_y$ and $w_x\neq w_y$. The points $x,y,w_x, w_y$ split $S^{1}$ into four arcs $A_{xy}$, $A_{yw_x}$, $A_{w_xw_y}$ and $A_{w_yx}$, \Cref{fig:CircleCrossing}. 
		\begin{figure}[h]
			\centering
			\begin{subfigure}[t]{0.32\textwidth}
				\centering
				\begin{tikzpicture}[scale=0.6, every node/.style={scale=0.6}]
				\node [circle, draw, minimum size=6cm] (c) {};
				
				\node[
				color = pink,
				draw, thick, 
				fill, 
				minimum size=5mm, 
				circle, 
				] (x) at (c.-70) {\color{black}$y$};
				
				\node[
				color = pink,
				draw, thick, 
				fill, 
				minimum size=5mm, 
				circle, 
				] (y) at (c.-110) {\color{black}$x$};
				
				\node[
				color = pink,
				draw, thick, 
				fill, 
				minimum size=5mm, 
				inner sep=1pt,
				circle, 
				] (Tx) at (c.130) {\color{black}$T(y)$};
				
				\node[
				color = pink,
				draw, thick, 
				fill, 
				minimum size=5mm, 
				inner sep=1pt,
				circle, 
				] (Ty) at (c.85) {\color{black}$T(x)$};
				
				\draw [line width=1pt,-, dashed] (x) -- (Tx);
				\draw [line width=1pt,-, dashed] (y) -- (Ty);
				
				\node[scale = 1.25] at ($(x)!0.5!(y) + (0,-.75)$) { \large$A_{xy}$} ;
				\node[scale = 1.25] at ($(y)!0.5!(Tx) + (-2.3,0)$) { \large$A_{w_yx}$} ;
				\node[scale = 1.25] at ($(Tx)!0.5!(Ty) + (-.2,.8)$) { \large$A_{w_xw_y}$} ;
				\node[scale = 1.25] at ($(x)!0.5!(Tx) + (4.2,0)$) { \large$A_{yw_x}$} ;
				\end{tikzpicture}
				\caption{The points $x$, $y$, $T(x)$, $T(y)$ partition $S^1$ into fours arcs.}
				\label{fig:CircleCrossing}
			\end{subfigure}
			\hfill
			\begin{subfigure}[t]{0.32\textwidth}
				\centering
				\begin{tikzpicture}[scale=0.6, every node/.style={scale=0.6}]
				\node [circle, draw, minimum size=6cm] (c) {};
				
				\node[
				color = pink,
				draw, thick, 
				fill, 
				minimum size=5mm, 
				inner sep=1pt,
				circle, 
				] (x) at (c.-60) {\color{black}$T^2(x)$};
				
				\node[
				color = pink,
				draw, thick, 
				fill, 
				minimum size=5mm, 
				circle, 
				] (y) at (c.-120) {\color{black}$x$};

				\node[
				color = pink,
				draw, thick, 
				fill, 
				minimum size=5mm, 
				inner sep=1pt,
				circle, 
				] (Ty) at (c.85) {\color{black}$T(x)$};

				\draw [line width=1pt,-, dashed] (y) -- (Ty);
				\draw [line width=1pt,-, dashed] (x) -- (Ty);
				
				\node[scale = 1.25] at ($(x)!0.5!(y) + (0,-1)$) { \large$A_{xy}$} ;
				\node[scale = 1.25] at ($(x)!0.5!(Ty) + (2.75,.2)$) { \large$A_{wy}$} ;
				\node[scale = 1.25] at ($(y)!0.5!(Ty) + (-3,.2)$) { \large$A_{wx}$} ;
				
				\end{tikzpicture}
				\caption{The crossing function $T$ has a point of discontinuity at $T(x)$.}
				\label{fig:CircleCrossing2}
			\end{subfigure}
			\hfill
			\begin{subfigure}[t]{0.32\textwidth}
				\centering
				\begin{tikzpicture}[scale=0.6, every node/.style={scale=0.6}]
				\node [circle, draw, minimum size=6cm] (c) {};
				\node [circle, minimum size=7.5cm] (d) {};
				
				\node[
				color = pink,
				draw, thick, 
				fill, 
				minimum size=5mm, 
				circle, 
				] (x) at (c.-70) {};
				
				\node[ scale = 1.2] at (d.-70) {$x$};
				
				\node[
				color = pink,
				draw, thick, 
				fill, 
				minimum size=5mm, 
				circle, 
				] (y) at (c.-90) {};
				
				\node[scale = 1.2]  at (d.-90  ) {$x+\delta$};
				
				\node[
				color = pink,
				draw, thick, 
				fill, 
				minimum size=5mm, 
				inner sep=1pt,
				circle, 
				] (Ty) at (c.120) {};
				\node[scale = 1.2]  at ($(d.130) + (.5, .5)$) {$T(x)-\varepsilon$};
				\node[
				color = pink,
				draw, thick, 
				fill, 
				minimum size=5mm, 
				inner sep=1pt,
				circle, 
				] (Tx) at (c.85) {};
				\node[scale = 1.2]  at (d.85) {$T(x)$};
				\draw [line width=1pt,-, dashed] (x) -- (Tx);
				\draw [line width=1pt,-, dashed] (y) -- (Ty);
				
				\end{tikzpicture}
				\caption{A transformation which is not locally monotonic at points of continuity is not crossing.}
				\label{fig:localMonotInc}
			\end{subfigure}
			\caption{}
		\end{figure}
		By the crossing condition we must have, 
		$$ A_{w_xw_y}\subseteq T^{-1}(A_{xy}\setminus \set{x, y }) \subseteq A_{w_xw_y}\cup \set{T(x),T(y)},$$ 
		which implies 
		$$\mu(A_{w_xw_y}) = \mu(T^{-1}A_{xy}).$$
		Moreover, since $\mu$ is $T$-compatible we have
		$$\mu(A_{xy}) + \mu(A_{yw_x}) = \frac{1}{2}= \mu(A_{w_yx} + \mu(A_{w_xw_y})$$
		and 
		$$\mu(A_{w_yx}) +\mu(A_{xy}) = \frac{1}{2}=\mu(A_{yw_x}) +  \mu(A_{w_xw_y})$$
		which implies $\mu(A_{xy}) = \mu(A_{w_xw_y}) = \mu(T^{-1} A_{xy})$.
	\end{proof}
	\begin{observation}\label{obs:crossingCountInvol}
		Every continuous crossing mapping $T:S^1\to S^1$ is an involution.
	\end{observation}
	\begin{proof}
		Suppose towards a contradiction that $y= T\circ T(x)\neq x$ for some $x\in S^{-1}$. We show that $T$ is not continuous at $w = T(x)$. The points $x$, $y$ and $w$ split $S^1$ into three arcs $A_{xy}$, $A_{yw}$ and $A_{wx}$, \Cref{fig:CircleCrossing2}. By the crossing condition, $T(z) \in A_{wx}$ for every $z\in A_{yw}$. Consider a sequence of points in $A_{yw}$ which converges to $w$ to conclude that $T$ is discontinuous at $w$. 
	\end{proof}
	Identifying, $S^1$ with $[0,1]/(0\sim 1)$, we say a function $f:S^1\to S^1$ is locally monotonic increasing at $x\in S^1$, if for all $\varepsilon >0$ there exists $\delta>0$ s.t. $f([x,x+\delta)) \subseteq [f(x), f(x) + \varepsilon)$.
	\begin{observation}\label{obs:localMono}
		Let $T:S^1\to S^1$ be a crossing function and suppose that $T$ is continuous at $x\in S^1$. Then $T$ is locally monotonic increasing at $x$.
	\end{observation}
	\begin{proof}
		Suppose that $T$ is not locally monotonic increasing at $x$. Since $T$ is continuous at $x$ this means we can find small $\varepsilon >0$, $\delta>0$ such that $f(x+\delta) = f(x) - \varepsilon$. But this implies that $T$ is not crossing, see \Cref{fig:localMonotInc}
	\end{proof}

	\begin{proposition}\label{prop:TcompatMeas}
		Let $T:S^1\to S^1$ be a $C^1$ crossing map and  $f:S^1\to [0, \infty)$ a continuous map. The measure defined by $d\mu(x) = f(x)dx$, where $dx$ denotes the Lebesgue measure, is $T$-compatible if and only if for all $y\in S^1$
		$$(f\circ T (y))\cdot  T'(y)=  f(y).$$
	\end{proposition}
	\begin{proof}
		First we note, since $T$ is continuous and crossing for all $x,y\in S^1$, $T([x,y]) = [T(x), T(y)].$ Now assume that $\mu$ is $T$-compatible. Then by \Cref{prop:MeasPresCrossing} $\mu$ is $T$ invariant. Therefore,
		\begin{equation*}
		\begin{split}
		\int_{y}^{y+\delta} f(x) dx & = \mu([y,y+\delta])\\
		&=\mu(T^{-1}([y,y+\delta]))\\
		&=\mu(T([y,y+\delta]))\\
		&=\mu([T(y),T(y+\delta)]))\\
		&=\int_{T(y)}^{T(y+\delta)} f(x) dx
		\end{split}
		\end{equation*}
		Since $T$ is $C^1$, $T(y+\delta) = T(y)+ \delta T'(y)+ o(\delta)$.
		Moreover,  we are dealing with continuous functions and therefore
		$$\int_{y}^{y+\delta} f(x) dx = f(y)\delta + o(\delta)$$
		and 
		$$\int_{T(y)}^{T(y+\delta)} f(x) dx=\int_{T(y)}^{T(y) + \delta T'(y)+ o(\delta)} f(x) dx = f\circ T(y) \cdot T'(y)\delta   + o(\delta).$$
		Taking $\delta \to 0$ we find 
		$$f(y) = f\circ T(y) \cdot T'(y).$$
		Now assume $f(y) = f\circ T(y) \cdot T'(y)$ for all $y\in S^1$, and let $a,b \in S^1$. Then
		\begin{equation*}
		\begin{split}
		\mu(T^{-1}([a,b]))&  = \mu(T([a,b]))\\
		& = \mu([T(a),T(b)]) \\
		& = \int_{T(a)}^{T(b)} f(x) dx\\
		&  \stackrel{\mathclap{x\to T(x)}}{=} \hspace{2mm} \int_{T^2(a)}^{T^2(b)} f\circ T(x)\cdot T'(x) dx\\
		& =\int_{a}^{b} f(x) dx\\
		& = \mu([a,b])
		\end{split}
		\end{equation*}
		The second to last line follows from the fact that $T$ is an involution and $f(y) = f\circ T(y) \cdot T'(y)$.
		That $\mu$ is $T$-compatible follows from \Cref{prop:MeasPresCrossing}.
	\end{proof}
	\begin{corollary}
		Let $T:S^1\to S^1$ be a $C^1$ crossing map. Then the measure $\mu$ defined by $d\mu  = (T'(x))^{1/2} dx$ is $T$-compatible.
	\end{corollary}
	\begin{proof}
		Observe that 
		$$
		1 = [x] ' =[T\circ T(x) ] '= T'(x) \cdot T'\circ T(x) \implies T'(x) = \frac{1}{T'\circ T(x)}.
		$$
		so that $T'$ is non-zero. Since by \Cref{obs:localMono} $T'\geq 0$, $T'$ is strictly positive and $\mu$ is a well defined positive measure. Moreover, 
		$$(T'(x))^{1/2} = (T'(x))^{-1/2}\cdot T'(x) =  (T'\circ T(x))^{1/2}\cdot T'(x).$$
		The rest follows \Cref{prop:TcompatMeas}, taking $f(x)= (T'(x))^{1/2}$.
	\end{proof}
	
	\section{The Computational Perspective}\label{sec:computation}
	In this section we discuss the computational complexity of the following decision problems:
	\begin{itemize}
		\item (Path System Metrizability) Decide if a given path system $\mathcal{P}$ in a graph $G$ is metrizable.
		\item (Graph Metrizability) Decide if a given graph $G$ is metrizable.
	\end{itemize}
	It is also of interest to determine if a path system/graph is {\em strictly} metrizable. 
	\begin{theorem}\label{thm:PathSystemMetPolynomial}
		Path System Metrizability and Strict Path System Metrizability can be decided in polynomial time.
	\end{theorem}

	The strict case of \Cref{thm:PathSystemMetPolynomial} is already proven in \Cite{Bo}. In that paper, Bodwin characterizes strict metrizability in terms of flow and uses that characterization to obtain a procedure which decides whether or not a path system is strictly metrizable by solving a linear program with only polynomially many constraints. Moreover, it seems likely that a similar approach can be used to deal with the non-strict version of the problem. However, our approach is different and builds on the classical theory of Gr\"{o}tschel, Lov\'{a}sz and Schrijver \Cite{GLS}.
	Let us recall some basic definitions from that theory. A {\em strong separation oracle} for a polyhedron $K \subseteq \R^n$ receives as input a point $x\in \Q^n$ and either asserts that $x\in K$ or returns a vector $c\in \Q^n$ s.t.\ $c^Tx < c^Ty$ for all $y\in K$. The {\em encoding length} of an integer $s$ or a simplified fraction $q=\frac{s}{t}$ is the least number of bits needed to express $s$ resp.\ $q$. The encoding length of a vector or a matrix the is the sum over their entries. Here is our main tool:
	\begin{theorem}[\Cite{GLS}]\label{thm:ellipsoid}
		Suppose that the polyhedron $K=\set{x\in \R^n : Ax\leq b}$ has a strong separation oracle, where $A \in M_{m\times n} (\Q)$, $b\in \Q^m$. If each of $K$'s defining inequalities $<a_i,x> \le b_i$ has encoding length $\le \varphi$, then it is possible to determine whether or not $K$ is empty in time $\text{poly}(n,\varphi)$, using the ellipsoid algorithm. 
	\end{theorem}
	\begin{proof}{[\Cref{thm:PathSystemMetPolynomial}]}
		We first deal with non-strict metrizability. Let $\mathcal{P}$ be a path system in $G$, and let $\mathcal{Q}_{u,v}$ denote the collection of all the simple $uv$-paths in $G$ not equal to $P_{u,v} \in \mathcal{P}$. Then
		$$A_{u,v} \coloneqq \set{x\in \R^E :~x>0,~ \forall Q \in \mathcal{Q}_{u,v}, \ \sum_{ e\in P_{u,v}} x_e - \sum_{e\in Q} x_e \leq 0 }$$
		is the collection of all positive edge weights which induce $P_{u,v}$ as a $uv$ geodesic. But clearly $x \in \R^E$ and $\alpha x$ induce the same path system for any $\alpha >0$. So $\mathcal{P}$ is metrizable iff it is induced by some $x\geq 1$. Let $B \coloneqq \set{x\in \R^E : x_e \geq 1 \text{~for every~}e\in E}$. Therefore $\mathcal{P}$ is metrizable if and only if the polyhedron
		$$K = \bigcap_{u,v\in V} A_{u,v} \cap B$$
		is not empty. Also, $\varphi\le O(n)$ since all the coefficients in $K$'s defining inequalities are $1,0,-1$ and each such inequality is supported on at most $2(n-1)$ coordinates. By \Cref{thm:ellipsoid} all that remains is to find a poly-time strong separation oracle for $K$. On input $w\in \Q^E$ we need to decide whether $w\in K$, and if not, provide a violated inequality. If $w\not\in B$ then one of the inequalities $x_e \geq 1$ is violated, so let us assume $w\in B$. We calculate the distance $d_w(u,v)$ for each pair of vertices $u,v\in V$. If $w(P_{u,v}) = d_w(u,v)$ for each $u,v \in V$ then $w\in K$. Otherwise there exists $u,v \in V$ and a $uv$ path $Q$ in $G$ s.t.\ $w(Q)  = d_{w}(u,v) < w(P_{u,v})$, which means that the inequality $\sum_{ e\in P_{u,v}} x_e - \sum_{e\in Q} x_e \leq 0$ is violated.\\
		For the strict case define $K$ in the same way except for $u,v\in V$ we set
		$$A_{u,v} \coloneqq \set{x\in \R^E : \forall Q \in \mathcal{Q}_{u,v}, \ \sum_{ e\in P_{u,v}} x_e - \sum_{e\in Q} x_e \leq -1 }.$$
		But now we need to verify, given $w\in \Q^E$ not only that $w(P_{u,v}) = d_w(u,v)$ but that $P_{u,v}$ is the unique shortest path between $u$ and $v$. If this is not the case, and $w(P_{u,v}) = d_w(u,v) = w(Q)$ for some $uv$ path $Q\neq P_{u,v}$, then the inequality $\sum_{ e\in P_{u,v}} x_e - \sum_{e\in Q} x_e \leq -1$ is violated. To find such $Q$ if one exists, we calculate $\min d_w(u,z)+d_w(z,v)$ over all vertices $z\not\in P_{u,v}$.
	\end{proof}
	It follows from Robertson and Seymour's forbidden minor theory that the metrizability of graphs can be efficiently decided.
	\begin{theorem}\label{thm:graphMetPoly}
		Graph Metrizability and Strict Graph Metrizability can be decided in polynomial time.
	\end{theorem}
	Here is what we need from the graph minors theory. 
	\begin{theorem}[\cite{RS}]\label{thm:RS}
		Fix a graph $H$. It can be decided in polynomial time whether a given graph $G$ contains a subdivision of $H$.
	\end{theorem}
	\begin{proof}{[\Cref{thm:graphMetPoly}]}
		This follows immediately from \Cref{thm:RS}: The graph $G$ is not metrizable if and only if it contains a subdivision of some $H\in \mathcal{F}_M$, and the set $\mathcal{F}_M$ is finite, by \Cref{thm:finiteMinGraphs}. 
	\end{proof}
	
	\subsection{Poly-time vs.\ practical algorithms}\label{subsec:praktisch}
	While \Cref{thm:graphMetPoly} proves the existence of a polynomial time algorithm to decide graph metrizability, we still lack a practical algorithm that achieves this. There are several reasons for this lacuna. While we know that $\mathcal{F}_M$ and $\mathcal{F}_{SM}$ are both finite, we are far from having the complete catalog. Even if we get to know the entire list of these minimal graphs, it is not inconceivable that the sheer size of these sets makes the algorithm in the proof of \Cref{thm:graphMetPoly} impractical. Thus, the search of a viable algorithm to decide metrizability is still on. There are really two problems at hand. It is reasonable to expect (but we do not know whether or not this is the case) that current LP solvers can practically find a certificate for the non-metrizability of a given non-metrizable path system. In some cases reported throughout the paper such a certificate was found by hand. In contrast, we have only a brute force algorithm\footnote{\url{https://github.com/dcizma1/testing-graph-metrizability}} to prove that a given graph $G$ is metrizable. Namely it generates all possible consistent path systems in $G$ and checks each for metrizability using a linear program as in the proof of \Cref{thm:PathSystemMetPolynomial}. Needless to say, this is practical only with small graphs. Indeed, this is how we found the non-metrizable graphs in \Cref{fig:MinimumGraphs} as well as path systems realizing their non-metrizability. Note, however, that we do have humanly verifiable proofs that all the graph in \Cref{fig:MinimumGraphs} are non-metrizable. These proofs can be found in  \Cref{append:certificates}.

	\section{Open problems}
	
	This paper suggests numerous open problems and new avenues of research. Below we list some of those.
	
	Here is the issue that we consider most pressing. We have seen throughout the paper several certificates that certain graphs are non-metrizable. These proofs proceed by comparing the weights of chosen paths to alternative ones. These inequalities are then combined to conclude that certain edge weights are non-positive. 
	
	\begin{open}
		Do there exist humanly verifiable certificates that certain graphs are metrizable?
	\end{open}
	
\begin{open}
Can it be decided in polynomial time whether a given consistent partial path system can be extended to a full consistent path system?
\end{open}
	
\begin{open}
The graph $\Theta_{a,b,c}$ has two vertices of degree $3$ that are connected by three openly disjoint paths of $a,b,c$ edges respectively. By \Cref{cor:outerPlanarMet} $\Theta_{a,b,c}$ is metrizable when $\min\{a,b,c\}=1$. Also, $\Theta_{3,3,4}\in \mathcal{F}_{M}$ (see \Cref{fig:graph11}). However, we do not know whether $\Theta_{a,b,c}$ is metrizable or not when $\min\{a,b,c\}=2$ or $a=b=c=3$.
\end{open}
	
Can we quantify the level of a graph's non-metrizability? Let $\Pi_G$ be the collection of all consistent path systems in $G$, and let $\mathcal{M}_G\subseteq \Pi_G$ be the collection of all those which are metrizable. Whether or not $G$ is metrizable is expressed by this inclusion being proper or not.
	
\begin{open}
We suspect that there exist $n$-vertex graphs $G$ for which $|\mathcal{M}_G|=o_n(|\Pi_G|)$. Actually we even believe that this is the case for {\em most} graphs.
\end{open}

\begin{open}
Associated with every connected graph $G$ is $\cal{A}_G$, a hyperplane arrangement in $\mathbb{R}^{E(G)}$ which encodes a lot of information on metrizable path systems in $G$. If  $P$ and $Q$ are openly disjoint paths between the same two vertices in $G$, then the hyperplane $\{\sum_{e\in P}x_e =\sum_{f\in Q}x_f\}$ is in $\cal{A}_G$. It would be interesting to investigate the basic features of such arrangements.
\end{open}

\begin{open}
Find a complete list of the graphs in $\mathcal{F}_{M}$ and $\mathcal{F}_{SM}$. 
\end{open}

\begin{open}
It makes sense to speak of consistent path systems in $1$-dimensional CW complexes. The case of $S^1$ viewed as a $1$-dimensional CW complex with a single vertex and a single edge is considered in \Cref{sec:continuous}. Is there an interesting theory of metrizability in this broader context?
\end{open}

\begin{open}
Which graphs $G=(V,E)$ have the property that every {\em neighborly} consistent path system is metrizable? I.e., we assume that for every $xy\in E$, this edge is the chosen $xy$-geodesic.
\end{open}

\begin{figure}
		\centering
		\begin{subfigure}{0.3\textwidth}
			\centering
			\resizebox{\textwidth}{!}{
				\begin{tikzpicture}
				
				\node[draw,circle,fill] (2) at (1*360/5 +90: 5cm) {$2$};
				\node[draw,circle,fill]  (3) at (2*360/5 +90: 5cm) {$3$};
				\node[draw,circle,fill] (4) at (3*360/5 +90: 5cm) {$4$};
				\node[draw,circle,fill] (5) at (4*360/5 +90: 5cm) {$5$};
				\node[draw,circle,fill] (1) at ($(2)!0.5!(5)$) {$1$};
				\node[draw,circle,fill] (6) at ($(2)!0.5!(5) + (0,-3)$) {$6$};
				\node[draw,circle,fill] (7) at  ($(2)!0.5!(5) + (0,3)$) {$7$};

				\draw [line width=3pt,-] (1) -- (2);
				\draw [line width=3pt,-] (2) -- (3);
				\draw [line width=3pt,-] (3) -- (4);
				\draw [line width=3pt,-] (4) -- (5);
				\draw [line width=3pt,-] (5) -- (1);
				\draw [line width=3pt,-] (2) -- (6);
				\draw [line width=3pt,-] (5) -- (6);
				\draw [line width=3pt,-] (2) -- (7);
				\draw [line width=3pt,-] (5) -- (7);
				
				\end{tikzpicture}
			}
			\caption{Graph 1}
			\label[graph]{fig:graph1}
		\end{subfigure}
		\hfill
		\begin{subfigure}{0.3\textwidth}
			\centering
			\resizebox{\textwidth}{!}{
				\begin{tikzpicture}

				\node[draw,circle,fill] (1) at (0*360/5 +90: 5cm) {$1$};
				\node[draw,circle,fill] (2) at (1*360/5 +90: 5cm) {$2$};
				\node[draw,circle,fill] (3) at (2*360/5 +90: 5cm) {$3$};
				\node[draw,circle,fill] (4) at (3*360/5 +90: 5cm) {$4$};
				\node[draw,circle,fill] (5) at (4*360/5 +90: 5cm) {$5$};
				
				\node[draw,circle,fill] (6) at ($(1)!0.5!(3)$) {$6$};
				\node[draw,circle,fill] (7) at  ($(1)!0.5!(4)$) {$7$};

				\draw [line width=3pt,-] (1) -- (2);
				\draw [line width=3pt,-] (2) -- (3);
				\draw [line width=3pt,-] (3) -- (4);
				\draw [line width=3pt,-] (4) -- (5);
				\draw [line width=3pt,-] (5) -- (1);
				\draw [line width=3pt,-] (1) -- (6);
				\draw [line width=3pt,-] (3) -- (6);
				\draw [line width=3pt,-] (1) -- (7);
				\draw [line width=3pt,-] (4) -- (7);
				
				\end{tikzpicture}
			}
			\caption{Graph 2}
			\label[graph]{fig:graph2}
		\end{subfigure}
		\hfill
		\begin{subfigure}{0.3\textwidth}
			\centering
			\resizebox{\textwidth}{!}{
				\begin{tikzpicture}

				\node[draw,circle,fill] (1) at (0*360/5 +90: 5cm) {$1$};
				\node[draw,circle,fill]  (2) at (1*360/5 +90: 5cm) {$2$};
				\node[draw,circle,fill] (3) at (2*360/5 +90: 5cm) {$3$};
				\node[draw,circle,fill]  (4) at (3*360/5 +90: 5cm) {$4$};
				\node[draw,circle,fill] (5) at (4*360/5 +90: 5cm) {$5$};
				
				\node[draw,circle,fill] (6) at ($(1)!0.5!(3)$) {$6$};
				\node[draw,circle,fill] (7) at  ($(1)!0.5!(4)$) {$7$};

				\draw [line width=3pt,-] (1) -- (2);
				\draw [line width=3pt,-] (2) -- (3);
				\draw [line width=3pt,-] (3) -- (4);
				\draw [line width=3pt,-] (4) -- (5);
				\draw [line width=3pt,-] (5) -- (1);
				\draw [line width=3pt,-] (6) -- (7);
				\draw [line width=3pt,-] (2) -- (6);
				\draw [line width=3pt,-] (5) -- (7);
				\draw [line width=3pt,-] (3) -- (7);
				\draw [line width=3pt,-] (4) -- (6);
				
				\end{tikzpicture}
			}
			\caption{Graph 3}
			\label[graph]{fig:graph3}
		\end{subfigure}
		\\
		\begin{subfigure}{0.3\textwidth}
			\centering
			\resizebox{\textwidth}{!}{
				\begin{tikzpicture}

				\node[draw,circle,fill] (1) at (0*360/5 +90: 5cm) {$1$};
				\node[draw,circle,fill]  (2) at (1*360/5 +90: 5cm) {$2$};
				\node[draw,circle,fill] (3) at (2*360/5 +90: 5cm) {$3$};
				\node[draw,circle,fill] (4) at (3*360/5 +90: 5cm) {$4$};
				\node[draw,circle,fill]  (5) at (4*360/5 +90: 5cm) {$5$};
				
				\node (m1) at ($(2)!.5! (3)$) {};
				\node (m2) at ($(4)!.5! (5)$) {};
				\node[draw,circle,fill] (6) at ($(m1)!.8!270:(2)$) {$6$};
				\node[draw,circle,fill] (7) at  ($(m2)!.8!90:(5)$) {$7$};

				\draw [line width=3pt,-] (1) -- (2);
				\draw [line width=3pt,-] (2) -- (3);
				\draw [line width=3pt,-] (3) -- (4);
				\draw [line width=3pt,-] (4) -- (5);
				\draw [line width=3pt,-] (5) -- (1);
				\draw [line width=3pt,-] (6) -- (7);
				\draw [line width=3pt,-] (2) -- (6);
				\draw [line width=3pt,-] (5) -- (7);
				\draw [line width=3pt,-] (3) -- (6);
				\draw [line width=3pt,-] (4) -- (7);

				\end{tikzpicture}
			}
			\caption{Graph 4}
			\label[graph]{fig:graph4}
		\end{subfigure}
		\hfill
		\begin{subfigure}{0.3\textwidth}
			\centering
			\resizebox{\textwidth}{!}{
				\begin{tikzpicture}
				
				\node[draw,circle,fill] (1) at (0*360/5 +90: 5cm) {$1$};
				\node[draw,circle,fill] (2) at (1*360/5 +90: 5cm) {$2$};
				\node[draw,circle,fill] (3) at (2*360/5 +90: 5cm) {$3$};
				\node[draw,circle,fill] (4) at (3*360/5 +90: 5cm) {$4$};
				\node[draw,circle,fill] (5) at (4*360/5 +90: 5cm) {$5$};
				\node[draw,circle,fill] (6) at ($(2)!.5!(5)+(0,-1)$) {$6$};
				\node[draw,circle,fill] (7) at ($(2)!.5!(5) + (0,-3)$) {$7$};
				
				\draw [line width=3pt,-] (1) -- (2);
				\draw [line width=3pt,-] (2) -- (3);
				\draw [line width=3pt,-] (3) -- (4);
				\draw [line width=3pt,-] (4) -- (5);
				\draw [line width=3pt,-] (5) -- (1);
				\draw [line width=3pt,-] (6) -- (7);
				\draw [line width=3pt,-] (2) -- (6);
				\draw [line width=3pt,-] (5) -- (6);
				\draw [line width=3pt,-] (3) -- (7);
				\draw [line width=3pt,-] (4) -- (7);
				
				\end{tikzpicture}
			}
			\caption{Graph 5}
			\label[graph]{fig:graph5}
		\end{subfigure}
		\hfill
		\begin{subfigure}{0.3\textwidth}
			\centering
			\resizebox{\textwidth}{!}{
				\begin{tikzpicture}
				
				\node[draw,circle,fill] (1) at (0*360/7 +90: 5cm) {$1$};
				\node[draw,circle,fill] (2) at (1*360/7 +90: 5cm) {$2$};
				\node[draw,circle,fill] (3) at (2*360/7 +90: 5cm) {$3$};
				\node[draw,circle,fill] (4) at (3*360/7 +90: 5cm) {$4$};
				\node[draw,circle,fill] (5) at (4*360/7 +90: 5cm) {$5$};
				\node[draw,circle,fill] (6) at (5*360/7 +90: 5cm) {$6$};
				\node[draw,circle,fill] (7) at (6*360/7 +90: 5cm) {$7$};
				\node[draw,circle,fill] (8) at (0,0) {$8$};
				
				\draw [line width=3pt,-] (1) -- (2);
				\draw [line width=3pt,-] (2) -- (3);
				\draw [line width=3pt,-] (3) -- (4);
				\draw [line width=3pt,-] (4) -- (5);
				\draw [line width=3pt,-] (5) -- (6);
				\draw [line width=3pt,-] (6) -- (7);
				\draw [line width=3pt,-] (7) -- (1);
				\draw [line width=3pt,-] (1) -- (8);
				\draw [line width=3pt,-] (3) -- (8);
				\draw [line width=3pt,-] (6) -- (8);
				
				\end{tikzpicture}
			}
			\caption{Graph 6}
			\label[graph]{fig:graph6}
		\end{subfigure}
		\\
		\begin{subfigure}{0.3\textwidth}
			\centering
			\resizebox{\textwidth}{!}{
				\begin{tikzpicture}
				
				\node[draw,circle,fill] (1) at (0*360/5 +90: 5cm) {$1$};
				\node[draw,circle,fill] (2) at (1*360/5 +90: 5cm) {$2$};
				\node[draw,circle,fill] (3) at (2*360/5 +90: 5cm){$3$};
				\node[draw,circle,fill] (4) at (3*360/5 +90: 5cm) {$4$};
				\node[draw,circle,fill] (5) at (4*360/5 +90: 5cm)  {$5$};
				\node[draw,circle,fill] (6) at (0,0) {$6$};
				\node[draw,circle,fill] (7) at ($(1)!.5!(6)$) {$7$};
				\node[draw,circle,fill] (8) at ($(3)!.5!(4)$) {$8$};
				
				\draw [line width=3pt,-] (1) -- (2);
				\draw [line width=3pt,-] (1) -- (5);
				\draw [line width=3pt,-] (1) -- (7);
				\draw [line width=3pt,-] (2) -- (3);
				\draw [line width=3pt,-] (3) -- (8);
				\draw [line width=3pt,-] (4) -- (5);
				\draw [line width=3pt,-] (4) -- (8);
				\draw [line width=3pt,-] (3) -- (6);
				\draw [line width=3pt,-] (4) -- (6);
				\draw [line width=3pt,-] (6) -- (7);
				
				\end{tikzpicture}
			}
			\caption{Graph 7}
			\label[graph]{fig:graph7}
		\end{subfigure}
		\hfill
		\begin{subfigure}{0.3\textwidth}
			\centering
			\resizebox{\textwidth}{!}{
				
				\begin{tikzpicture}
				
				\node[draw,circle,fill] (1) at (0*360/5 +90: 5cm) {$1$};
				\node[draw,circle,fill] (2) at (1*360/5 +90: 5cm) {$2$};
				\node[draw,circle,fill] (3) at (2*360/5 +90: 5cm){$3$};
				\node[draw,circle,fill] (4) at (3*360/5 +90: 5cm) {$4$};
				\node[draw,circle,fill] (5) at (4*360/5 +90: 5cm)  {$5$};
				\node[draw,circle,fill] (6) at (0,0) {$6$};
				\node[draw,circle,fill] (7) at ($(1)!.33!(6)$) {$7$};
				\node[draw,circle,fill] (8) at ($(1)!.66!(6)$) {$8$};
				
				\draw [line width=3pt,-] (1) -- (2);
				\draw [line width=3pt,-] (1) -- (5);
				\draw [line width=3pt,-] (1) -- (7);
				\draw [line width=3pt,-] (2) -- (3);
				\draw [line width=3pt,-] (3) -- (4);
				\draw [line width=3pt,-] (4) -- (5);
				\draw [line width=3pt,-] (6) -- (7);
				\draw [line width=3pt,-] (7) -- (8);
				\draw [line width=3pt,-] (3) -- (6);
				\draw [line width=3pt,-] (4) -- (6);
				
				\end{tikzpicture}
			}
			\caption{Graph 8}
			\label[graph]{fig:graph8}
		\end{subfigure}
		\hfill
		\begin{subfigure}{0.3\textwidth}
			\centering
			\resizebox{\textwidth}{!}{
				\begin{tikzpicture}

				\node[draw,circle,fill] (1) at (0*360/5 +90: 5cm) {$1$};
				\node[draw,circle,fill]  (2) at (1*360/5 +90: 5cm) {$2$};
				\node[draw,circle,fill] (3) at (2*360/5 +90: 5cm) {$3$};
				\node[draw,circle,fill] (4) at (3*360/5 +90: 5cm) {$4$};
				\node[draw,circle,fill]  (5) at (4*360/5 +90: 5cm) {$5$};
				\node[draw,circle,fill]  (6) at (0,0) {$6$};
				\node[draw,circle,fill]  (7) at ($(3)!0.5!(6)$) {$7$};
				\node[draw,circle,fill] (8) at ($(4)!0.5!(6)$) {$8$};

				\draw [line width=3pt,-] (1) -- (2);
				\draw [line width=3pt,-] (1) -- (5);
				\draw [line width=3pt,-] (1) -- (6);
				\draw [line width=3pt,-] (2) -- (3);
				\draw [line width=3pt,-] (3) -- (4);
				\draw [line width=3pt,-] (3) -- (7);
				\draw [line width=3pt,-] (4) -- (5);
				\draw [line width=3pt,-] (4) -- (8);
				\draw [line width=3pt,-] (6) -- (7);
				\draw [line width=3pt,-] (6) -- (8);

				\end{tikzpicture}
			}
			\caption{Graph 9}
			\label[graph]{fig:graph9}
		\end{subfigure}
		\\
		\begin{subfigure}{0.3\textwidth}
			\centering
			\resizebox{\textwidth}{!}{
				\begin{tikzpicture}
				
				\node[draw,circle,fill] (1) at (0*360/6 +120: 5cm) {$1$};
				\node[draw,circle,fill] (2) at (1*360/6 +120: 5cm) {$2$};
				\node[draw,circle,fill] (3) at (2*360/6 +120: 5cm) {$3$};
				\node[draw,circle,fill] (5) at (3*360/6 +120: 5cm) {$5$};
				\node[draw,circle,fill] (6) at (4*360/6 +120: 5cm) {$6$};
				\node[draw,circle,fill] (7) at (5*360/6 +120: 5cm) {$7$};
				\node[draw,circle,fill] (4) at ($(3)!.5!(5)$) {$4$};
				\node[draw,circle,fill] (8) at ($(1)!.5!(3)$) {$8$};
				\node[draw,circle,fill] (9) at ($(5)!.5!(7)$) {$9$};
				
				\draw [line width=3pt,-] (1) -- (2);
				\draw [line width=3pt,-] (1) -- (7);
				\draw [line width=3pt,-] (1) -- (8);
				\draw [line width=3pt,-] (2) -- (3);
				\draw [line width=3pt,-] (3) -- (4);
				\draw [line width=3pt,-] (3) -- (8);
				\draw [line width=3pt,-] (4) -- (5);
				\draw [line width=3pt,-] (5) -- (6);
				\draw [line width=3pt,-] (5) -- (9);
				\draw [line width=3pt,-] (6) -- (7);
				\draw [line width=3pt,-] (7) -- (9);
				
				\end{tikzpicture}
			}
			\caption{Graph 10}
			\label[graph]{fig:graph10}
		\end{subfigure}
		\hfill
		\begin{subfigure}{0.3\textwidth}
			\centering
			\resizebox{\textwidth}{!}{
				\begin{tikzpicture}
				
				\node[draw,circle,fill] (1) at (0*360/6 +120: 5cm) {$1$};
				\node[draw,circle,fill] (2) at (1*360/6 +120: 5cm) {$2$};
				\node[draw,circle,fill] (3) at (2*360/6 +120: 5cm) {$3$};
				\node[draw,circle,fill] (4) at (3*360/6 +120: 5cm) {$4$};
				\node[draw,circle,fill] (5) at (4*360/6 +120: 5cm) {$5$};
				\node[draw,circle,fill] (6) at (5*360/6 +120: 5cm) {$6$};
				\node[draw,circle,fill] (7) at ($(2)!.25!(5)$) {$7$};
				\node[draw,circle,fill] (8) at ($(2)!.5!(5)$) {$8$};
				\node[draw,circle,fill] (9) at ($(2)!.75!(5)$) {$9$};
				
				\draw [line width=3pt,-] (1) -- (2);
				\draw [line width=3pt,-] (1) -- (6);
				\draw [line width=3pt,-] (2) -- (3);
				\draw [line width=3pt,-] (2) -- (7);
				\draw [line width=3pt,-] (3) -- (4);
				\draw [line width=3pt,-] (4) -- (5);
				\draw [line width=3pt,-] (5) -- (6);
				\draw [line width=3pt,-] (5) -- (9);
				\draw [line width=3pt,-] (7) -- (8);
				\draw [line width=3pt,-] (8) -- (9);
				
				\end{tikzpicture}
			}
			\caption{Graph 11}
			\label[graph]{fig:graph11}
		\end{subfigure}
		\caption{Currently known topologically minimal non-metrizable graphs}
		\label{fig:MinimumGraphs}
	\end{figure}
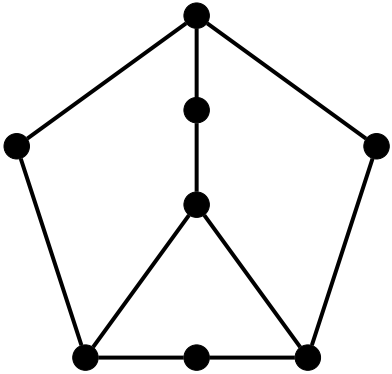
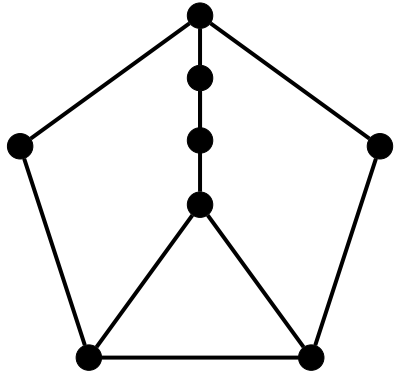
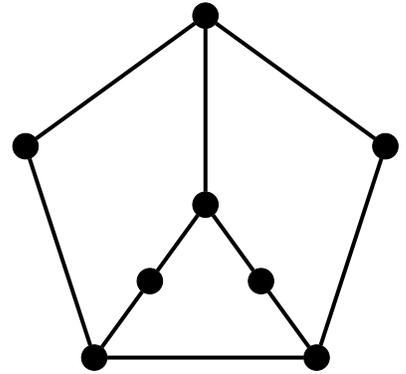
	\newpage
	\appendix
	\section{Certificates of non-metrizability}\label{append:certificates}
	For each graph $G$ in \Cref{fig:MinimumGraphs} we give a path system in $G$ along with a system of inequalities a weight function inducing this path system must satisfy. In each case, these inequalities imply at least one edge in the graph must have a non-positive weight, showing the graph in not metrizable.\\
	\vspace{5mm}
	
	\noindent
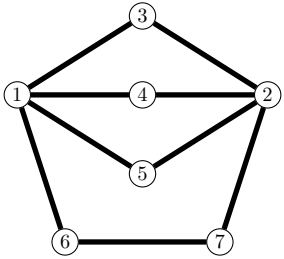
	\begin{minipage}[c]{0.4\textwidth}
		\centering
		\begin{tikzpicture}[scale=0.35, every node/.style={scale=0.35}]

		\node[draw,circle,minimum size=.5cm,inner sep=1pt] (2) at (1*360/5 +90: 5cm) [scale = 2]{$1$};
		\node[draw,circle,minimum size=.5cm,inner sep=1pt] (3) at (2*360/5 +90: 5cm)[scale = 2] {$6$};
		\node[draw,circle,minimum size=.5cm,inner sep=1pt] (4) at (3*360/5 +90: 5cm) [scale = 2]{$7$};
		\node[draw,circle,minimum size=.5cm,inner sep=1pt] (5) at (4*360/5 +90: 5cm) [scale = 2]{$2$};
		\node[draw,circle,minimum size=.5cm,inner sep=1pt] (1) at ($(2)!0.5!(5)$) [scale = 2]{$4$};
		\node[draw,circle,minimum size=.5cm,inner sep=1pt] (6) at ($(2)!0.5!(5) + (0,-3)$)[scale = 2] {$5$};
		\node[draw,circle,minimum size=.5cm,inner sep=1pt] (7) at  ($(2)!0.5!(5) + (0,3)$)[scale = 2] {$3$};

		\draw [line width=2pt,-] (1) -- (2);
		\draw [line width=2pt,-] (2) -- (3);
		\draw [line width=2pt,-] (3) -- (4);
		\draw [line width=2pt,-] (4) -- (5);
		\draw [line width=2pt,-] (5) -- (1);
		\draw [line width=2pt,-] (2) -- (6);
		\draw [line width=2pt,-] (5) -- (6);
		\draw [line width=2pt,-] (2) -- (7);
		\draw [line width=2pt,-] (5) -- (7);
		
		\end{tikzpicture}
	\end{minipage}
	\begin{minipage}[c]{0.4\textwidth}
		\centering
		\small
		\begin{equation*}
		\begin{gathered}
		(132), \ (13),\ (14), \ (15), \ (16),\ (167), \   (23), \\
		(24), \ (25), \ (276), \ (27),\ (314) \ (325),\ (316),\\
		(3167), \ (415),\ (4276),\ (427),  \  (516),\ (527),\ (67)
		\end{gathered}
		\end{equation*}
	\end{minipage}
	\\
	\vspace{4mm}
	\begin{minipage}{0.45\textwidth}
		\small
		\begin{equation*}
		\begin{split}
		w_{2,3} + w_{2,5} & \leq w_{1,3} + w_{1,5}\\
		w_{1,4} + w_{1,5} & \leq w_{2,4} + w_{2,5}\\
		w_{2,4} + w_{2,7} + w_{6,7} & \leq w_{1,4} + w_{1,6}\\
		w_{1,3} + w_{1,6} + w_{6,7} & \leq w_{2,3} + w_{2,7}
		\end{split}
		\end{equation*}
		
	\end{minipage}
	$\implies $
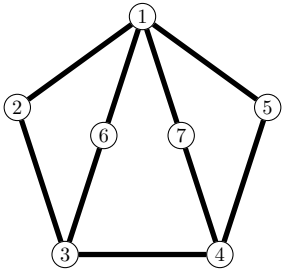
	\begin{minipage}{0.2\textwidth}
		\small
		$$w_{6,7} \leq 0$$
	\end{minipage}
	\begin{center}
		\line(1,0){400}
	\end{center}
	\begin{minipage}[c]{0.4\textwidth}
		\centering
		\begin{tikzpicture}[scale=0.35, every node/.style={scale=0.35}]

		\node[draw,circle,minimum size=.5cm,inner sep=1pt] (1) at (0*360/5 +90: 5cm) [scale=2]{$1$};
		\node[draw,circle,minimum size=.5cm,inner sep=1pt] (2) at (1*360/5 +90: 5cm) [scale=2]{$2$};
		\node[draw,circle,minimum size=.5cm,inner sep=1pt] (3) at (2*360/5 +90: 5cm) [scale=2]{$3$};
		\node[draw,circle,minimum size=.5cm,inner sep=1pt] (4) at (3*360/5 +90: 5cm) [scale=2]{$4$};
		\node[draw,circle,minimum size=.5cm,inner sep=1pt] (5) at (4*360/5 +90: 5cm) [scale=2]{$5$};
		
		\node[draw,circle,minimum size=.5cm,inner sep=1pt] (6) at ($(1)!0.5!(3)$)[scale=2] {$6$};
		\node[draw,circle,minimum size=.5cm,inner sep=1pt] (7) at  ($(1)!0.5!(4)$) [scale=2]{$7$};

		\draw [line width=2pt,-] (1) -- (2);
		\draw [line width=2pt,-] (2) -- (3);
		\draw [line width=2pt,-] (3) -- (4);
		\draw [line width=2pt,-] (4) -- (5);
		\draw [line width=2pt,-] (5) -- (1);
		\draw [line width=2pt,-] (1) -- (6);
		\draw [line width=2pt,-] (3) -- (6);
		\draw [line width=2pt,-] (1) -- (7);
		\draw [line width=2pt,-] (4) -- (7);
		
		\end{tikzpicture}
	\end{minipage}
	\begin{minipage}[c]{0.4\textwidth}
		\small
		\begin{equation*}
		\begin{gathered}
		(12), \ (163), \ (174), \ (15), \ (16), \ (17), \ (23) \\
		(234), \ (2345), \ (216), \ (217), \ (34), \ (345), \ (36)\\
		(347), \ (45), \ (436), (47), \ (5436), \ (517), \ (6347)
		\end{gathered}
		\end{equation*}
	\end{minipage}
	\\
	\vspace{4mm}
	\begin{minipage}{0.45\textwidth}
		\small
		\begin{equation*}
		\begin{split}
		w_{2,3} + w_{3,4} +w_{4,5}& \leq w_{1,2} + w_{1,5}\\
		w_{1,2} + w_{1,6} & \leq w_{2,3} + w_{3,6}\\
		w_{1,5} + w_{1,7} & \leq w_{4,5} + w_{4,7}\\
		w_{3,6} + w_{3,4} + w_{4,7} & \leq w_{1,6} + w_{1,7}
		\end{split}
		\end{equation*}
	\end{minipage}
	$\implies $
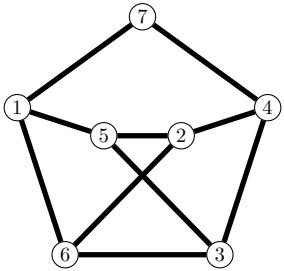
	\begin{minipage}{0.2\textwidth}
		\small $$w_{3,4} \leq 0$$
	\end{minipage}
	\begin{center}
		\line(1,0){400}
	\end{center}
	\begin{minipage}[c]{0.4\textwidth}
		\centering
		\begin{tikzpicture}[scale=0.35, every node/.style={scale=0.35}]

		\node[draw,circle,minimum size=.5cm,inner sep=1pt] (1) at (0*360/5 +90: 5cm) [scale=2]{$7$};
		\node[draw,circle,minimum size=.5cm,inner sep=1pt] (2) at (1*360/5 +90: 5cm) [scale=2]{$1$};
		\node[draw,circle,minimum size=.5cm,inner sep=1pt] (3) at (2*360/5 +90: 5cm)[scale=2] {$6$};
		\node[draw,circle,minimum size=.5cm,inner sep=1pt] (4) at (3*360/5 +90: 5cm) [scale=2]{$3$};
		\node[draw,circle,minimum size=.5cm,inner sep=1pt] (5) at (4*360/5 +90: 5cm) [scale=2]{$4$};
		
		\node[draw,circle,minimum size=.5cm,inner sep=1pt] (6) at ($(1)!0.5!(3)$) [scale=2]{$5$};
		\node[draw,circle,minimum size=.5cm,inner sep=1pt] (7) at  ($(1)!0.5!(4)$) [scale=2]{$2$};

		\draw [line width=2pt,-] (1) -- (2);
		\draw [line width=2pt,-] (2) -- (3);
		\draw [line width=2pt,-] (3) -- (4);
		\draw [line width=2pt,-] (4) -- (5);
		\draw [line width=2pt,-] (5) -- (1);
		\draw [line width=2pt,-] (6) -- (7);
		\draw [line width=2pt,-] (2) -- (6);
		\draw [line width=2pt,-] (5) -- (7);
		\draw [line width=2pt,-] (3) -- (7);
		\draw [line width=2pt,-] (4) -- (6);
		
		\end{tikzpicture}
	\end{minipage}
	\begin{minipage}[c]{0.4\textwidth}
		\centering
		\small
		\begin{equation*}
		\begin{gathered}
		(1742), \ (163), \ (174), \ (15),\ (16),\ (17), \\ 
		(263),\ (24),  \ (25), \ (26),\ (247), \ (34), \ (35), \ (36), \\
		(347),\ (435), \ (436), \ (47), \ (516), \ (5347), \ (617)
		\end{gathered}
		\end{equation*}
	\end{minipage}
	\\
	\vspace{5mm}
	\begin{minipage}{0.45\textwidth}
		\small
		\begin{equation*}
		\begin{split}
		w_{2,6} + w_{3,6} & \leq w_{2,4} + w_{3,4}\\
		w_{1,5} + w_{1,6} & \leq w_{3,5} + w_{3,6}\\
		w_{1,7} + w_{4,7} + w_{2,4} & \leq w_{1,6} + w_{2,6}\\
		w_{3,5} + w_{3,4} + w_{4,7} & \leq w_{1,5} + w_{1,7}
		\end{split}
		\end{equation*}
	\end{minipage}
	$\implies $
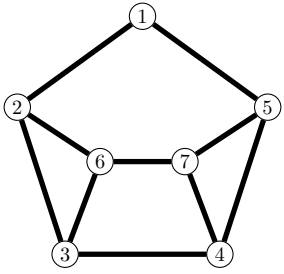
	\begin{minipage}{0.2\textwidth}
		\small
		$$w_{4,7} \leq 0$$
	\end{minipage}
	\begin{center}
		\line(1,0){400}
	\end{center}
	\begin{minipage}[c]{0.4\textwidth}
		\centering
		\begin{tikzpicture}[scale=0.35, every node/.style={scale=0.35}]

		\node[draw,circle,minimum size=.5cm,inner sep=1pt] (1) at (0*360/5 +90: 5cm)[scale =2] {$1$};
		\node[draw,circle,minimum size=.5cm,inner sep=1pt] (2) at (1*360/5 +90: 5cm)[scale =2] {$2$};
		\node[draw,circle,minimum size=.5cm,inner sep=1pt] (3) at (2*360/5 +90: 5cm)[scale =2] {$3$};
		\node[draw,circle,minimum size=.5cm,inner sep=1pt] (4) at (3*360/5 +90: 5cm)[scale =2] {$4$};
		\node[draw,circle,minimum size=.5cm,inner sep=1pt] (5) at (4*360/5 +90: 5cm)[scale =2] {$5$};
		
		\node (m1) at ($(2)!.5! (3)$) {};
		\node (m2) at ($(4)!.5! (5)$) {};
		\node[draw,circle,minimum size=.5cm,inner sep=1pt] (6) at ($(m1)!.8!270:(2)$)[scale =2] {$6$};
		\node[draw,circle,minimum size=.5cm,inner sep=1pt] (7) at  ($(m2)!.8!90:(5)$)[scale =2] {$7$};

		\draw [line width=2pt,-] (1) -- (2);
		\draw [line width=2pt,-] (2) -- (3);
		\draw [line width=2pt,-] (3) -- (4);
		\draw [line width=2pt,-] (4) -- (5);
		\draw [line width=2pt,-] (5) -- (1);
		\draw [line width=2pt,-] (6) -- (7);
		\draw [line width=2pt,-] (2) -- (6);
		\draw [line width=2pt,-] (5) -- (7);
		\draw [line width=2pt,-] (3) -- (6);
		\draw [line width=2pt,-] (4) -- (7);

		\end{tikzpicture}
	\end{minipage}
	\begin{minipage}[c]{0.4\textwidth}
		\centering
		\small
		\begin{equation*}
		\begin{gathered}
		(12), \ (123), \ (1234), \ (15), \ (126), \ (157), \ (23), \\
		(234), \ (215), \ (26), \ (2157), \ (34), \ (345), \ (36), \\
		(367), \ (45), \ (476), \ (47), \ (5126), \ (57), \ (67)
		\end{gathered}
		\end{equation*}
	\end{minipage}
	\\
	\vspace{5mm}
	\begin{minipage}{0.45\textwidth}
		\small
		\begin{equation*}
		\begin{split}
		w_{1,2} + w_{2,3} + w_{3,4} & \leq w_{1,5} +  w_{4,5}\\
		w_{1,2} + w_{1,5} + w_{5,7} & \leq w_{2,6} + w_{6,7}\\
		w_{1,5} + w_{1,2} + w_{2,6} & \leq w_{5,7} + w_{6,7}\\
		w_{3,6} + w_{6,7} & \leq w_{3,4} + w_{4,7}\\
		w_{4,7} + w_{6,7} & \leq w_{3,4} + w_{3,6}\\
		w_{3,4} + w_{4,5}  & \leq w_{1,5} + w_{1,2} + w_{2,3}
		\end{split}
		\end{equation*}
	\end{minipage}
	$\implies $
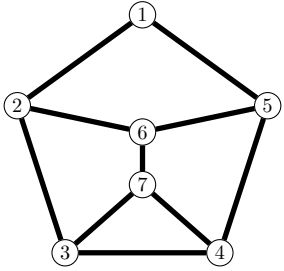
	\begin{minipage}{0.2\textwidth}
		\small
		$$w_{1,2} \leq 0$$
	\end{minipage}
	\begin{center}
		\line(1,0){400}
	\end{center}
	\begin{minipage}[c]{0.4\textwidth}
		\centering
		\begin{tikzpicture}[scale=0.35, every node/.style={scale=0.35}]

		\node[draw,circle,minimum size=.5cm,inner sep=1pt] (1) at (0*360/5 +90: 5cm)[scale=2] {$1$};
		\node[draw,circle,minimum size=.5cm,inner sep=1pt] (2) at (1*360/5 +90: 5cm)[scale=2] {$2$};
		\node[draw,circle,minimum size=.5cm,inner sep=1pt] (3) at (2*360/5 +90: 5cm)[scale=2] {$3$};
		\node[draw,circle,minimum size=.5cm,inner sep=1pt] (4) at (3*360/5 +90: 5cm)[scale=2] {$4$};
		\node[draw,circle,minimum size=.5cm,inner sep=1pt] (5) at (4*360/5 +90: 5cm)[scale=2] {$5$};
		
		\node[draw,circle,minimum size=.5cm,inner sep=1pt] (6) at ($(2)!.5!(5)+(0,-1)$)[scale=2] {$6$};
		\node[draw,circle,minimum size=.5cm,inner sep=1pt] (7) at ($(2)!.5!(5) + (0,-3)$)[scale=2] {$7$};

		\draw [line width=2pt,-] (1) -- (2);
		\draw [line width=2pt,-] (2) -- (3);
		\draw [line width=2pt,-] (3) -- (4);
		\draw [line width=2pt,-] (4) -- (5);
		\draw [line width=2pt,-] (5) -- (1);
		\draw [line width=2pt,-] (6) -- (7);
		\draw [line width=2pt,-] (2) -- (6);
		\draw [line width=2pt,-] (5) -- (6);
		\draw [line width=2pt,-] (3) -- (7);
		\draw [line width=2pt,-] (4) -- (7);

		\end{tikzpicture}
	\end{minipage}
	\begin{minipage}[c]{0.4\textwidth}
		\centering
		\small
		\begin{equation*}
		\begin{gathered}
		(12), \ (1543), \ (154), \ (15), \ (156), \ (1567), \\
		(23), \ (2154), \ (215), \ (26), \ (237), \ (34), \ (345)\\
		(326), \ (37), \ (45),\ (47), \ (476), \ (56), \ (567), \ (67)
		\end{gathered}
		\end{equation*}
	\end{minipage}
	\\
	\vspace{5mm}
	\begin{minipage}{0.45\textwidth}
		\small
		\begin{equation*}
		\begin{split}
		w_{1,5} + w_{4,5} + w_{3,4} & \leq w_{1,2} + w_{2,3}\\
		w_{1,2} + w_{1,5} + w_{4,5} & \leq w_{2,3} + w_{3,4}\\
		w_{2,3} + w_{3,7}  & \leq w_{2,6} + w_{6,7}\\
		w_{2,3} + w_{2,6}  & \leq w_{3,7} + w_{6,7}\\
		w_{5,6} + w_{6,7} & \leq w_{4,5} + w_{4,7}\\
		w_{4,7} + w_{6,7}  & \leq w_{4,5} + w_{5,6}
		\end{split}
		\end{equation*}
	\end{minipage}
	$\implies $
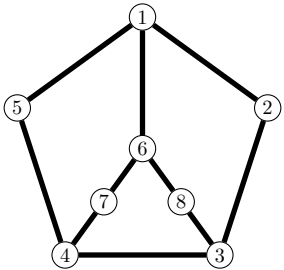
	\begin{minipage}{0.2\textwidth}
		\small
		$$w_{1,5} \leq 0$$
	\end{minipage}
	\begin{center}
		\line(1,0){400}
	\end{center}
	\begin{minipage}[c]{0.4\textwidth}
		\centering
		\begin{tikzpicture}[scale=0.35, every node/.style={scale=0.35}]

		\node[draw,circle,minimum size=.5cm,inner sep=1pt] (1) at (0*360/5 +90: 5cm) [scale =2]{$1$};
		\node[draw,circle,minimum size=.5cm,inner sep=1pt] (2) at (-1*360/5 +90: 5cm) [scale =2]{$2$};
		\node[draw,circle,minimum size=.5cm,inner sep=1pt] (3) at (-2*360/5 +90: 5cm) [scale =2]{$3$};
		\node[draw,circle,minimum size=.5cm,inner sep=1pt] (4) at (-3*360/5 +90: 5cm)[scale =2] {$4$};
		\node[draw,circle,minimum size=.5cm,inner sep=1pt] (5) at (-4*360/5 +90: 5cm) [scale =2]{$5$};
		\node[draw,circle,minimum size=.5cm,inner sep=1pt] (6) at (0,0) [scale =2]{$6$};
		\node[draw,circle,minimum size=.5cm,inner sep=1pt] (7) at ($(4)!.5!(6)$) [scale =2]{$7$};
		\node[draw,circle,minimum size=.5cm,inner sep=1pt] (8) at ($(3)!.5!(6)$) [scale =2]{$8$};

		\draw [line width=2pt,-] (1) -- (2);
		\draw [line width=2pt,-] (1) -- (5);
		\draw [line width=2pt,-] (1) -- (6);
		\draw [line width=2pt,-] (2) -- (3);
		\draw [line width=2pt,-] (3) -- (4);
		\draw [line width=2pt,-] (3) -- (8);
		\draw [line width=2pt,-] (4) -- (5);
		\draw [line width=2pt,-] (4) -- (7);
		\draw [line width=2pt,-] (6) -- (7);
		\draw [line width=2pt,-] (6) -- (8);

		\end{tikzpicture}
	\end{minipage}
	\begin{minipage}[c]{0.4\textwidth}
		\centering
		\small
		\begin{equation*}
		\begin{gathered}
		(12), \ (1543), \ (154), \ (15), \ (16), \ (167), \ (168), \ (23), \\
		(234), \ (2345), \ (216), \ (2167), \ (2168), \ (34), \ (345), \\ 
		(386), \ (347), \ (38), \ (45), \ (4516), \ (47), \ (438), \\
		(516), \ (547), \ (5168), \ (67), \ (68), \ (7438)
		\end{gathered}
		\end{equation*}
	\end{minipage}
	\\
	\vspace{5mm}
	\begin{minipage}{0.45\textwidth}
		\small
		\begin{equation*}
		\begin{split}
		w_{4,7} + w_{3,4} + w_{3,8} & \leq  w_{6,7} +w_{6,8}\\
		w_{2,3} + w_{3,4} + w_{4,5} & \leq  w_{1,2} +w_{1,5}\\
		w_{1,5} + w_{1,6} + w_{6,8} & \leq  w_{4,5} +w_{3,4}+w_{3,8}\\
		w_{1,2} + w_{1,6} + w_{6,7} & \leq  w_{2,3} +w_{3,4}+w_{4,7}
		\end{split}
		\end{equation*}
	\end{minipage}
	$\implies $
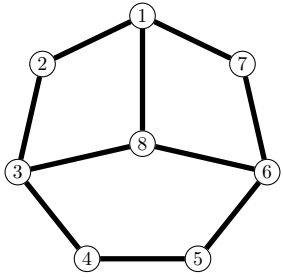
	\begin{minipage}{0.2\textwidth}
		\small
		$$w_{1,6} \leq 0$$
	\end{minipage}
	\begin{center}
		\line(1,0){400}
	\end{center}
	\begin{minipage}[c]{0.4\textwidth}
		\centering
		\begin{tikzpicture}[scale=0.34, every node/.style={scale=0.34}]

		\node[draw,circle,minimum size=.5cm,inner sep=1pt] (1) at (0*360/7 +90: 5cm)[scale =2] {$1$};
		\node[draw,circle,minimum size=.5cm,inner sep=1pt] (2) at (1*360/7 +90: 5cm)[scale =2] {$2$};
		\node[draw,circle,minimum size=.5cm,inner sep=1pt] (3) at (2*360/7 +90: 5cm) [scale =2]{$3$};
		\node[draw,circle,minimum size=.5cm,inner sep=1pt] (4) at (3*360/7 +90: 5cm) [scale =2]{$4$};
		\node[draw,circle,minimum size=.5cm,inner sep=1pt] (5) at (4*360/7 +90: 5cm) [scale =2]{$5$};
		\node[draw,circle,minimum size=.5cm,inner sep=1pt] (6) at (5*360/7 +90: 5cm) [scale =2]{$6$};
		\node[draw,circle,minimum size=.5cm,inner sep=1pt] (7) at (6*360/7 +90: 5cm) [scale =2]{$7$};
		
		\node[draw,circle,minimum size=.5cm,inner sep=1pt] (8) at (0,0) [scale =2]{$8$};

		\draw [line width=2pt,-] (1) -- (2);
		\draw [line width=2pt,-] (2) -- (3);
		\draw [line width=2pt,-] (3) -- (4);
		\draw [line width=2pt,-] (4) -- (5);
		\draw [line width=2pt,-] (5) -- (6);
		\draw [line width=2pt,-] (6) -- (7);
		\draw [line width=2pt,-] (7) -- (1);
		\draw [line width=2pt,-] (1) -- (8);
		\draw [line width=2pt,-] (3) -- (8);
		\draw [line width=2pt,-] (6) -- (8);

		\end{tikzpicture}
	\end{minipage}
	\begin{minipage}[c]{0.4\textwidth}
		\centering
		\small
		\begin{equation*}
		\begin{gathered}
		(12), \ (123), \ (17654), \ (1765), \ (176),\ (17), \ (18), \\ 
		(23), (234), (2345), \ (23456), \ (217), \ (218), \ (34), \\
		(345), \ (3456), \ (3217), \ (38), \ (45), \ (456), \ (4567), \\
		(438), \  (56).\ (567), \ (5438), \ (67), \ (68), \ (768)
		\end{gathered}
		\end{equation*}
	\end{minipage}
	\\
	\vspace{5mm}
	\begin{minipage}{0.45\textwidth}
		\small
		\begin{equation*}
		\begin{split}
		w_{4,5} + w_{3,4}+w_{3,8} & \leq w_{5,6} + w_{6,8}\\
		w_{1,2} + w_{1,8} & \leq w_{2,3} + w_{3,8}\\
		w_{6,7} + w_{6,8} & \leq w_{1,7} + w_{1,8}\\
		w_{1,7}+w_{6,7} + w_{5,6} + w_{4,5} & \leq w_{1,2} + w_{2,3} + w_{3,4}\\
		w_{2,3}+w_{3,4} + w_{4,5} + w_{5,6} & \leq w_{1,2} + w_{1,7} + w_{6,7}\\
		w_{2,3} + w_{1,2} + w_{1,7} & \leq w_{3,4} + w_{4,5} + w_{5,6} + w_{6,7}
		\end{split}
		\end{equation*}
	\end{minipage}
	$\implies $
	\begin{minipage}{0.2\textwidth}
		\small
		$$w_{4,5} \leq 0$$
	\end{minipage}
	\begin{center}
		\line(1,0){400}
	\end{center}
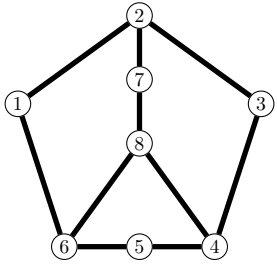
	\begin{minipage}[c]{0.4\textwidth}
		\centering
		\begin{tikzpicture}[scale=0.34, every node/.style={scale=0.34}]
		
		\node[draw,circle,minimum size=.5cm,inner sep=1pt] (2) at (0*360/5 +90: 5cm) [scale =2]{$2$};
		\node[draw,circle,minimum size=.5cm,inner sep=1pt] (1) at (1*360/5 +90: 5cm) [scale =2]{$1$};
		\node[draw,circle,minimum size=.5cm,inner sep=1pt] (6) at (2*360/5 +90: 5cm) [scale =2]{$6$};
		\node[draw,circle,minimum size=.5cm,inner sep=1pt] (4) at (3*360/5 +90: 5cm)[scale =2] {$4$};
		\node[draw,circle,minimum size=.5cm,inner sep=1pt] (3) at (4*360/5 +90: 5cm)[scale =2] {$3$};
		\node[draw,circle,minimum size=.5cm,inner sep=1pt] (8) at (0,0)[scale =2] {$8$};
		\node[draw,circle,minimum size=.5cm,inner sep=1pt] (5) at ($(4)!.5!(6)$)[scale =2] {$5$};
		\node[draw,circle,minimum size=.5cm,inner sep=1pt] (7) at ($(2)!.5!(8)$)[scale =2] {$7$};

		\draw [line width=2pt,-] (1) -- (2);
		\draw [line width=2pt,-] (1) -- (6);
		\draw [line width=2pt,-] (2) -- (3);
		\draw [line width=2pt,-] (2) -- (7);
		\draw [line width=2pt,-] (3) -- (4);
		\draw [line width=2pt,-] (4) -- (5);
		\draw [line width=2pt,-] (4) -- (8);
		\draw [line width=2pt,-] (5) -- (6);
		\draw [line width=2pt,-] (6) -- (8);
		\draw [line width=2pt,-] (7) -- (8);

		\end{tikzpicture}
	\end{minipage}
	\begin{minipage}[c]{0.4\textwidth}
		\centering
		\small
		\begin{equation*}
		\begin{gathered}
		(12), \ (123), \ (1234), \ (165), \ (16), \ (1687), \ (168),\\
		(23), \ (234), \ (2165), \ (216),\ (27),\ (278), \ (34), \\ 
		(345), \ (3456), \ (327), \ (3278), \ (45), \ (456), \ (487), \\ 
		(48),\ (56), \ (5487), \ (548), \ (67), \ (68),\ (78)
		\end{gathered}
		\end{equation*}
	\end{minipage}
	\\
	\vspace{5mm}
	\begin{minipage}{0.45\textwidth}
		\small
		\begin{equation*}
		\begin{split}
		w_{1,6} + w_{6,8} + w_{7,8}  & \leq w_{1,2} + w_{2,7}\\
		w_{2,3} + w_{2,7} + w_{7,8}  & \leq w_{3,4} + w_{4,8}\\
		w_{4,5} + w_{4,8}  & \leq w_{5,6} + w_{6,8}\\
		w_{1,2} + w_{2,3} + w_{3,4}  & \leq w_{1,6} + w_{5,6} + w_{4,5}\\
		w_{1,2} + w_{1,6} + w_{5,6}  & \leq w_{2,3} + w_{3,4} + w_{4,5}\\
		w_{3,4} + w_{4,5} + w_{5,6}  & \leq w_{2,3} + w_{1,2} + w_{1,6}
		\end{split}
		\end{equation*}
	\end{minipage}
	$\implies $
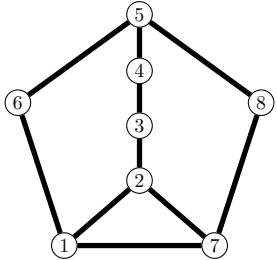
	\begin{minipage}{0.2\textwidth}
		\small
		$$w_{7,8} \leq 0$$
	\end{minipage}
	\begin{center}
		\line(1,0){400}
	\end{center}
	\begin{minipage}[c]{0.4\textwidth}
		\centering
		\begin{tikzpicture}[scale=0.34, every node/.style={scale=0.34}]

		\node[draw,circle,minimum size=.5cm,inner sep=1pt] (5) at (0*360/5 +90: 5cm) [scale =2] {$5$};
		\node[draw,circle,minimum size=.5cm,inner sep=1pt] (6) at (1*360/5 +90: 5cm) [scale =2] {$6$};
		\node[draw,circle,minimum size=.5cm,inner sep=1pt] (1) at (2*360/5 +90: 5cm) [scale =2] {$1$};
		\node[draw,circle,minimum size=.5cm,inner sep=1pt] (7) at (3*360/5 +90: 5cm)  [scale =2]{$7$};
		\node[draw,circle,minimum size=.5cm,inner sep=1pt] (8) at (4*360/5 +90: 5cm)  [scale =2]{$8$};
		\node[draw,circle,minimum size=.5cm,inner sep=1pt] (2) at (0,-1.5)  [scale =2]{$2$};
		\node[draw,circle,minimum size=.5cm,inner sep=1pt] (3) at ($(2)!.33!(5)$)  [scale =2]{$3$};
		\node[draw,circle,minimum size=.5cm,inner sep=1pt] (4) at ($(2)!.66!(5)$) [scale =2] {$4$};

		\draw [line width=2pt,-] (1) -- (2);
		\draw [line width=2pt,-] (1) -- (6);
		\draw [line width=2pt,-] (1) -- (7);
		\draw [line width=2pt,-] (2) -- (3);
		\draw [line width=2pt,-] (2) -- (7);
		\draw [line width=2pt,-] (3) -- (4);
		\draw [line width=2pt,-] (4) -- (5);
		\draw [line width=2pt,-] (5) -- (6);
		\draw [line width=2pt,-] (5) -- (8);
		\draw [line width=2pt,-] (7) -- (8);

		\end{tikzpicture}
	\end{minipage}
	\begin{minipage}[c]{0.4\textwidth}
		\centering
		\small
		\begin{equation*}
		\begin{gathered}
		(12), \ (16543), \ (1654), \ (165), \ (16), \ (17), \ (178), \\ 
		(23), \ (234), \ (2345), \ (216), \ (27), \ (23458), \ (34), \\
		(345), \ (3456), \ (327), \ (3458), \ (45), \ (456), \ (4327), \\
		(458), \ (56), \ (587), \ (58), \ (6587), \ (658), \  (78)
		\end{gathered}
		\end{equation*}
	\end{minipage}
	\\
	\vspace{4mm}
	\begin{minipage}{0.45\textwidth}
		\small
		\begin{equation*}
		\begin{split}
		w_{1,6} + w_{5,6} + w_{4,5} + w_{3,4}  & \leq w_{1,2} + w_{2,3}\\
		w_{2,3} + w_{3,4} + w_{4,5} + w_{5,8} & \leq w_{2,7} +w_{7,8}\\
		w_{5,6} + w_{5,8} + w_{7,8} & \leq w_{1,6} +w_{1,7}\\
		w_{3,4}+ w_{2,3} + w_{2,7}  & \leq w_{4,5} +w_{5,8}+ w_{7,8} \\
		w_{1,7} + w_{7,8} &\leq  w_{1,6} + w_{5,6} + w_{5,8}\\
		w_{1,2} + w_{1,6} &\leq  w_{2,3} + w_{3,4}+w_{4,5} +w_{5,6}
		\end{split}
		\end{equation*}
	\end{minipage}
	$\implies $
	\begin{minipage}{0.2\textwidth}
		\small
		$$w_{3,4} \leq 0$$
	\end{minipage}
	\begin{center}
		\line(1,0){400}
	\end{center}
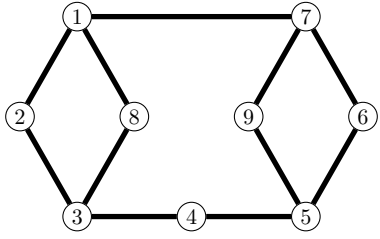
	\begin{minipage}[c]{0.4\textwidth}
		\centering
		\begin{tikzpicture}[scale=0.38, every node/.style={scale=0.38}]

		\node[draw,circle,minimum size=.5cm,inner sep=1pt] (1) at (-4,3.5) [scale=2]{$1$};
		\node[draw,circle,minimum size=.5cm,inner sep=1pt] (2) at (-6,0) [scale=2]{$2$};
		\node[draw,circle,minimum size=.5cm,inner sep=1pt] (3) at (-4,-3.5) [scale=2]{$3$};
		\node[draw,circle,minimum size=.5cm,inner sep=1pt] (5) at (4,-3.5) [scale=2]{$5$};
		\node[draw,circle,minimum size=.5cm,inner sep=1pt] (6) at (6,0) [scale=2]{$6$};
		\node[draw,circle,minimum size=.5cm,inner sep=1pt] (7) at (4,3.5) [scale=2]{$7$};
		\node[draw,circle,minimum size=.5cm,inner sep=1pt] (4) at ($(3)!.5!(5)$) [scale=2]{$4$};
		\node[draw,circle,minimum size=.5cm,inner sep=1pt] (8) at (-2,0) [scale=2]{$8$};
		\node[draw,circle,minimum size=.5cm,inner sep=1pt] (9) at (2,0) [scale=2]{$9$};

		\draw [line width=2pt,-] (1) -- (2);
		\draw [line width=2pt,-] (1) -- (7);
		\draw [line width=2pt,-] (1) -- (8);
		\draw [line width=2pt,-] (2) -- (3);
		\draw [line width=2pt,-] (3) -- (4);
		\draw [line width=2pt,-] (3) -- (8);
		\draw [line width=2pt,-] (4) -- (5);
		\draw [line width=2pt,-] (5) -- (6);
		\draw [line width=2pt,-] (5) -- (9);
		\draw [line width=2pt,-] (6) -- (7);
		\draw [line width=2pt,-] (7) -- (9);

		\end{tikzpicture}
	\end{minipage}
	\begin{minipage}[c]{0.4\textwidth}
		\centering
		\small
		\begin{equation*}
		\begin{gathered}
		(12) , \ (183), \ (1834) , \ (1765), \ (176), \ (17), \ (18), \ (179), \\
		(23), \ (234), \ (2345), \ (23456), \ (217), \ (218), \  (2179), \\
		(34), \ (345), \ (3456), \ (3817), \ (38), \ (3459), \ (45), \\
		(456), \ (43817), \ (438), \ (459), \ (56), \ (567), \ (5438), \\ 
		(59), \ (67), \ (6718), \ (679), \ (718), \ (79), \ (83459)
		\end{gathered}
		\end{equation*}
	\end{minipage}
	\\
	\vspace{5mm}
	\begin{minipage}{0.45\textwidth}
		\small
		\begin{equation*}
		\begin{split}
		w_{2,3} + w_{3,4} + w_{4,5} + w_{5,6} & \leq w_{1,2} + w_{1,7} + w_{6,7}\\
		w_{3,4} + w_{3,8} + w_{1,8} + w_{1,7} & \leq w_{4,5} + w_{5,6} + w_{6,7}\\
		w_{3,8} + w_{3,4} + w_{4,5} + w_{5,9} & \leq w_{1,8} + w_{1,7} + w_{7,9}\\
		w_{1,7} + w_{6,7} + w_{5,6} & \leq w_{1,8} + w_{3,8} + w_{3,4} + w_{4,5}\\
		w_{1,2} + w_{1,8} & \leq w_{2,3} + w_{3,8} \\
		w_{6,7} + w_{7,9} & \leq w_{5,6} + w_{5,9}
		\end{split}
		\end{equation*}
	\end{minipage}
	$\implies $
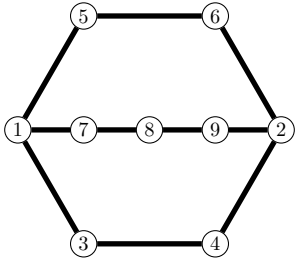
	\begin{minipage}{0.2\textwidth}
		\small
		$$w_{3,4} \leq 0$$
	\end{minipage}
	\begin{center}
		\line(1,0){400}
	\end{center}
	\begin{minipage}[c]{0.4\textwidth}
		\centering
		\begin{tikzpicture}[scale=0.35, every node/.style={scale=0.35}]

		\node[draw,circle,minimum size=.5cm,inner sep=1pt] (1) at (0*360/6 +180: 5cm) [scale =2] {$1$};
		\node[draw,circle,minimum size=.5cm,inner sep=1pt] (2) at (3*360/6 +180: 5cm) [scale =2] {$2$};
		\node[draw,circle,minimum size=.5cm,inner sep=1pt] (3) at (1*360/6 +180: 5cm) [scale =2] {$3$};
		\node[draw,circle,minimum size=.5cm,inner sep=1pt] (4) at (2*360/6 +180: 5cm)  [scale =2]{$4$};
		\node[draw,circle,minimum size=.5cm,inner sep=1pt] (5) at (5*360/6 +180: 5cm)  [scale =2]{$5$};
		\node[draw,circle,minimum size=.5cm,inner sep=1pt] (6) at (4*360/6 +180: 5cm)  [scale =2]{$6$};
		\node[draw,circle,minimum size=.5cm,inner sep=1pt] (7) at ($(1)!.25!(2)$)  [scale =2]{$7$};
		\node[draw,circle,minimum size=.5cm,inner sep=1pt] (8) at ($(1)!.5!(2)$)  [scale =2]{$8$};
		\node[draw,circle,minimum size=.5cm,inner sep=1pt] (9) at ($(1)!.75!(2)$)  [scale =2]{$9$};

		\draw [line width=2pt,-] (1) -- (3);
		\draw [line width=2pt,-] (1) -- (5);
		\draw [line width=2pt,-] (1) -- (7);
		\draw [line width=2pt,-] (2) -- (4);
		\draw [line width=2pt,-] (2) -- (6);
		\draw [line width=2pt,-] (2) -- (9);
		\draw [line width=2pt,-] (3) -- (4);
		\draw [line width=2pt,-] (5) -- (6);
		\draw [line width=2pt,-] (7) -- (8);
		\draw [line width=2pt,-] (8) -- (9);

		\end{tikzpicture}
	\end{minipage}
	\begin{minipage}[c]{0.4\textwidth}
		\centering
		\small
		\begin{equation*}
		\begin{gathered}
		(17892), \ (13), \ (134), \ (15), \ (156), \ (17), \ (178), \ (1789), \\
		(243), \ (24), \ (265),\ (26), \ (2987), \ (298), \ (29), \ (34), \\
		(315), \ (3156), \ (317), \ (3178), \ (3429), \ (4265), \ (426), \\
		(4317), \ (43178), \ (429), \ (56),\ (517), \ (5178), \\
		(51789), \ (62987), \ (6298), \ (629), \ (78), \ (789), \ (89)
		\end{gathered}
		\end{equation*}
	\end{minipage}
	\\
	\vspace{5mm}
	\begin{minipage}{0.45\textwidth}
		\small
		\begin{equation*}
		\begin{split}
		w_{2,6} + w_{2,9} + w_{8,9} + w_{7,8} & \leq w_{5,6} + w_{1,5} + w_{1,7}\\
		w_{3,4} + w_{1,3} + w_{1,7} + w_{7,8} & \leq w_{2,4} + w_{2,9} + w_{8,9}\\
		w_{1,5} + w_{1,7} + w_{7,8} + w_{8,9} & \leq w_{5,6} + w_{2,6} + w_{2,9}\\
		w_{2,4} + w_{2,6} + w_{5,6}  & \leq w_{3,4} + w_{1,3} + w_{1,5}\\
		w_{1,3} + w_{1,5} + w_{5,6}  & \leq w_{3,4} + w_{2,4} + w_{2,6}\\
		w_{3,4} + w_{2,4} + w_{2,9}  & \leq w_{1,3} + w_{1,7} + w_{7,8} + w_{8,9}
		\end{split}
		\end{equation*}
	\end{minipage}
	$\implies $
	\begin{minipage}{0.2\textwidth}
		\small
		$$w_{7,8} \leq 0$$
	\end{minipage}
	\begin{center}
		\line(1,0){400}
	\end{center}
	
	\newpage
	\printbibliography

\end{document}